\documentclass[11pt,english]{article}
\usepackage[utf8]{inputenc}
\usepackage{babel}
\usepackage{mathtools}
\usepackage{dsfont}
\usepackage{amsmath}
\usepackage{amsthm}
\usepackage{amssymb}
\usepackage{stmaryrd}
\usepackage[a4paper]{geometry}
\usepackage[bookmarks=true,bookmarksnumbered=false,bookmarksopen=false,
 breaklinks=false,pdfborder={0 0 1},backref=false,colorlinks=false]
 {hyperref}
\hypersetup{
 linkcolor=blue}

\makeatletter
\numberwithin{equation}{section}
\numberwithin{figure}{section}
\theoremstyle{plain}
\newtheorem{thm}{\protect\theoremname}[section]
\theoremstyle{remark}
\newtheorem{rem}[thm]{\protect\remarkname}
\theoremstyle{definition}
\newtheorem{problem}[thm]{\protect\problemname}
\theoremstyle{definition}
\newtheorem{defn}[thm]{\protect\definitionname}
\theoremstyle{plain}
\newtheorem{lem}[thm]{\protect\lemmaname}
\theoremstyle{plain}
\newtheorem{prop}[thm]{\protect\propositionname}
\newenvironment{lyxlist}[1]
	{\begin{list}{}
		{\settowidth{\labelwidth}{#1}
		 \setlength{\leftmargin}{\labelwidth}
		 \addtolength{\leftmargin}{\labelsep}
		 }}
	{\end{list}}
\theoremstyle{plain}
\newtheorem{cor}[thm]{\protect\corollaryname}

\usepackage{ dsfont }

\newcommand{\N}{\mathbb{N}}

\usepackage{tikz}
\usepackage{tikz-cd}
\usetikzlibrary{tqft}

\newcommand{\df}{\mathrm{d}}

\allowdisplaybreaks

\makeatother

\providecommand{\corollaryname}{Corollary}
\providecommand{\definitionname}{Definition}
\providecommand{\lemmaname}{Lemma}
\providecommand{\problemname}{Problem}
\providecommand{\propositionname}{Proposition}
\providecommand{\remarkname}{Remark}
\providecommand{\theoremname}{Theorem}

\begin{document}
\global\long\def\df{\mathrm{def}}%
\global\long\def\eqdf{\stackrel{\df}{=}}%
\global\long\def\ep{\varepsilon}%
\global\long\def\ind{\mathds{1}}%
\global\long\def\cl{\mathrm{cl}}%
\global\long\def\N{\mathbb{N}}%
\global\long\def\C{\mathbb{C}}%
\global\long\def\tr{\textup{tr}}%
\global\long\def\d{\mathbf{d}}%
\global\long\def\t{\mathbf{t}}%

\title{Spectral gap with polynomial rate for Weil-Petersson random surfaces}
\author{Will Hide, Davide Macera and Joe Thomas}
\maketitle
\begin{abstract}
We show that there is a constant $c>0$ such that a genus $g$ closed
hyperbolic surface, sampled at random from the moduli space $\mathcal{M}_{g}$
with respect to the Weil-Petersson probability measure, has Laplacian
spectral gap at least $\frac{1}{4}-O\left(\frac{1}{g^{c}}\right)$
with probability tending to $1$ as $g\to\infty$. This extends and
gives a new proof of a recent result of Anantharaman and Monk proved
in the series of works \cite{An.Mo2022,An.Mo2023,An.Mo.sg2024,AM-Mobius,An.Mo2025}.

Our approach adapts the polynomial method for the strong convergence
of random matrices, introduced by Chen, Garza-Vargas, Tropp and van
Handel \cite{Ch.Ga.Tr.va2024}, and its generalization to the strong
convergence of surface groups by Magee, Puder and van Handel \cite{Ma.Pu.vH2025},
to the Laplacian on Weil-Petersson random hyperbolic surfaces.

{\footnotesize\tableofcontents{}}{\footnotesize\par}

\newpage{}
\end{abstract}

\section{Introduction}

For a closed and connected hyperbolic surface $X$, the $L^{2}$-spectrum
of the Laplacian $\Delta_{X}$ consists of discrete eigenvalues 
\[
0=\lambda_{0}\left(X\right)<\lambda_{1}\left(X\right)\leqslant\dots\leqslant\lambda_{j}\left(X\right)\leqslant\dots,
\]
with $\lambda_{j}\left(X\right)\to\infty$ as $j\to\infty$. The spectral
gap $\lambda_{1}\left(X\right)$ captures important geometric and
dynamical information about the surface $X$. It governs the exponential
rate of mixing of the geodesic flow, gives error terms in prime geodesic
theorems and quantifies how highly connected the surface is. By a
result of Huber \cite{Huber}, for any sequence of closed surfaces
$\left\{ X_{i}\right\} $ with genera $g_{i}\to\infty$, $\limsup_{i\to\infty}\lambda_{1}\left(X_{i}\right)\leqslant\frac{1}{4}$
so that $\frac{1}{4}$ is the asymptotically optimal spectral gap
in large genus. In this paper, we study the size of the spectral gap
for random closed hyperbolic surfaces.

Let $\mathcal{M}_{g}$ denote the moduli space of genus $g$ hyperbolic
surfaces. The normalised volume form arising from the Weil-Petersson
metric gives a natural probability measure $\mathbb{P}_{g}$ on $\mathcal{M}_{g}$
($\S$\ref{subsec:Background}). The main theorem of this paper is
the following.
\begin{thm}
\label{thm:Main-Thm}There is a $c>0$ such that a Weil-Petersson
random hyperbolic surface $X\in\mathcal{M}_{g}$ satisfies 
\begin{equation}
\lambda_{1}\left(X\right)\geqslant\frac{1}{4}-O\left(\frac{1}{g^{c}}\right)\label{eq:main-theorem}
\end{equation}
with probability tending to $1$ as $g\to\infty$.
\end{thm}

For the random covering model (c.f. $\S$\ref{subsec:Related-work}),
the analogue of Theorem \ref{thm:Main-Thm} was proven by the authors
in \cite{Three-mates}, relying on the recent breakthrough work of
Magee, Puder and van Handel \cite{Ma.Pu.vH2025}. Theorem \ref{thm:Main-Thm}
with $o(1)$ error instead of $O\left(\frac{1}{g^{c}}\right)$ was
proven by Anantharaman and Monk in a series of papers \cite{An.Mo2022,An.Mo2023,An.Mo.sg2024,AM-Mobius,An.Mo2025};
our work improves upon and gives a new proof of their result. The
significance of the polynomial error rate $O\left(\frac{1}{g^{c}}\right)$
is explained in the next section.

The overall approach of the current article is inspired heavily by
\cite{Ma.Pu.vH2025}, which builds upon the polynomial method \cite{Ch.Ga.Tr.va2024,Magee-delaSalle,Ch.Ga.Ha2024}
and shows how spectral gap information can be deduced from Gevrey
statistics. Adapting a similar strategy to approach the Weil-Petersson
model brings about several significant challenges that require new
technical innovations to overcome that we believe will be of use for
results beyond Theorem \ref{thm:Main-Thm}. We explain this strategy
and the difficulties in detail in $\S$ \ref{subsec:Overview-of-the-proof}.
\begin{rem}
The constant $c$ in Theorem \ref{thm:Main-Thm} can in principle
be made explicit but we do not pursue this here. 
\end{rem}

\subsubsection*{Motivation}

Selberg's Conjecture \cite{Se1965} predicts that $\lambda_{1}\left(X\left(N\right)\right)\geqslant\frac{1}{4}$
for every $N\geqslant1$ where $X\left(N\right)\eqdf\Gamma\left(N\right)\backslash\mathbb{H}$
and $\Gamma\left(N\right)$ denotes the principal congruence subgroup
of $\textup{SL}_{2}\left(\mathbb{Z}\right)$ of level $N$. One might
call surfaces $X$ with $\lambda_{1}\left(X\right)\geqslant\frac{1}{4}$
$\textit{Ramanujan surfaces}$ in analogy with Ramanujan graphs or
$\textit{Selberg surfaces}$ in light of Selberg's conjecture. 

The fact that there exists a sequence of closed surfaces $\left\{ X_{i}\right\} _{i\in\mathbb{N}}$
with genera $g_{i}\to\infty$ and $\lambda\left(X_{i}\right)\to\frac{1}{4}$
was conjectured by Buser \cite{Buser2} and proven by Magee and the
first named author \cite{Hi.Ma2023}. A significant and well-known
open problem -- likely originating in some form since the conception
of Buser's conjecture and experiencing a resurgence in interest since
its recent resolution -- is whether there exist $\left\{ X_{i}\right\} _{i\in\mathbb{N}}$
with $g_{i}\to\infty$ and $\lambda_{1}\left(X_{i}\right)\geqslant\frac{1}{4}$. 
\begin{problem}
\label{prob:Selberg-surfaces}Do there exist Selberg surfaces of arbitrarily
large genus?
\end{problem}

The existence of infinite families of Ramanujan graphs was proven
by Lubotzky, Phillips and Sarnak \cite{LPS} and independently by
Margulis \cite{M88} in seminal works. In another remarkable work,
Marcus, Spielman and Srivastava \cite{MSS15} proved the existence
of infinite families of bipartite Ramanujan graphs of every degree.

Recently, Huang, Mckenzie and Yau \cite{Hu.Mc.Ya2024} studied the
distribution of $\left(\lambda_{1}\left(\mathcal{G}_{n}\right)-2\sqrt{d-1}\right)n^{\frac{2}{3}}$
for uniformly random $d$-regular graphs $\mathcal{G}_{n}$ on $n$
vertices, showing that, after re-scaling by a constant $C(d)$, it
converges to the Tracy-Widom distribution with $\beta=1$ as $n\to\infty$.
As a consequence, they conclude that approximately $69\%$ of $d$-regular
graph are Ramanujan. 

It is expected from the Bohigas, Giannoni and Schmidt conjecture \cite{Bo.Gi.Sc1984}
that due to the time-reversal symmetry and chaotic nature of the geodesic
flow on hyperbolic surfaces, the spectral statistics of the Laplacian
on `typical' hyperbolic surfaces should exhibit fluctuation properties
akin to the Gaussian Orthogonal Ensemble. For this random matrix ensemble,
the limiting behavior of the largest eigenvalue (with suitable normalization)
is given by the Tracy-Widom distribution with $\beta=1$ \cite{Tr.Wi1996}.
It is therefore natural to conjecture that, after possibly re-scaling
by a constant, $\left(\frac{1}{4}-\lambda_{1}\left(X\right)\right)\textup{\ensuremath{\left(\textup{Vol}\left(X\right)\right)}}^{\frac{2}{3}}$
converges to the Tracy-Widom distribution for suitable models of random
hyperbolic surfaces, which would provide a positive answer to Problem
\ref{prob:Selberg-surfaces}. 

A first step in proving this difficult conjecture would be to determine
the correct scale at which $\lambda_{1}$ fluctuates around $\frac{1}{4}$.
Theorem \ref{thm:Main-Thm} makes significant progress in this direction.
By the above discussion, it is natural to expect that the optimal
value of $c$ one could obtain in (\ref{eq:main-theorem}) is $\frac{2}{3}$.

\subsection{Related work\label{subsec:Related-work}}

\subsubsection*{Spectral gaps of random hyperbolic surfaces}
\begin{enumerate}
\item The Brooks-Makover model \cite{BrooksMakover}: take $2n$ ideal hyperbolic
triangles, glue them according to a random $3$-regular graph and
apply a compactification procedure.
\item The Random Covering model \cite{Ma.Na.Pu2022}: fix a base manifold
and consider degree-$n$ Riemannian covers uniformly at random. Here
one considers the relative spectral gap $\lambda_{1}^{\textup{new}}$,
ignoring the eigenvalues of the base manifold.
\item The Weil-Petersson model \cite{Gu.Pa.Yo11,MirzakhaniRandom}: obtained
by normalising the Weil-Petersson volume form on $\mathcal{M}_{g}$.
\end{enumerate}
Brooks and Makover \cite{BrooksMakover} were the first to study the
spectral gap of random hyperbolic surfaces, proving that a random
closed surface in their model has a spectral gap with high probability
(w.h.p.). Mirzakhani \cite{MirzakhaniRandom} obtained the first explicit
spectral gap result, proving that a Weil-Petersson random surface
has spectral gap at least $\frac{1}{4}\left(\frac{\log(2)}{2\pi+\log(2)}\right)^{2}\approx0.0024$
w.h.p..

For the random covering model, the first spectral gap result was achieved
by Magee, Naud and Puder \cite{Ma.Na.Pu2022} who proved that random
covers of closed hyperbolic surfaces have relative spectral gap at
least $\frac{3}{16}-\varepsilon$ w.h.p.. A gap of $\frac{3}{16}-\varepsilon$
was subsequently proved for the Weil-Petersson model by two independent
teams, Wu and Xue \cite{Wu.Xu2022} and Lipnowski and Wright \cite{Li.Wr2024}.

The first proof that random hyperbolic surfaces have near optimal
spectral gaps was obtained by Magee and the first named author \cite{Hi.Ma2023},
where it was shown that random covers of finite-area non-compact hyperbolic
surfaces have spectral gap at least $\frac{1}{4}-\ep$ w.h.p.. Using
a compactification procedure of Buser, Burger and Dodzuik \cite{BBD},
this led to a proof of Buser's conjecture \cite{Buser2}.

As a byproduct of a recent breakthrough on the strong convergence
of surface groups, Magee, Puder and van Handel \cite{Ma.Pu.vH2025}
obtained the relative spectral gap of $\frac{1}{4}-\varepsilon$ for
random covers of closed hyperbolic surfaces w.h.p.. Later work of
the present authors \cite{Three-mates} showed that $\varepsilon$
can be taken to be $O\left(\frac{1}{n^{c}}\right)$.

For near optimal spectral gaps in the Weil-Petersson model, Anantharaman
and Monk proved a spectral gap of $\frac{1}{4}-\varepsilon$ w.h.p.
in a series of papers \cite{An.Mo2022,An.Mo2023,An.Mo.sg2024,AM-Mobius,An.Mo2025}.
This followed an intermediate spectral gap result of $\frac{2}{9}-\varepsilon$
by the same authors \cite{An.Mo2023}.

We also highlight the recent work of Calderon, Magee and Naud \cite{CMN}
where is was proved that random covers of Schottky surfaces have near
(conjecturally) optimal spectral gap. Moreover, for closed negatively
curved surfaces \cite{Hi.Mo.Nau2025} and subsequently for geometrically
finite surfaces (possibly of infinite-area and possibly with cusps)
\cite{Mo25}, a near optimal spectral gap was obtained for random
covers, see also the recent work \cite{Ba.Mo.Po2025}.

\subsubsection*{Random graphs and strong convergence}

It was conjectured by Alon \cite{Alon} and proved by Friedman \cite{Friedman}
that a random regular graph $\mathcal{G}_{n}$ on $n$ vertices has
spectral gap $\lambda_{1},\left|\lambda_{n}\right|\leqslant2\sqrt{d-1}+\ep$
w.h.p.. In other words, a random $d$-regular graph is a near optimal
expander. Bordenave gave an alternative proof \cite{bordenave2015new}
of Friedman's Theorem, and showed that one can take $\ep=\text{const}\cdot\left(\frac{\log\log n}{\log n}\right)^{2}.$
Subsequently, it was shown by Huang and Yau \cite{HY21} that one
can take $\ep=O\left(n^{-c}\right)$ for some $c>0$ and then Mckenzie,
Huang and Yau proved the optimal result with $\ep=\frac{1}{n^{\frac{2}{3}-o(1)}}$
\cite{Hu.Mc.Ya2024}. 

Shortly after the appearance of the article \cite{Hu.Mc.Ya2024},
a breakthrough of Chen, Tropp, Garza-Vargas and van Handel \cite{Ch.Ga.Tr.va2024}
gave a remarkable new proof of Friedman's theorem with $\ep=O\left(\frac{1}{n^{\frac{1}{8}}}\right)$
(subsequently improved to $O\left(\frac{1}{n^{\frac{1}{6}}}\right)$
in \cite{Ch.Ga.Ha2024} by a refinement of their methods). It is this
method \cite{Ch.Ga.Tr.va2024}, coined the polynomial method, and
its subsequent development in \cite{Ch.Ga.Ha2024,Magee-delaSalle,Ma.Pu.vH2025},
that the methods of the current article are based on. 

It was conjectured by Friedman \cite{FriedmanRelative} that for any
fixed finite graph $\mathcal{G}$ and for any $\ep>0$, a uniformly
random degree $n$ cover $\mathcal{G}_{n}$ has no new eigenvalues
with absolute value above $\rho\left(\tilde{\mathcal{G}}\right)+\ep$
with probability tending to $1$ as $n\to\infty$. Here $\rho\left(\tilde{\mathcal{G}}\right)$
is the spectral radius of the adjacency operator on $\ell^{2}\left(\tilde{\mathcal{G}}\right)$,
where $\tilde{\mathcal{G}}$ is the universal cover of $\mathcal{G}$.
Friedman's Conjecture was proved in a breakthrough of Bordenave and
Collins \cite{BordenaveCollins}. In fact they proved a stronger statement,
which we now explain.

For a discrete group $\Gamma$, a sequence of finite-dimensional unitary
representations $\left\{ \left(\rho_{i},V_{i}\right)\right\} $ are
said to $\textit{strongly converge}$ to the left regular representation
$\left(\lambda,\ell^{2}\left(\Gamma\right)\right)$ if for any $z\in\mathbb{C}\left[\Gamma\right]$,

\begin{equation}
\|\rho_{i}\left(z\right)\|_{\textup{Op}:V_{i}\to V_{i}}\to\|\lambda\left(z\right)\|_{\textup{Op}:\ell^{2}\left(\Gamma\right)\to\ell^{2}\left(\Gamma\right)},\label{eq:strong-conv}
\end{equation}
as $i\to\infty$. We refer the reader to the recent articles of Magee
\cite{Magee-Survey} and van Handel \cite{vH2025} for more complete
historical accounts and surveys on this property. When $\Gamma=\mathbb{F}$
is a finitely generated free group, the existence of a sequence $\left\{ \left(\rho_{i},V_{i}\right)\right\} _{i\in\mathbb{N}}$
satisfying (\ref{eq:strong-conv}) was proved by Haagerup and Thorbjørnsen
\cite{HaagerupThr}. Sampling $\phi\in\textup{Hom}\left(\mathbb{F},S_{n}\right)$
uniformly at random, Bordenave and Collins \cite{BordenaveCollins}
prove that the representations $\left(\textup{Std}_{n-1}\circ\phi_{i},\ell_{0}^{2}\left(\left\{ 1,\dots,n\right\} \right)\right)$
strongly converge to $\left(\lambda,\ell^{2}\left(\mathbb{F}\right)\right)$
in probability where $\left(\textup{Std}_{n-1},\ell_{0}^{2}\left(\left\{ 1,\dots,n\right\} \right)\right)$
is the standard $n-1$ dimensional irreducible representation of $S_{n}$.
Bordenave and Collins also prove strong quantitative forms of this
statement in \cite{Bo.Co2023}. Chen, Tropp, Garza-Vargas and van
Handel \cite{Ch.Ga.Ha2024} introduced a new and robust approach to
proving quantitative forms of (\ref{eq:strong-conv}), which as a
special case implies Friedman's Theorem with $\ep=O\left(\frac{1}{n^{\frac{1}{8}}}\right)$.

Property (\ref{eq:strong-conv}) is particularly relevant for producing
spectral gaps of covering manifolds. It is shown by Magee and the
first named author in \cite{Hi.Ma2023,Lo.Ma2023} that for a finite-area
hyperbolic surface $\Gamma\backslash\mathbb{\mathbb{H}}$, property
(\ref{eq:strong-conv}) for $\left(\textup{Std}_{n-1}\circ\phi_{i},\ell_{0}^{2}\left(\left\{ 1,\dots,n\right\} \right)\right)$
where $\phi_{i}\in\textup{Hom}\left(\Gamma,S_{n}\right)$ implies
that the relative spectral gap $\lambda_{1}^{\textup{new}}\left(X_{i}\right)\to\frac{1}{4}$
where $X_{i}=\textup{Stab}_{\phi_{i}}\left(1\right)\backslash\mathbb{H}$.
Therefore proving near optimal spectral gaps in the random covering
model is reduced to establishing (\ref{eq:strong-conv}) for uniformly
random $\phi\in\textup{Hom}\left(\Gamma,S_{n}\right)$.

When the base manifold is a finite-area non-compact hyperbolic surface,
the work of Bordenave and Collins \cite{BordenaveCollins} can be
applied, as in \cite{Hi.Ma2023}. When $\Gamma\cong\Sigma_{g}$ is
a surface group of genus $g$, the situation is more difficult. In
this case, the existence of strongly convergent representations factoring
through $S_{n}$ was proven by Magee and Louder \cite{Lo.Ma2023}.
Recently the impressive work of Magee, Puder and van Handel \cite{Ma.Pu.vH2025},
proved that $\left(\textup{Std}_{n-1}\circ\phi_{i},\ell_{0}^{2}\left(\left\{ 1,\dots,n\right\} \right)\right)$
strongly converge to $\left(\lambda,\ell^{2}\left(\Sigma_{g}\right)\right)$
in probability when $\phi_{i}\in\textup{Hom}\left(\Sigma_{g},S_{n}\right)$
is sampled uniformly at random, extending the polynomial method \cite{Ch.Ga.Tr.va2024}
to this setting. For $3$-manifolds and right-angled Artin groups,
see \cite{Ma.Th23}.

Finally we mention that the polynomial method has been applied and
further developed by Magee and de la Salle to prove strong convergence
for unitary representations of quasi-exponential dimension \cite{Magee-delaSalle}.
Along this direction, a particularly striking result of Cassidy \cite{Cassidy}
shows that for fixed $r$, the random Schreier graphs corresponding
to the action of $r$ random permutations on $k_{n}$-tuples of distinct
elements of $\left\{ 1,\dots,n\right\} $ are near optimal expanders
w.h.p. for any $k\leqslant n^{\frac{1}{12}-o(1)}$. 

\subsection{Overview of the proof\label{subsec:Overview-of-the-proof}}

\subsubsection*{Adapting the polynomial method\label{subsec:overview-Analytic-part}}

We aim to adapt the strategy of \cite{Ma.Pu.vH2025} to our setting.
We fix a function $f$ (see $\S$\ref{sec:Geometric-estimates}) whose
inverse Fourier transform $\check{f}$ is compactly supported, even
and non-negative on the real line, and consider the operator $f\left(\sqrt{\Delta-\frac{1}{4}}\right)$
which will play the role of the self-adjoint random matrix in the
polynomial method. 

A first point to note is that $f\left(\sqrt{\Delta-\frac{1}{4}}\right)$
has a trivial eigenvalue $f\left(\frac{i}{2}\right)$, corresponding
to the zero eigenvalue of $\Delta$, which we want to discard. By
the choice of $f$, we have $t\mapsto f(it)$ is strictly increasing
and $f(0)>0$. So, if we let $h$ be a smooth non-negative function
which is equal to $1$ in $\left[f\left(\sqrt{\ep}\right),f\left(\sqrt{\frac{1}{4}-\delta}\right)\right]$,
then
\[
\mathbb{P}_{g}\left[\delta<\lambda_{1}<\frac{1}{4}-\ep\right]\leqslant\mathbb{E}_{g}\left[\text{Tr}\left(h\left(f\left(\sqrt{\Delta-\tfrac{1}{4}}\right)\right)\right)\right].
\]
By, for example, Mirzakhani's spectral gap result \cite{MirzakhaniRandom},
we know that there is some $\delta>0$ such that $\mathbb{P}_{g}\left[\lambda_{1}<\delta\right]\to0$,
so our goal is to show that the right hand side goes to $0$ as $g\to\infty$.

We want to apply the generalised polynomial method of \cite{Ma.Pu.vH2025}
to be able to bound 
\begin{equation}
\mathbb{E}_{g}\left[\text{Tr}\left(h\left(f\left(\sqrt{\Delta-\tfrac{1}{4}}\right)\right)\right)\right]\label{eq:exp-trace-smooth}
\end{equation}
for $h\in C^{\infty}\left(\mathbb{R}\right)$. Let $\mathrm{tr}\eqdf\frac{1}{g}\mathrm{Tr}$
be a normalisation of the usual trace of a trace-class operator. The
key input is the following result, analogous to Assumption 1.3 in
\cite{Ma.Pu.vH2025}. 
\begin{thm}
\label{thm:main-technical-result}There exists a constant $c>0$ such
that for each $t\in\mathbb{N}$, there are constants $\left\{ a_{i}^{t}\right\} _{i\geq0}$
so that
\[
\left|\mathbb{E}_{g}\left[\textup{tr}\left(\left(f\left(\sqrt{\Delta-\tfrac{1}{4}}\right)\right)^{t}\right)\right]-\sum_{i=0}^{q-1}\frac{a_{i}^{t}}{g^{i}}\right|\leqslant\frac{\left(cq\right)^{cq}}{g^{q}},
\]
for all $q>t$ and $g>cq^{c}$. We have that
\[
a_{0}^{t}=\int_{\frac{1}{4}}^{\infty}f\left(\sqrt{r-\tfrac{1}{4}}\right)^{t}\tanh\left(\pi\sqrt{r-\tfrac{1}{4}}\right)\mathrm{d}r,
\]
 and 
\[
a_{1}^{t}=\int_{0}^{\infty}\sum_{k=1}^{\infty}2\frac{\sinh\left(\frac{\ell}{2}\right)^{2}}{\sinh\left(\frac{k\ell}{2}\right)}\check{f}^{*t}(k\ell)\mathrm{d}\ell-\int_{\frac{1}{4}}^{\infty}f\left(\sqrt{r-\tfrac{1}{4}}\right)^{t}\tanh\left(\pi\sqrt{r-\tfrac{1}{4}}\right)\mathrm{d}r.
\]
\end{thm}

Establishing Theorem \ref{thm:main-technical-result} occupies the
bulk of the paper. It requires establishing new estimates for ratios
of Weil-Petersson volumes, building on work of Mirzakhani and Zograf
\cite{Mi.Zo2015}. This problem is discussed in more detail in the
next subsection.

Once we have Theorem \ref{thm:main-technical-result}, there are several
points of departure from the works \cite{Ch.Ga.Tr.va2024,Ma.Pu.vH2025}.
The idea is to use Theorem \ref{thm:main-technical-result} to extend
bounds on (\ref{eq:exp-trace-smooth}) from polynomials to smooth
functions $h$. One issue is that assumptions on $h$ are required
to even ensure that $h\left(f\left(\sqrt{\Delta-\frac{1}{4}}\right)\right)$
is a trace class operator, i.e. one should not expect to be able to
extend to all $h\in C^{\infty}\left(\mathbb{R}\right)$. Instead,
we consider functions of the form $h(x)\eqdf x\tilde{h}(x)$ and extend
bounds on (\ref{eq:exp-trace-smooth}) from polynomial $\tilde{h}$
to all $\tilde{h}\in C^{\infty}\left(\mathbb{R}\right)$. This ensures
that $h\left(f\left(\sqrt{\Delta-\frac{1}{4}}\right)\right)$ is trace
class and also allows us to establish an a priori bound for $\mathbb{E}\left[\text{Tr}\left(h\left(f\left(\sqrt{\Delta-\frac{1}{4}}\right)\right)\right)\right]$
(c.f. (\ref{eq:apriori-bound})), which essentially follows from weak
convergence. We then show that there is a compactly supported distribution
$\nu_{1}$ so that the following holds.
\begin{thm}
There exist $C,m>0$ such that for any $\tilde{h}\in C^{\infty}\left(\mathbb{R}\right)$
and any $g\geqslant2$, 
\begin{align*}
 & \left|\mathbb{E}_{g}\left[\tr\,h\left(f\left(\sqrt{\Delta-\tfrac{1}{4}}\right)\right)\right]-\int_{\frac{1}{4}}^{\infty}h\left(f\left(\sqrt{r-\tfrac{1}{4}}\right)\right)\tanh\left(\pi\sqrt{r-\tfrac{1}{4}}\right)\mathrm{d}r-\frac{\nu_{1}\left(\tilde{h}\right)}{g}\right|\\
 & \hspace*{\fill}\leqslant\frac{C}{g^{2}}\left(\|w^{(m)}\|_{[0,2\pi]}+\|\tilde{h}\|_{\left[-f\left(\frac{i}{2}\right),f\left(\frac{i}{2}\right)\right]}\right),
\end{align*}
where $w(\theta)\eqdf\tilde{h}\left(f\left(\frac{i}{2}\right)\cos\theta\right)$.
\end{thm}

Here and throughout, $\|\xi\|_{A}$ denotes the sup norm of a function
$\xi$ on the set $A$. We then want to show that $\nu_{1}(\tilde{h})=0$
for $h$ supported away from $\left[0,f(0)\right]\cup\left[f\left(i\sqrt{\frac{1}{4}-\delta}\right),f\left(\frac{i}{2}\right)\right]$,
which would allow us to conclude Theorem \ref{thm:Main-Thm} since
this is precisely where $f\left(\sqrt{\Delta-\tfrac{1}{4}}\right)$
maps the spectrum of the Laplacian in $[0,\delta]\cup[\frac{1}{4},\infty)$.
Remarkably, and in contrast to previous settings in which the polynomial
method has been applied, the distribution $\nu_{1}$ has a fairly
explicit description which allows us to compute its moments, and hence
its support (see \cite[Lemma 4.9]{Ch.Ga.Tr.va2024}), with relative
ease, c.f. Lemma \ref{lem:support-dists}. 

\subsubsection*{Regularity of Weil-Petersson volumes\label{subsec:overview-Effective-approximation-by}}

The key technical result of the paper is Theorem \ref{thm:main-technical-result}.
It implies that the spectral statistics of Weil-Petersson random $f\left(\sqrt{\Delta-\frac{1}{4}}\right)$
are in a Gevrey class. A related expansion appears in the work of
Anantharaman and Monk \cite{An.Mo2023}. We stress two points here. 

Firstly, it is not too difficult to prove the $\textit{existence}$
of an asymptotic expansion of $\mathbb{E}_{g}\left[\textup{tr}\left(\left(f\left(\sqrt{\Delta-\frac{1}{4}}\right)\right)^{t}\right)\right]$
in powers of $g$ using Selberg's trace formula (c.f. Theorem \ref{thm:Selberg-TF}),
an application of Mirzakhani's integration formula (\cite[Theorem 6.1]{An.Mo2023},
c.f. Lemma \ref{lem:integration-lemma}) and large genus expansions
of ratios of Weil-Petersson volumes due to Mirzakhani-Zograf \cite{Mi.Zo2015}
and Anantharaman-Monk \cite{An.Mo2022}. However, obtaining a detailed
understanding of the coefficients presents a serious challenge, and
this issue is the focus of the vast majority of the impressive articles
\cite{An.Mo2023,An.Mo2025}. In contrast, we do not require any knowledge
of the coefficients $a_{i}^{t}$ beyond $a_{0}^{t}$ and $a_{1}^{t}$
which are already known explicitly \cite{Mi.Pe19}.

Secondly, the main difficulty for us is to obtain an inequality with
the precise error term $\frac{\left(cq\right)^{cq}}{g^{q}}$. In order
to achieve this we need to prove effective genus expansions for ratios
of Weil-Petersson volumes with explicit error terms, extending results
of Mirzakhani-Zograf \cite{Mi.Zo2015}. An example of the kind of
results we prove (c.f. Section \ref{sec:Large-genus-expansion}) is
the following which we believe to be of independent interest.
\begin{thm}
\label{thm:Wp-vols}There exists $c>0$ such that for each $n$ there
are continuous functions $\left\{ F_{n,j}\left(\mathbf{x}\right)\right\} _{j}$
such that for each $k$, 
\[
\left|\frac{V_{g,n}\left(\mathbf{x}\right)}{V_{g,n}}-\sum_{j=0}^{k}\frac{F_{n,j}\left(\mathbf{x}\right)}{g^{j}}\right|\leqslant\frac{\left(|\mathbf{x}|+1\right)^{ck}\left(ckn\right)^{ck}\exp\left(|\mathbf{x}|\right)}{g^{k+1}},
\]
for all $g\geqslant(ck)^{c}$.
\end{thm}

Theorem \ref{thm:Wp-vols} and the results of Section \ref{sec:Large-genus-expansion}
rely on analogous expansions for the intersection numbers $\left[\tau_{1}\cdots\tau_{n}\right]_{g,n}$
of tautological classes and of ratios of Weil-Petersson volumes, ultimately
showing that they lie in a Gevrey class.

In relation to this, it was shown by Mirzakhani and Zograf \cite{Mi.Zo2015}
that for each $n,k$ there is a polynomial $P_{n,k}$ of degree $2k$
such that for any $\mathbf{d}=\left(d_{1},\dots,d_{n}\right)\in\mathbb{Z}_{\geqslant0}^{n}$,
there are constants $\left\{ e_{n,\mathbf{d},i}\right\} _{i\geqslant0}$
such that for any $g$, 
\begin{equation}
\left|\frac{\left[\tau_{d_{1}}\dots\tau_{d_{n}}\right]_{g,n}}{V_{g,n}}-\sum_{i=1}^{k-1}\frac{e_{n,\mathbf{d},i}}{g^{i}}\right|\leqslant\frac{P_{n,k}\left(d_{1},\dots,d_{n}\right)}{g^{k}}.\label{eq:int-num-mz}
\end{equation}

To establish estimates of the strength of Theorem \ref{thm:main-technical-result}
we need to obtain a much more precise understanding of $P_{n,k}\left(d_{1},\dots,d_{n}\right)$
in terms of $n$ and $k$ than that obtained by Mirzakhani and Zograf.
For this, we develop a more involved version of their induction algorithm
used to prove (\ref{eq:int-num-mz}). 

The reasoning for this is that when applying an inductive argument
to obtain (\ref{eq:int-num-mz}) based on the recursion relations
of intersection numbers, we find that in the passage to an expansion
of order $k+1$ from order $k$ the degree $\asymp k$ terms in the
polynomial $P_{n,k}\left(d_{1},\dots,d_{n}\right)$ create new factors
of size $\asymp$$k!$ in their contribution to $P_{n,k+1}\left(d_{1},\dots,d_{n}\right)$.
For the desired $(ck)^{ck}$ error, it is essential to that these
factors only contribute additively at each step rather than multiplicatively
which would result in errors of the size $(ck)^{ck^{2}}$ that are
too large for the polynomial method to work.

It is therefore essential for us to track and control the individual
coefficients of monomials in $P_{n,k}$ (and in fact, for similar
polynomial bounds for the individual coefficients $e_{n,\mathbf{d},i}$)
at each step throughout the induction. This results in our enhanced
induction process to be highly non-trivial and far more involved in
order to obtain the desired level of precision, especially in light
of the complexity of the recursive formulae that govern the intersection
numbers (see Theorem \ref{thm-Recurrsions}). 

With all this considered, we eventually prove (see Theorem \ref{thm:main-estimate-WP-expansion})
that the coefficient of $d_{1}^{t_{1}}\dots d_{n}^{t_{n}}$ in $P_{n,k}$
are bounded by
\[
\frac{(ck)^{ck}(n+k)^{k}}{t_{1}!\cdots t_{n}!}.
\]

We believe that this new understanding of the expansions of volumes
and intersection numbers also bring us closer to those obtained in
the computations found through physics heuristics \cite{Ey.Ga.Gr.Le.Sc2024},
see also \cite{Ey.Ga.Gi.Gr.Le2023} for recent rigorous expansions.

\subsection*{Organisation of the paper}

The effective genus expansions for Weil-Petersson volumes proven in
Section \ref{sec:Large-genus-expansion} are the most difficult and
technical part of the paper. The key result of $\S$\ref{sec:Large-genus-expansion}
is Corollary \ref{thm:main-estimate-WP-expansion} which can treated
as a black box in the remainder of the paper.

In Section \ref{sec:Geometric-estimates}, we prove Theorem \ref{thm:main-technical-result},
assuming the results of $\S$\ref{sec:Large-genus-expansion}. In
Section \ref{sec:Proof-of-Theorem} we prove Theorem \ref{thm:Main-Thm}.

\subsection*{Notation}

For $\mathbf{x}=\left(x_{1},\dots,x_{n}\right)\in\mathbb{R}^{n}$
we write $\left|\mathbf{x}\right|=\sum_{i=1}^{n}\left|x_{i}\right|$.

For functions $x,y:\mathbb{N}\to\mathbb{R}$ we write $x=O\left(y(g)\right)$
to mean that there are constants $C_{0},g_{0}>0$ so that $\left|x(g)\right|\leqslant C_{0}y(g)$
for all $g\geqslant g_{0}$. If the constants $C_{0},g_{0}$ depend
on another parameter $\ep$ we indicate this by writing $x=O_{\ep}\left(y(g)\right)$.

Throughout the paper, \textbf{with the exception of Theorem \ref{thm:expansions}
and its constituent propositions}, $C$ and $c$ will denote positive
universal constants whose value we do not need to track. We warn that
sometimes the precise values of $C$ and $c$ may change from line
to line as they absorb other universal constants.

For a closed hyperbolic surface $X$ of genus $g$ and a trace class
operator $F:L^{2}\left(X\right)\to L^{2}\left(X\right)$, we denote
the normalised trace $\tr\ F\eqdf\frac{1}{g}\textup{Tr}\ F$, where
$\textup{Tr}$ denotes the usual (non-normalised) trace. 

We use the convention that the Fourier transform of an appropriate
function $\phi$ is given by $\hat{\phi}(r)=\int_{-\infty}^{\infty}e^{-irx}\phi(x)\mathrm{d}x$. 

\subsection*{Acknowledgments}

We thank Alessandro Giacchetto, Michael Magee and Ramon van Handel
for very helpful conversations and suggestions on an earlier version
of this work. D.M. is funded by an Argelander Grant awarded by the
University of Bonn. He also gratefully acknowledges funding from the
Deutsche Forschungsgemeinschaft under Germany's Excellence Strategy-GZ
2047/1, project-id 390685813. J.T. is funded by the Leverhulme Trust
through a Leverhulme Early Career Fellowship (Grant No. ECF-2024-440).

\section{Geometric estimates\label{sec:Geometric-estimates}}

We now fix the function $f$ that we will take of $\sqrt{\Delta-\frac{1}{4}}$
in the remainder of the paper. Let $f_{0}\in C_{c}^{\infty}(\mathbb{R})$
be a non-negative and even function whose support is equal to $(-1,1)$
and whose Fourier transform $\hat{f_{0}}$ is non-negative on $\mathbb{R}\cup i\mathbb{R}$.
The existence of $f_{0}$ is shown in \cite[Section 2.2]{Ma.Na.Pu2022}.
 We now set $f\eqdf\hat{f_{0}}$ so that $f\left(\sqrt{\Delta-\frac{1}{4}}\right)$
is a bounded operator. By evenness and non-negativity of $f_{0}$,
the function $t\in[0,\frac{1}{2}]\mapsto f(ti)$ (corresponding to
eigenvalues of $\Delta$ below $\frac{1}{4}$) is strictly increasing,
$f(0)>0$, and $0\leqslant f([0,\infty))\leqslant f(0)$ (corresponding
to eigenvalues of $\Delta$ above $\frac{1}{4}$). In particular,
the spectrum of $f\left(\sqrt{\Delta-\frac{1}{4}}\right)$ is contained
in $[0,f(\frac{i}{2})]$.

With the convention of the Fourier transform as $\hat{\phi}(r)=\int_{-\infty}^{\infty}e^{-irx}\phi(x)\mathrm{d}x$,
the convolution theorem states that $\widehat{\phi_{1}*\phi_{2}}=\hat{\phi_{1}}\hat{\phi_{2}}$.

The purpose of this section is to establish Theorem \ref{thm:main-technical-result}.

\subsection{Background\label{subsec:Background}}

\subsection*{Moduli space}

Let $\Sigma_{g,n}$ denote a topological surface with genus $g$ and
$n$ labeled boundary components where $2g+n\geqslant3$. A \emph{marked
surface} of signature $\left(g,n\right)$ is a pair $\left(X,\varphi\right)$
where $X$ is a hyperbolic surface and $\varphi:\Sigma_{g,n}\to X$
is a homeomorphism. Given $\left(\ell_{1},...,\ell_{n}\right)\in\mathbb{R}_{\geqslant0}^{n}$,
we define the \emph{Teichmüller space} $\mathcal{T}_{g,n}\left(\ell_{1},\dots,\ell_{n}\right)$
by
\begin{align*}
\mathcal{T}_{g,n}\left(\ell_{1},...,\ell_{n}\right) & \stackrel{\text{def}}{=}\left\{ \substack{\text{\text{Marked surfaces} }\left(X,\varphi\right)\text{ of signature }\left(g,n\right)\\
\text{with labelled totally geodesic boundary components }\\
\left(\beta_{1},\dots,\beta_{n}\right)\text{ with lengths }\left(\ell_{1},\dots,\ell_{n}\right)
}
\right\} /\sim,
\end{align*}
where $\left(X_{1},\varphi_{1}\right)\sim\left(X_{2},\varphi_{2}\right)$
if and only if there exists an isometry $m:X_{1}\to X_{2}$ such that
$\varphi_{2}$ and $m\circ\varphi_{1}$ are isotopic. Let $\text{Homeo}^{+}\left(\Sigma_{g,n}\right)$
denote the group of orientation preserving homeomorphisms of $\Sigma_{g,n}$
which leave every boundary component setwise fixed and let $\text{Homeo}_{0}^{+}\left(\Sigma_{g,n}\right)$
denote the subgroup of homeomorphisms isotopic to the identity. The
\emph{mapping class group} is defined as 
\[
\text{MCG}_{g,n}\stackrel{\text{def}}{=}\text{Homeo}^{+}\left(\Sigma_{g,n}\right)/\text{Homeo}_{0}^{+}\left(\Sigma_{g,n}\right).
\]
$\text{Homeo}^{+}\left(\Sigma_{g,n}\right)$ acts on $\mathcal{T}_{g,n}\left(\ell_{1},...,\ell_{n}\right)$
by pre-composition of the marking, and $\text{Homeo}_{0}^{+}\left(\Sigma_{g,n}\right)$
acts trivially, hence $\text{MCG}_{g,n}$ acts on $\mathcal{T}_{g,n}\left(\ell_{1},...,\ell_{n}\right)$
and we define the \emph{moduli space} $\mathcal{M}_{g,n}\left(\ell_{1},...,\ell_{n}\right)$
by 
\[
\mathcal{M}_{g,n}\left(\ell_{1},...,\ell_{n}\right)\stackrel{\text{def}}{=}\mathcal{T}_{g,n}\left(\ell_{1},...,\ell_{n}\right)/\text{MCG}_{g,n}.
\]

\subsection*{Weil-Petersson metric}

By the work of Goldman \cite{Go1984}, the space $\mathcal{T}_{g,n}\left(\mathbf{x}\right)$
carries a natural symplectic structure known as the Weil-Petersson
symplectic form and is denoted by $\omega_{WP}$. In the case where
$\mathbf{x}=\mathbf{0}$, up to a scalar multiple, this agrees with
the form arising from the Weil-Petersson K$\ddot{\text{a}}$hler metric
on $\mathcal{T}_{g,n}$. It is invariant under the action of the mapping
class group and descends to a symplectic form on the quotient $\mathcal{M}_{g,n}\left(\mathbf{x}\right)$.
The form $\omega_{WP}$ induces the volume form
\[
\text{dVol}_{WP}\stackrel{\text{def}}{=}\frac{1}{\left(3g-3+n\right)!}\bigwedge_{i=1}^{3g-3+n}\omega_{WP},
\]
which is also invariant under the action of the mapping class group
and descends to a volume form on $\mathcal{M}_{g,n}\left(\mathbf{x}\right)$.
We write $\mathrm{d}X$ as shorthand for $\text{dVol}_{WP}(X)$. We
let $V_{g,n}\left(\mathbf{x}\right)$ denote $\text{Vol}_{WP}\left(\mathcal{M}_{g,n}\left(\mathbf{x}\right)\right)$,
the total volume of $\mathcal{M}_{g,n}\left(\mathbf{x}\right)$, which
is finite. We write $V_{g,n}$ to denote $V_{g,n}\left(\mathbf{0}\right)$.

As in \cite{Gu.Pa.Yo11,MirzakhaniRandom}, we define a probability
measure $\mathbb{P}_{g}$ on $\mathcal{M}_{g}$ by normalising $\text{dVol}_{WP}$.
We write $\mathbb{E}_{g}$ to denote expectation with respect to $\mathbb{P}_{g}$.

\subsection*{Selberg's trace formula}

We briefly recall Selberg's trace formula \cite{Selberg,Be2016} .
\begin{thm}
\label{thm:Selberg-TF}Let $X$ be a closed hyperbolic surface. For
any $\phi\in C_{c}^{\infty}\left(\mathbb{R}\right)$,
\begin{align*}
\sum_{j}\hat{\phi}\left(\sqrt{\lambda_{j}-\frac{1}{4}}\right)= & \frac{\textup{Vol}\left(X\right)}{4\pi}\int_{\frac{1}{4}}^{\infty}\hat{\phi}\left(\sqrt{r-\frac{1}{4}}\right)\tanh\left(\pi\sqrt{r-\frac{1}{4}}\right)\mathrm{d}r\\
 & +\sum_{\gamma\in\mathcal{P}(X)}\sum_{k=1}^{\infty}\frac{\ell_{\gamma}\left(X\right)}{2\sinh\left(\frac{k\ell_{\gamma}(x)}{2}\right)}\phi\left(k\ell_{\gamma}\left(X\right)\right),
\end{align*}
where $\mathcal{P}(X)$ denotes the set of primitive oriented closed
geodesics on $X$ and $\hat{\phi}$ denotes the Fourier transform
of $\phi$.
\end{thm}

\subsection*{Filling geodesics}
\begin{defn}
Let $X$ be a hyperbolic surface, possibly with totally geodesic boundary,
and let $\gamma$ be a closed geodesic on $X.$ We say that $\gamma$
\emph{fills} $X$ if $X\setminus\gamma$ is a union of discs and cylinders
such that each cylinder is homotopic to a boundary component of $X$.
\end{defn}

\begin{lem}
\label{lem:filling-bound}There is a $C>0$ such that for any hyperbolic
surface $X$, any geodesic $\gamma$ which fills $X$ satisfies $\ell_{\gamma}\left(X\right)\geqslant C\left|\chi\left(X\right)\right|.$
\end{lem}

\begin{proof}
Let $\mathcal{G}$ be the graph on $X$ whose vertices are the self-intersections
of $\gamma$ and whose edges are the geodesic arcs between them. Consider
a regular neighbourhood $N_{\ep}\left(\gamma\right)\subset X$ of
$\gamma$ in $X$. Since $\gamma$ is filling, the inclusion map $\iota:N_{\ep}\left(\gamma\right)\to X$
defines a surjective homomorphism 
\[
\iota_{*}:\pi_{1}\left(N_{\ep}\left(\gamma\right)\right)\to\pi_{1}\left(X\right).
\]
Since $N_{\ep}\left(\gamma\right)$ deformation retracts onto $\mathcal{G}$,
we have $\pi_{1}\left(N_{\ep}\left(\gamma\right)\right)=\pi_{1}\left(\mathcal{G},v\right)$.
It follows that there are at least $\left|\chi\left(X\right)\right|+1$
loops in $\mathcal{G}$ which are homotopically non-trivial in $X$.
Partitioning these into pairs we can find at least $\lfloor\frac{\left|\chi\left(X\right)\right|+1}{2}\rfloor$
disjoint components of $\gamma$ which are freely homotopic in $X$
to a figure of eight. Since any non-simple geodesic has length at
least $2\textup{arccosh}3$ \cite[Theorem 4.2.2]{Bu2010}, we conclude
that 

\[
\ell_{\gamma}(X)\geqslant\lfloor\frac{\left|\chi\left(X\right)\right|+1}{2}\rfloor\cdot2\textup{arccosh}3\geqslant C|\chi(X)|.
\]
\end{proof}
We record a basic upper bound on the number of filling geodesics on
a hyperbolic surface, see for example \cite[Lemma 10]{Wu.Xu2022}.
\begin{lem}
\label{lem:counting-filling}There is a $C>0$ so that for any hyperbolic
surface $X$, the number of closed geodesics of length $\leqslant L$
that fill $X$ is bounded above by $C\left|\chi\left(X\right)\right|e^{L}$.
\end{lem}

\subsection{Estimates for Weil-Petersson volumes}

In this subsection we recall some useful estimates for Weil-Petersson
volumes which we will apply throughout the remainder of the section.
\begin{lem}[{\cite[Lemma 3.2, part 2]{MirzakhaniRandom}}]
For any $g,n$ with $2g+n\geqslant3$ and for any $\mathbf{x}=\left(x_{1},\dots,x_{n}\right)$,\label{lem:Wp-expbound}
\[
1\leqslant\frac{V_{g,n}\left(\mathbf{x}\right)}{V_{g,n}}\leqslant\exp\left(\frac{\sum_{i=1}^{n}x_{i}}{2}\right).
\]
\end{lem}

We often apply the following bound which is a simple consequence of
work of Gruschevsky \cite{Gr2001} and Schumacher-Trapani \cite{ST2001}.
\begin{lem}
\label{lem:(Grushevsky's-bound,-)}There is a $C>0$ so that for any
$g,n$ with $2g+n\geqslant3$,
\[
V_{g,n}\leqslant C^{2g+n}\left(2g+n\right)!
\]
\end{lem}

\begin{proof}
By \cite[Lemma 3.2, part 3]{MirzakhaniRandom} we have that 
\[
V_{g,n}\leqslant V_{g+1,n-2}
\]
which implies that 
\[
V_{g,n}\leqslant V_{g+\lfloor\frac{n}{2}\rfloor,n-2\lfloor\frac{n}{2}\rfloor}.
\]
It then follows by \cite[Section 7]{Gr2001} when $n$ is odd and
by \cite[Corollary 1]{ST2001} when $n$ is even, that there is a
constant $C>0$ so that 
\[
V_{g,n}\leqslant C^{2g+n}\left(2g+n\right)!.
\]
\end{proof}
Mirzakhani and Zograf \cite{Mi.Zo2015} provide a large genus asymptotic
formula for $V_{g,n}$.
\begin{thm}[{\cite[Theorem 1.8]{Mi.Zo2015}}]
\label{Large-g-Vg}There exists a constant $B>0$ such that for any
$n\geqslant0$,
\[
V_{g,n}=\frac{B}{\sqrt{g}}\left(2g-3+n\right)!\left(4\pi^{2}\right)^{2g-3+n}\left(1+O\left(\frac{1+n^{2}}{g}\right)\right),
\]
as $g\to\infty.$
\end{thm}

For convenience we import notation from \cite{NXW2023,Wu.Xu2022}.
We define

\begin{equation}
W_{r}=\begin{cases}
V_{\frac{r}{2}+1,0} & \text{if \ensuremath{r} is even, }\\
V_{\frac{r+1}{2},1} & \text{if \ensuremath{r} is odd. }
\end{cases}\label{eq:W-notation}
\end{equation}

\begin{lem}[{\cite[Lemma 24]{NXW2023}}]
\label{Lemma 22} Assume $q\geqslant1,n_{1},...,n_{q}\geqslant0,r\geqslant2$.
There exists two universal constants $c,D>0$ such that 
\[
\sum_{\mathbf{g}}V_{g_{1},n_{1}}\cdots V_{g_{q},n_{q}}\leqslant c\left(\frac{D}{r}\right)^{q-1}W_{r},
\]
where the sum is taken over $\mathbf{g}=(g_{1},...,g_{q})\in\mathbb{Z}_{\geqslant0}^{q}$
such that $2g_{i}-2+n_{i}\geqslant1\text{ for }i=1,...,q$ and $\sum_{i=1}^{q}(2g_{i}-2+n_{i})=r$.
\end{lem}

\subsection{The contribution of simple geodesics}

For a closed hyperbolic surface $X$ we write $\mathcal{P}^{\textup{simp}}\left(X\right)$
(resp. $\mathcal{P}^{\textup{n-simp}}\left(X\right)$) to denote the
set of primitive oriented simple (resp. non-simple) closed geodesics
on $X$. The purpose of this section is to prove the following.
\begin{prop}
\label{prop:effective-simple-expansion}There is a constant $c>0$
such that for any $L>0$ and any function $F_{L}$ supported in $[0,L]$
such that $\ell\mapsto\ell F_{L}\left(\ell\right)$ is continuous,
there exist continuous functions $\left\{ f_{j}^{\textup{simp}}\right\} $
such that for any $k\geqslant1$,
\begin{align*}
 & \left|\mathbb{E}_{g}\left[\sum_{\gamma\in\mathcal{P}^{\textup{simp}}\left(X\right)}F_{L}\left(\ell_{\gamma}\left(X\right)\right)\right]-\int_{0}^{\infty}4\frac{\sinh\left(\frac{\ell}{2}\right)^{2}}{\ell}F_{L}\left(\ell\right)d\ell-\sum_{j=1}^{k}\frac{1}{g^{j}}\int_{0}^{\infty}F_{L}\left(\ell\right)f_{j}^{\textup{simp}}(\ell)\mathrm{d}\ell\right|\\
\leqslant & \frac{L^{ck}\left(ck\right)^{ck}}{g^{k+1}}\int_{0}^{\infty}\ell F_{L}\left(\ell\right)\exp\left(\ell\right)\mathrm{d}\ell,
\end{align*}
for $g>(ck)^{c}$.
\end{prop}

\begin{proof}
By Mirzakhani's integration formula \cite{Mi2007},
\begin{equation}
\mathbb{E}_{g}\left[\sum_{\gamma\in\mathcal{P}^{\textup{simp}}\left(X\right)}F_{L}\left(\ell_{\gamma}\left(X\right)\right)\right]=\int_{0}^{\infty}\frac{V_{g-1,2}\left(\ell,\ell\right)}{V_{g}}F_{L}\left(\ell\right)\ell\mathrm{d}\ell+\int_{0}^{\infty}\sum_{i=1}^{\lfloor\frac{g}{2}\rfloor}\frac{V_{i,1}\left(\ell\right)V_{g-i,1}\left(\ell\right)}{V_{g}}F_{L}\left(\ell\right)\ell\mathrm{d}\ell.\label{eq:MIF_simple}
\end{equation}
We begin by treating the first summand. 

By Corollary \ref{thm:main-estimate-WP-expansion}, there exists
a constant $c>0$ and continuous functions $\left\{ w_{i}(\ell)\right\} _{i}$
such that for any $k\in\mathbb{N}$ and $g>(ck)^{c}$,
\[
\left|\frac{V_{g-1,2}\left(\ell,\ell\right)}{V_{g}}-\frac{4\sinh\left(\frac{\ell}{2}\right)^{2}}{\ell^{2}}-\sum_{i=1}^{k}\frac{w_{i}(\ell)}{g^{i}}\right|\leqslant\frac{\left((\ell+1)ck\right)^{ck}\exp\left(\ell\right)}{g^{k+1}}.
\]
It follows that 
\begin{align}
 & \left|\int_{0}^{\infty}\frac{V_{g-1,2}\left(\ell,\ell\right)}{V_{g}}F_{L}\left(\ell\right)\ell\mathrm{d}\ell-\int\frac{4\sinh\left(\frac{\ell}{2}\right)^{2}}{\ell^{2}}\ell F_{L}\left(\ell\right)\mathrm{d}\ell-\sum_{i=1}^{k}\int_{0}^{\infty}\frac{w_{i}(\ell)}{g^{i}}F_{L}\left(\ell\right)\ell\mathrm{d}\ell\right|\label{eq:simple-leading}\\
\leqslant & \int_{0}^{\infty}\frac{\left((\ell+1)ck\right)^{ck}\exp\left(\ell\right)}{g^{k}}F_{L}\left(\ell\right)\ell\mathrm{d}\ell\leqslant\frac{L^{ck}\left(ck\right)^{ck}}{g^{k+1}}\int_{0}^{\infty}\ell F_{L}\left(\ell\right)\exp\left(\ell\right)\mathrm{d}\ell.\nonumber 
\end{align}
 Also for $1\leqslant i\leqslant\left\lceil \frac{k}{2}\right\rceil $
by Corollary \ref{thm:main-estimate-WP-expansion} there exist continuous
functions $\left\{ f_{j}^{i}\left(\ell\right)\right\} $ so that 
\[
\left|\frac{V_{g-i,1}\left(\ell\right)}{V_{g}}-\sum_{j=2i-1}^{k}\frac{f_{j}^{i}\left(\ell\right)}{g^{j}}\right|\leqslant\frac{\left(ck(\ell+1)\right)^{ck}\exp\left(\frac{\ell}{2}\right)}{g^{k+1}}.
\]
Now for $i>\left\lceil \frac{k}{2}\right\rceil $, 
\[
\frac{V_{g-i,1}\left(\ell\right)V_{i,1}\left(\ell\right)}{V_{g}}\leqslant e^{\ell}\frac{V_{g-i,1}V_{i,1}}{V_{g}}.
\]
It follows from Lemma \ref{lem:WP-sum-product-bound} and Remark \ref{rem:volume-product-cylinder}
that
\[
\sum_{i=\left\lceil \frac{k}{2}\right\rceil +1}^{\lfloor\frac{g}{2}\rfloor}\frac{V_{i,1}V_{g-i,1}}{V_{g}}\leqslant\frac{\left(ck\right)^{ck}}{g^{k+1}}.
\]
Thus we see that 
\begin{align}
 & \left|\int_{0}^{\infty}\sum_{i=1}^{\lfloor\frac{g}{2}\rfloor}\frac{V_{i,1}\left(\ell\right)V_{g-i,1}\left(\ell\right)}{V_{g}}F_{L}\left(\ell\right)\ell\mathrm{d}\ell-\sum_{i=1}^{\left\lceil \frac{k}{2}\right\rceil }\sum_{j=2i-1}^{k}\int_{0}^{\infty}\frac{f_{j}^{i}\left(\ell\right)V_{i,1}(\ell)}{g^{j}}\ell F_{L}\left(\ell\right)\mathrm{d}\ell\right|\label{eq:Simple-error}\\
\leqslant & \sum_{i=1}^{\left\lceil \frac{k}{2}\right\rceil }\frac{V_{i,1}\cdot\left(ck\right)^{ck}}{g^{k+1}}\int_{0}^{\infty}(\ell+1)^{ck}e^{\ell}\ell F_{L}(\ell)\mathrm{d}\ell+\sum_{i=\left\lceil \frac{k}{2}\right\rceil +1}^{\lfloor\frac{g}{2}\rfloor}\frac{V_{i,1}V_{g-i,1}}{V_{g}}\int_{0}^{\infty}\ell F_{L}\left(\ell\right)e^{\ell}\mathrm{d}\ell\nonumber \\
\leqslant & \frac{(k+1)!L^{ck}\left(ck\right)^{ck}}{g^{k+1}}\int_{0}^{\infty}\ell F_{L}\left(\ell\right)e^{\ell}\mathrm{d}\ell\leqslant\frac{L^{ck}\left(ck\right)^{ck}}{g^{k+1}}\int_{0}^{\infty}\ell F_{L}\left(\ell\right)e^{\ell}\mathrm{d}\ell,\nonumber 
\end{align}
where the penultimate inequality uses Lemma \ref{lem:(Grushevsky's-bound,-)}
and the constant $c$ changes from line to line. Note that the coefficient
of $g^{-j}$ in (\ref{eq:Simple-error}) can equivalently be written
as
\[
\sum_{i=1}^{\left\lceil \frac{j}{2}\right\rceil }\int_{0}^{\infty}f_{j}^{i}\left(\ell\right)V_{i,1}(\ell)\ell F_{L}\left(\ell\right)\mathrm{d}\ell.
\]
The conclusion then follows from (\ref{eq:MIF_simple}), (\ref{eq:simple-leading})
and (\ref{eq:Simple-error}).
\end{proof}

\subsection{The contribution of non-simple geodesics}

In this section we prove the following.
\begin{prop}
\label{prop:effective-nonsimple-expansion}There is a constant $c>2$
such that for any $L>0$ and any smooth function $F_{L}$ supported
in $[0,L]$ there exist constants $\left\{ f_{j}^{\textup{n-simp}}\right\} $
such that for any $k\geqslant1$,
\[
\left|\mathbb{E}_{g}\left[\sum_{\gamma\in\mathcal{P}^{\textup{n-simp}}\left(X\right)}F_{L}\left(\ell_{\gamma}\left(X\right)\right)\right]-\sum_{j=1}^{k}\frac{f_{j}^{\textup{n-simp}}}{g^{j}}\right|\leqslant\frac{L^{c(L+k)}\left(ck\right)^{ck}}{g^{k+1}}\|F_{L}\|_{\infty},
\]
for $g>\max\left\{ (ck)^{c},cL\right\} $. 
\end{prop}

We begin by introducing some notation. For $g_{0},n_{0}$ with $2g_{0}+n_{0}\geqslant3$,
we define $\mathcal{A}_{g_{0},n_{0}}$ to be the set of triples $\left(\mathbf{i},\mathbf{j},q\right)$
where $q\geqslant1$, $\mathbf{i}=\left(g_{1},\dots,g_{q}\right)$
and $\mathbf{j}$ is a partition of $\left\{ 1,\dots,n_{0}\right\} $
into $q$ non-empty subsets $I_{1},\dots,I_{q}$ with $\left|I_{i}\right|\eqdf n_{i}$
so that $\sum_{i=1}^{q}n_{i}=n_{0}$ and $i\mapsto\min I_{i}$ is
an increasing function. We further require that 
\begin{lyxlist}{00.00.0000}
\item [{i)}] $2g_{i}+n_{i}-2\geqslant1$ and $g_{i}\geqslant0$ or $(g_{i},n_{i})=(0,2),$ 
\item [{ii)\label{ii)index-condition-ii}}] $\sum_{i=1}^{q}2g_{i}-2+n_{i}=2g-2g_{0}-n_{0}.$
\end{lyxlist}
Given $\mathbf{x}=\left(x_{1},\dots,x_{n_{0}}\right)$, we write $\mathbf{x}^{(j)}=\left(x_{i}:i\in I_{j}\right)\in\mathbb{R}^{n_{j}}$. 

The following lemma is proven by Anantharaman and Monk \cite[Theorem 6.1]{An.Mo2023}.
\begin{lem}
\label{lem:integration-lemma}We have that 
\begin{multline*}
\mathbb{E}_{g}\left[\sum_{\gamma\in\mathcal{P}^{\textup{n-simp}}\left(X\right)}F_{L}\left(\ell_{\gamma}\left(X\right)\right)\right]\\
=\sum_{\substack{\left(g_{0},n_{0}\right)\\
2g_{0}-2+n_{0}\geqslant1
}
}\frac{1}{n_{0}!}\int_{\mathbb{R}_{\geqslant0}^{n_{0}}}\frac{1}{V_{g_{0},n_{0}}\left({\bf x}\right)}\int_{\mathcal{M}_{g_{0},n_{0}}(\mathbf{x})}\sum_{\alpha\textup{ filling }S_{g_{0},n_{0}}}F_{L}\left(\ell_{\alpha}(Y)\right)\mathrm{d}\textup{Vol}_{\mathrm{WP}}\left(Y\right)\phi_{g}^{\left(g_{0},n_{0}\right)}(\mathbf{x})\mathrm{d}\mathbf{x},
\end{multline*}
where 
\begin{equation}
\phi_{g}^{\left(g_{0},n_{0}\right)}\left(\mathbf{x}\right)=x_{1}\dots x_{n_{0}}\frac{V_{g_{0},n_{0}}\left(\mathbf{x}\right)}{V_{g}}\sum_{\left(\mathbf{i},\mathbf{j},q\right)\in\mathcal{A}_{g_{0},n_{0}}}V_{g_{1},n_{1}}\left(\mathbf{x}^{(1)}\right)\dots V_{g_{q},n_{q}}\left(\mathbf{x}^{(q)}\right).\label{eq:vol-terms}
\end{equation}
When $(g_{i},n_{i})=(0,2)$, the volume $V_{g_{i},n_{i}}(x,y)$ is
interpreted as $\frac{1}{x}\delta_{x=y}$, where $\delta_{x=y}=1$
if $x=y$ and is zero otherwise.
\end{lem}

In order to prove Proposition \ref{prop:effective-nonsimple-expansion},
we first show that $\phi_{g}^{\left(g_{0},n_{0}\right)}\left(\mathbf{x}\right)$
has an effective genus expansion.
\begin{lem}
\label{lem:vol-function-bound}There exists a constant $c>2$ such
that for any $\left(g_{0},n_{0}\right)$ and $k\geqslant1$ there
are continuous functions $\left\{ h_{j}^{g_{0},n_{0}}\left(\mathbf{x}\right)\right\} _{j=1}^{k}$
such that 
\[
\left|\phi_{g}^{\left(g_{0},n_{0}\right)}\left(\mathbf{x}\right)-\sum_{j=1}^{k}\frac{h_{j}^{g_{0},n_{0}}\left(\mathbf{x}\right)}{g^{j}}\right|\leqslant\frac{\left(ck\right)^{ck}\left(c\left(g_{0}+n_{0}\right)\right)^{c\left(g_{0}+n_{0}\right)}(1+\left|\mathbf{x}\right|)^{ck+n_{0}}\exp\left(|\mathbf{x}|\right)}{g^{k+1}},
\]
for all $g>\left(ck\right)^{c}$.
\end{lem}

Before proceeding with the proof of Lemma \ref{lem:vol-function-bound},
we prove an estimate for sums and products of Weil-Petersson volumes
which will help control the error terms. The proof is a straightforward
adaptation of \cite[Lemma 24]{NXW2023}.
\begin{lem}
\label{lem:WP-sum-product-bound}There is a $c>2$ such that for $0\leqslant k\leqslant\sqrt{g}$
and $(g_{0},n_{0})$, if we let $\tilde{\mathcal{A}}_{g_{0},n_{0}}$
denote the subset of $\left(\mathbf{i},\mathbf{j},q\right)\in\mathcal{A}_{g_{0},n_{0}}$
with $\max_{1\leqslant j\leqslant q}2g_{j}+n_{j}-3\leqslant2g-4-k$,
then we have
\[
\frac{1}{V_{g}}\sum_{\left(\mathbf{i},\mathbf{j},q\right)\in\tilde{\mathcal{A}}_{g_{0},n_{0}}}V_{g_{1},n_{1}}\dots V_{g_{q},n_{q}}\leqslant\frac{\left(ck\right)^{ck}\left(c\left(g_{0}+n_{0}\right)\right)^{c\left(g_{0}+n_{0}\right)}}{g^{k+1}},
\]
where we omit any factors of $V_{0,2}$ appearing in the product. 
\end{lem}

\begin{rem}
\label{rem:volume-product-cylinder}The same result also holds with
an analogous proof in the case where $g_{0}=0$ and $n_{0}=2$. This
corresponds to the case of a cylinder which is filled by a simple
closed geodesic.
\end{rem}

\begin{proof}
By \cite[Lemma 3.2, part 3]{MirzakhaniRandom} we see that for each
$n_{i}\geqslant2$, we have 
\[
V_{g_{i},n_{i}}\leqslant V_{g_{i}+\lfloor\frac{n_{i}-2}{2}\rfloor,n_{i}-2\lfloor\frac{n_{i}-2}{2}\rfloor}.
\]
 This allows us to apply Theorem \ref{Large-g-Vg} which tells us
that there exists $C_{1}>0$ with
\begin{align}
V_{g_{1},n_{1}}\cdots V_{g_{q},n_{q}} & \leqslant C_{1}^{q}\prod_{j=1}^{q}\frac{\left(4\pi^{2}\right)^{2g_{j}+n_{j}-3}\left(2g_{j}+n_{j}-3\right)!}{\sqrt{g_{j}+\max\left\{ \lfloor\frac{n_{j}-2}{2}\rfloor,0\right\} }},\label{eq:WPvolreduced}
\end{align}
where, since $V_{0,3}=1$, we interpret the product in (\ref{eq:WPvolreduced})
as only over $V_{g_{i},n_{i}}$ with $g_{j}+\max\left\{ \lfloor\frac{n_{j}-2}{2}\rfloor,0\right\} >0$
(recall that we omit the $V_{0,2}$ factors, so they also can be interpreted
as $1$ throughout). We further see by Theorem \ref{Large-g-Vg} that
there is a $B>0$ with
\begin{equation}
V_{g}\geqslant\frac{B}{\sqrt{g}}\left(2g-3\right)!\left(4\pi^{2}\right)^{2g-3}.\label{Vg-asymp}
\end{equation}
We introduce the notation $\text{\ensuremath{\overline{n_{j}}}}\eqdf\max\left\{ \lfloor\frac{n_{j}-2}{2}\rfloor,0\right\} $.
By applying (\ref{eq:WPvolreduced}) and (\ref{Vg-asymp}) we see
that 

\[
\sum_{\left(\mathbf{i},\mathbf{j},q\right)\in\tilde{\mathcal{A}}_{g_{0},n_{0}}}\frac{V_{g_{1},n_{1}}\cdots V_{g_{q},n_{q}}}{V_{g}}\leqslant\sum_{\left(\mathbf{i},\mathbf{j},q\right)\in\tilde{\mathcal{A}}_{g_{0},n_{0}}}\frac{C_{1}^{q}\sqrt{g}}{\prod_{j=1}^{q}\sqrt{g_{j}+\ensuremath{\overline{n_{j}}}}}\frac{\prod_{j=1}^{q}\left(2g_{j}+n_{j}-3\right)!}{\left(2g-3\right)!},
\]
where the indices $j$ where $(g_{j},n_{j})=(0,2)$ are omitted. We
recall Stirling's approximation which tells us that there exist constants
$1<c_{1}<c_{2}<2$ with 
\begin{equation}
c_{1}\cdot\sqrt{2\pi w}\left(\frac{w}{e}\right)^{w}<w!<c_{2}\cdot\sqrt{2\pi w}\left(\frac{w}{e}\right)^{w},\label{Sterling's approximation}
\end{equation}
for all $w\geqslant1$. We apply Stirling's approximation (\ref{Sterling's approximation})
to see that 
\begin{align}
\frac{\left(2g_{j}+n_{j}-3\right)!}{\sqrt{g_{j}+\ensuremath{\overline{n_{j}}}}} & <c_{2}\frac{\sqrt{2\pi\left(2g_{j}+n_{j}-3\right)}}{\sqrt{g_{j}+\ensuremath{\overline{n_{j}}}}}\cdot\left(\frac{2g_{j}+n_{j}-3}{e}\right)^{2g_{j}+n_{j}-3}\nonumber \\
 & <4\sqrt{\pi}\cdot\left(\frac{2g_{j}+n_{j}-3}{e}\right)^{2g_{j}+n_{j}-3}.\label{stirling 1}
\end{align}
Applying Stirling's approximation again, we see that 
\begin{align}
\frac{\sqrt{g}}{\left(2g-3\right)!} & <\frac{1}{c_{1}}\frac{\sqrt{g}}{\sqrt{2\pi\left(2g-3\right)}}\cdot\left(\frac{e}{2g-3}\right)^{2g-3}\nonumber \\
 & <C\left(\frac{e}{2g-3}\right)^{2g-3}.\label{stirling 2}
\end{align}
Thus by (\ref{stirling 1}) and (\ref{stirling 2}),
\begin{align}
 & \frac{C^{q}\sqrt{g}}{\prod_{j=1}^{q}\sqrt{g_{j}+\ensuremath{\overline{n_{j}}}}}\frac{\prod_{j=1}^{q}\left(2g_{j}+n_{j}-3\right)!}{\left(2g-3\right)!}\nonumber \\
< & C^{q}\frac{\prod_{j=1}^{q}\left(2g_{j}+n_{j}-3\right)^{2g_{j}+n_{j}-3}}{\left(2g-3\right)^{2g-3}}.\label{stirling 3}
\end{align}
Given $s$ integers $x_{i}\geqslant1$ with sum $\sum_{i}x_{i}=A$,
then the product of $x_{i}^{x_{i}}$ is maximised when $s-1$ of the
$x_{i}$ are equal to $1$, that is,
\begin{equation}
\prod_{i=1}^{s}x_{i}^{x_{i}}\leqslant(A-s+1)^{A-s+1}.\label{eq:exponentbound}
\end{equation}
Now suppose that $\max_{1\leqslant i\leqslant q}2g_{j}+n_{j}-3=R$.
We note that by condition (ii) on the indices, $R\geqslant\lfloor\frac{2g-\left(2g_{0}+n_{0}\right)-q}{q}\rfloor$
and if exactly $0\leqslant m_{1}\leqslant q$ indices satisfy $2g_{a}+n_{a}=3$
and exactly $0\leqslant m_{2}\leqslant q-m_{1}$ indices satisfy $(g_{a},n_{a})=(0,2)$
(so that they are omitted from the product in (\ref{stirling 3}))
and we separate out one of the indices satsifying $\max_{1\leqslant i\leqslant q}2g_{a}+n_{a}-3=R$
then the index condition implies that the sum of $2g_{j}+n_{j}-3$
over the remaining $q-m_{1}-m_{2}-1$ indices is equal to $2g-2g_{0}-n_{0}-q-m_{2}-R$.
Thus (\ref{eq:exponentbound}) implies that (\ref{stirling 3}) is
bounded above by 
\[
C^{q}\frac{R^{R}\left(2g-n_{0}-2g_{0}-R-2q+m_{1}+2\right)^{2g-n_{0}-2g_{0}-R-2q+m_{1}+2}}{\left(2g-3\right)^{\left(2g-3\right)}}.
\]
Then, if we write 
\begin{align}
\sum_{\left(\mathbf{i},\mathbf{j},q\right)\in\tilde{\mathcal{A}}_{g_{0},n_{0}}}\prod_{i=1}^{q}V_{g_{i},n_{i}} & =\sum_{q\geqslant1}\sum_{R=\lfloor\frac{2g-\left(2g_{0}+n_{0}\right)-q}{q}\rfloor}^{2g-4-k}\sum_{m_{1}=0}^{q}\sum_{m_{2}=0}^{q-m_{1}}\sum_{\text{\ensuremath{\substack{g_{j},n_{j}\\
2g_{j}+n_{j}\geqslant2\\
\text{exactly \ensuremath{m_{1}} of \ensuremath{2g_{i}+n_{i}=3}}\\
\text{exactly \ensuremath{m_{2}} of \ensuremath{(g_{i},n_{i})=(0,2)}}\\
\sum_{j}\left(2g_{j}+n_{j}-2\right)=2g-\left(2g_{0}+n_{0}\right)\\
\max_{1\leqslant j\leqslant q}2g_{j}+n_{j}-3=R
}
}}}\prod_{\substack{s=1,\ldots,q\\
(g_{s},n_{s})\neq(0,3),(0,2)
}
}V_{g_{s},n_{s}}.\label{eq:arranging-sum-wp-prod-bound}\\
 & =\sum_{q\geqslant1}\sum_{R=\lfloor\frac{2g+n-\left(2g_{0}+n_{0}\right)-q}{q}\rfloor}^{2g-4-k}\sum_{m_{1},m_{2}}\sum_{n_{j}}\sum_{g_{j}}\prod_{\substack{s=1,\ldots,q\\
(g_{s},n_{s})\neq(0,3),(0,2)
}
}V_{g_{s},n_{s}},\nonumber 
\end{align}
where the summation over indices which satisfy the conditions of the
previous line. Given $R$, $q$, $m_{1}$, $m_{2}$ and $\left(n_{1},\dots,n_{q}\right)$
the number of $g_{i}$ in the last summation of (\ref{eq:arranging-sum-wp-prod-bound})
is bounded above by 
\[
q{q \choose m_{1}}{q-m_{1} \choose m_{2}}{2g-2g_{0}-n_{0}-R-q-1-m_{2} \choose q-m_{1}-m_{2}-2}.
\]
This is because we have at most $q$ choices for the index $i$ of
a pair $(g_{i},n_{i})$ with $2g_{i}+n_{i}-3=R$, followed by at most
${q \choose m_{1}}$ choices for the $m_{1}$ indices of pairs $(g_{i},n_{i})$
with $2g_{i}+n_{i}-3=0$, followed by at most ${q-m_{1} \choose m_{2}}$
choices for the $m_{2}$ indices of pairs $(g_{i},n_{i})=(0,2)$,
and then the sum of $2g_{i}+n_{i}-3$ over the remaining indices equals
$2g-2g_{0}-n_{0}-R-q-m_{2}$ and $2g_{i}+n_{i}-3\geq1$, so there
are less than ${2g-2g_{0}-n_{0}-R-q-m_{2}-1 \choose q-m_{1}-m_{2}}$
choices for them. We thus bound 
\begin{align*}
 & \sum_{\left(\mathbf{i},\mathbf{j},q\right)\in\tilde{\mathcal{A}}_{g_{0},n_{0}}}\prod_{i=1}^{q}V_{g_{i},n_{i}}\\
\leqslant & \sum_{q\geqslant1}\sum_{R=\lfloor\frac{2g-\left(2g_{0}+n_{0}\right)-q}{q}\rfloor}^{2g-4-k}\sum_{n_{j}}\sum_{m_{1},m_{2}}qC^{q}{q \choose m_{1}}{q-m_{1} \choose m_{2}}{2g-2g_{0}-n_{0}-R-q-1-m_{2} \choose q-m_{1}-m_{2}-2}\\
 & \,\,\,\,\,\,\,\,\,\,\,\,\,\,\,\,\,\,\,\,\,\,\,\,\,\,\,\,\,\,\,\,\,\,\,\,\,\,\,\,\,\,\,\,\,\,\,\,\,\,\,\,\,\,\,\cdot\frac{R^{R}\left(2g-n_{0}-2g_{0}-R-2q+m_{1}+2\right)^{2g-n_{0}-2g_{0}-R-2q+m_{1}+2}}{\left(2g-3\right)^{2g-3}}\\
\leqslant & \sum_{q\geqslant1}\sum_{R=\lfloor\frac{2g-\left(2g_{0}+n_{0}\right)-q}{q}\rfloor}^{2g-4-k}\sum_{n_{j}}\sum_{m_{1},m_{2}}qC^{q}{q \choose m_{1}}{q-m_{1} \choose m_{2}}\left(2g-2g_{0}-n_{0}-R-q-m_{2}-1\right)^{q-m_{1}-m_{2}-2}\\
 & \,\,\,\,\,\,\frac{R^{R}\left(2g-n_{0}-2g_{0}-R-2q+m_{1}+2\right)^{2g-n_{0}-2g_{0}-R-2q+m_{1}+2}}{\left(2g-3\right)^{2g-3}}\\
\leqslant & \left(c\left(2g_{0}+n_{0}\right)\right)^{c\left(2g_{0}+n_{0}\right)}\sum_{q\geqslant1}\sum_{n_{j}}\sum_{R=\lfloor\frac{2g-\left(2g_{0}+n_{0}\right)-q}{q}\rfloor}^{2g-4-k}\frac{R^{R}\left(2g-n_{0}-2g_{0}-R\right)^{\left(2g-n_{0}-2g_{0}-R\right)}}{\left(2g-3\right)^{2g-3}}.
\end{align*}
For the third inequality we used that $q\leqslant n_{0}$ to bound
the $\sum_{m_{1},m_{2}}qC^{q}{q \choose m_{1}}{q-m_{1} \choose m_{2}}$
contribution by $\left(c\left(2g_{0}+n_{0}\right)\right)^{c\left(2g_{0}+n_{0}\right)}$.
Now 
\begin{align*}
 & \sum_{q\geqslant1}\sum_{n_{j}}\sum_{R=\lfloor\frac{2g-\left(2g_{0}+n_{0}\right)-q}{q}\rfloor}^{2g-4-k}\frac{R^{R}\left(2g-n_{0}-2g_{0}-R\right)^{\left(2g-2g_{0}-n_{0}-R\right)}}{\left(2g-3\right)^{\left(2g-3\right)}}\\
\leqslant & \frac{\left(ck\right)^{ck}}{g^{k+1}}\sum_{q\geqslant1}\sum_{n_{j}}1.
\end{align*}
Indeed, $x\mapsto x^{x}(a-x)^{a-x}$ has minimum at $x=\frac{a}{2}$,
so with $a=2g-(2g_{0}+n_{0})$, the summand is bounded above by its
value at either $R=2g-4-k$ or $R=\lfloor\frac{2g-\left(2g_{0}+n_{0}\right)-q}{q}\rfloor$
for $q\geq2$ and at $R=2g-4-k$ for $q=1$. After extracting the
term $R=2g-4-k$ out of the summation which is bounded by $\frac{\left(ck\right)^{ck}}{g^{k+1}}$,
the remaining contribution is bounded by the value at $R=2g-5-k$
which is at most $\frac{\left(ck\right)^{ck}}{g^{k+2}}$ and since
there are $<2g$ terms in the inner most sum, we obtain the stated
bound. Given $q\geqslant1$ there are at most 
\[
{q+n_{0} \choose q}
\]
choices for the $n_{i}$ which sum to $n_{0}$. Since $q\leqslant n_{0}$,
we see that 
\[
\sum_{\left(\mathbf{i},\mathbf{j},q\right)\in\tilde{\mathcal{A}}_{g_{0},n_{0}}}\prod_{i=1}^{q}V_{g_{i},n_{i}}\leqslant\frac{\left(ck\right)^{ck}\left(c\left(g_{0}+n_{0}\right)\right)^{c\left(g_{0}+n_{0}\right)}}{g^{k+1}},
\]
as required.
\end{proof}
We now complete the proof of Lemma \ref{lem:vol-function-bound}.
\begin{proof}[Proof of Lemma \ref{lem:vol-function-bound}.]
We first deal with the leading contribution. Since $k<g^{\frac{1}{c}}$
for $c>2$, any term of (\ref{eq:vol-terms}) can have at most one
factor $V_{g_{l},n_{l}}$ with $2g_{l}+n_{l}-3\geqslant2g-3-k$. Thus,
$\mathcal{A}_{g_{0},n_{0}}\backslash\tilde{\mathcal{A}}_{g_{0},n_{0}}$
denotes the set of $\left(\mathbf{i},\mathbf{j},q\right)$ with exactly
one such entry and so since it is unique we can re-index the entries
so that it is the $q$th one, but such that $i\mapsto\min I_{i}$
is still increasing for the first $q-1$ indices and this is a bijective
correspondence. After doing this, we also have $2g_{q}+n_{q}-3\leqslant2g-3-(2g_{0}+n_{0}-2)$
and so $2g_{q}+n_{q}-3=2g-3-r$ for some $r\in\left\{ 2g_{0}+n_{0}-2,\ldots,k\right\} $
and the remaining indices satisfy 
\begin{equation}
\sum_{i=1}^{q-1}2g_{i}+n_{i}-3=r+3-2g_{0}-n_{0}-q.\label{eq:r-sum}
\end{equation}
So, if we define for $r\in\left\{ 2g_{0}+n_{0}-2,\ldots,k\right\} $
the sets $\mathcal{A}_{g_{0},n_{0}}^{(r)}$ to be the collections
$(\textbf{i},\textbf{j},q)$ such that $\textbf{i}=(g_{1},\ldots,g_{q-1})$
and $\textbf{j}$ to be a collection of pairwise disjoint non-empty
subsets $I_{i}\subseteq\left\{ 1,\ldots,n_{0}\right\} $ with $i\mapsto\min I_{i}$
increasing for $i=1,\ldots,q-1$, $n_{i}\eqdf|I_{i}|$ and such that
(\ref{eq:r-sum}) holds, then we have a bijection between $\mathcal{A}_{g_{0},n_{0}}\backslash\tilde{\mathcal{A}}_{g_{0},n_{0}}$
and $\bigcup_{r=2g_{0}+n_{0}-2}^{k}\mathcal{A}_{g_{0},n_{0}}^{(r)}$. 

Notice that each set $\mathcal{A}_{g_{0},n_{0}}^{(r)}$ is itself
defined independently of $g$. Given $(\textbf{i},\textbf{j},q)\in\mathcal{A}_{g_{0},n_{0}}^{(r)}$,
we can set $n_{q}=n_{0}-\sum_{i=1}^{q-1}n_{i}$ and $g_{q}=g-\frac{1}{2}\left(r+n_{q}\right)$
so that $n_{q}\leqslant n_{0}\leqslant3r$, and thus by Corollary
\ref{thm:main-estimate-WP-expansion}, there is a universal constant
$c>0$ and continuous functions $\left\{ f_{t}^{g_{q},n_{q}}\right\} _{t=r}^{k}$
such that 
\[
\left|\frac{V_{g_{q},n_{q}}\left(\mathbf{x}^{(q)}\right)}{V_{g}}-\sum_{t=r}^{k}\frac{f_{t}^{g_{q},n_{q}}\left(\mathbf{x}^{(q)}\right)}{g^{t}}\right|\leqslant\frac{\left(|\mathbf{x}^{(q)}|+1\right)^{ck}\left(ckn_{0}\right)^{ck}\exp\left(\frac{1}{2}|\mathbf{x}^{(q)}|\right)}{g^{k+1}},
\]
whenever $g>(ck)^{c}$. Then, 
\begin{align}
 & \Bigg|x_{1}\dots x_{n_{0}}\frac{V_{g_{0},n_{0}}\left(\mathbf{x}\right)}{V_{g}}\sum_{\left(\mathbf{i},\mathbf{j},q\right)\in\mathcal{A}_{g_{0},n_{0}}\backslash\tilde{\mathcal{A}}_{g_{0},n_{0}}}V_{g_{1},n_{1}}\left(\mathbf{x}^{(1)}\right)\dots V_{g_{q},n_{q}}\left(\mathbf{x}^{(q)}\right)\label{eq:error-term-leading-contribution}\\
 & -x_{1}\dots x_{n_{0}}V_{g_{0},n_{0}}\left(\mathbf{x}\right)\sum_{r=2g_{0}+n_{0}-2}^{k}\sum_{\left(\mathbf{i},\mathbf{j},q\right)\in\mathcal{A}_{g_{0},n_{0}}^{(r)}}V_{g_{1},n_{1}}\left(\mathbf{x}^{(1)}\right)\dots V_{g_{q-1},n_{q-1}}\left(\mathbf{x}^{(q-1)}\right)\sum_{t=r}^{k}\frac{f_{t}^{g_{q},n_{q}}\left(\mathbf{x}^{(q)}\right)}{g^{t}}\Bigg|\nonumber \\
 & \leqslant\frac{x_{1}\dots x_{n_{0}}V_{g_{0},n_{0}}(\mathbf{x})\left(ckn_{0}\right)^{ck}\left(|\mathbf{x}|+1\right)^{ck}}{g^{k+1}}\sum_{r=2g_{0}+n_{0}-2}^{k}\sum_{\left(\mathbf{i},\mathbf{j},q\right)\in\mathcal{A}_{g_{0},n_{0}}^{(r)}}e^{\frac{1}{2}|\mathbf{x}^{(q)}|}V_{g_{1},n_{1}}\left(\mathbf{x}^{(1)}\right)\dots V_{g_{q-1},n_{q-1}}\left(\mathbf{x}^{(q-1)}\right).\nonumber 
\end{align}
Note that after interchanging summations, we see that the coefficient
of $g^{-t}$ in the expansion on the second line is
\[
h_{t}^{g_{0},n_{0}}(\mathbf{x})\eqdf x_{1}\dots x_{n_{0}}V_{g_{0},n_{0}}\left(\mathbf{x}\right)\sum_{r=2g_{0}+n_{0}-2}^{t}\sum_{\left(\mathbf{i},\mathbf{j},q\right)\in\mathcal{A}_{g_{0},n_{0}}^{(r)}}V_{g_{1},n_{1}}\left(\mathbf{x}^{(1)}\right)\dots V_{g_{q-1},n_{q-1}}\left(\mathbf{x}^{(q-1)}\right)f_{t}^{g_{q},n_{q}}\left(\mathbf{x}^{(q)}\right).
\]
Using Lemma \ref{lem:Wp-expbound}, we bound (\ref{eq:error-term-leading-contribution})
by
\[
\exp\left(|\mathbf{x}|\right)\frac{\left(ckn_{0}\right)^{ck}}{g^{k+1}}\left(|\mathbf{x}|+1\right)^{ck+n_{0}}V_{g_{0},n_{0}}\sum_{r=2g_{0}+n_{0}-2}^{k}\sum_{\left(\mathbf{i},\mathbf{j},q\right)\in\mathcal{A}_{g_{0},n_{0}}^{(r)}}\prod_{s=1}^{q-1}V_{g_{s},n_{s}},
\]
where we omit the terms $V_{0,2}$ if they appear in the product on
the right-hand side and we have bounded the product of the $x_{i}$
by $\left(|\mathbf{x}|+1\right)^{n_{0}}$. By Lemma \ref{lem:(Grushevsky's-bound,-)},
we have
\begin{equation}
V_{g_{0},n_{0}}\leqslant C^{2g_{0}+n_{0}}\left(2g_{0}+n_{0}\right)!.\label{eq:Vg0-bound}
\end{equation}
We now write
\begin{align*}
\sum_{\left(\mathbf{i},\mathbf{j},q\right)\in\mathcal{A}_{g_{0},n_{0}}^{(r)}}\prod_{s=1}^{q-1}V_{g_{s},n_{s}} & =\sum_{q\geqslant1}\sum_{\text{\ensuremath{\substack{g_{j},n_{j}\\
2g_{j}+n_{j}\geqslant2\\
\sum\left(2g_{j}+n_{j}-2\right)=r+2-2g_{0}-n_{0}
}
}}}\prod_{s=1}^{q-1}V_{g_{s},n_{s}}\\
 & =\sum_{q\geqslant1}\sum_{n_{j}}\sum_{g_{j}}\prod_{s=1}^{q-1}V_{g_{s},n_{s}},
\end{align*}
where the summations over the indices on the second line are subject
to the conditions of those on the first line, and any $V_{0,2}$ term
in the products are omitted.

If $r=2g_{0}+n_{0}-2$, then since the $2g_{j}+n_{j}\geq2$, the product
over the volumes in this case is exactly $1$ since all terms must
be $V_{0,2}$. Likewise, if $r=2g_{0}+n_{0}-1$, then all terms in
the product are $V_{0,2}$ except for one which is either $V_{0,3}=1$
or $V_{1,1}=\frac{\pi^{2}}{6}$.

For when $r\geq2g_{0}+n_{0}$, by Lemma \ref{Lemma 22}, there exist
universal constants $C,D>0$ such that for each $q,r$ and selection
of $n_{j}$, 
\begin{equation}
\sum_{g_{j}}\prod_{s=1}^{q-1}V_{g_{s},n_{s}}\leqslant C\bigg(\dfrac{D}{r-(2g_{0}+n_{0}-2)}\bigg)^{q-1}W_{r+2-2g_{0}-n_{0}},\label{eq:W-bound}
\end{equation}
where we recall the notation $W_{m}$ from (\ref{eq:W-notation}).
Since we have $2\leq r-(2g_{0}+n_{0}-2)\leq r-1$, the right-hand
side is bounded by

\[
C^{q}W_{r-1}\leqslant C^{q}(cr)^{cr},
\]
for some universal $C,c>0$ using Lemma \ref{lem:(Grushevsky's-bound,-)}.
Combining these cases, there exists a universal $c>0$ such that
\[
\sum_{\left(\mathbf{i},\mathbf{j},q\right)\in\mathcal{A}_{g_{0},n_{0}}^{(r)}}\prod_{s=1}^{q-1}V_{g_{s},n_{s}}\leqslant\sum_{q\geqslant1}\sum_{n_{j}}C^{q}(cr)^{cr}\leqslant\sum_{q\geqslant1}C^{q}(cr)^{cr}{q+n_{0} \choose q}\leqslant(cn_{0})^{cn_{0}}(cr)^{cr},
\]
using Stirling's bounds on factorials and that $q\leq n_{0}$. In
total, we find that (\ref{eq:error-term-leading-contribution}) is
bounded by
\begin{align*}
 & \exp\left(|\mathbf{x}|\right)\frac{\left(ckn_{0}\right)^{ck}}{g^{k+1}}\left(|\mathbf{x}|+1\right)^{ck+n_{0}}\left(c\left(g_{0}+n_{0}\right)\right)^{c\left(g_{0}+n_{0}\right)}\sum_{r=2g_{0}+n_{0}-2}^{k}(cr)^{cr}\\
 & \leqslant\exp\left(|\mathbf{x}|\right)\frac{\left(ckn_{0}\right)^{ck}}{g^{k+1}}\left(|\mathbf{x}|+1\right)^{ck+n_{0}}\left(c\left(g_{0}+n_{0}\right)\right)^{c\left(g_{0}+n_{0}\right)},
\end{align*}
for some universal $c>0$ that changes from line to line. 

On the other hand, for indices $\left(\mathbf{i},\mathbf{j},q\right)\in\tilde{\mathcal{A}}_{g_{0},n_{0}}$,
we have

\begin{align}
 & x_{1}\dots x_{n_{0}}\frac{V_{g_{0},n_{0}}\left(\mathbf{x}\right)}{V_{g}}\sum_{\left(\mathbf{i},\mathbf{j},q\right)\in\tilde{\mathcal{A}}_{g_{0},n_{0}}}V_{g_{1},n_{1}}\left(\mathbf{x}^{(1)}\right)\dots V_{g_{q},n_{q}}\left(\mathbf{x}^{(q)}\right)\nonumber \\
\leqslant & \left(|\mathbf{x}|+1\right)^{n_{0}}\exp\left(|\mathbf{x}|\right)\frac{V_{g_{0},n_{0}}}{V_{g}}\sum_{\left(\mathbf{i},\mathbf{j},q\right)\in\tilde{\mathcal{A}}_{g_{0},n_{0}}}V_{g_{1},n_{1}}\dots V_{g_{q},n_{q}}\nonumber \\
\leqslant & \frac{\left(ck\right)^{ck}\left(c\left(g_{0}+n_{0}\right)\right)^{c\left(g_{0}+n_{0}\right)}\exp\left(|\mathbf{x}|\right)\left(|\mathbf{x}|+1\right)^{n_{0}}}{g^{k+1}}.\label{eq:error-bound-second}
\end{align}
by applying Lemmas \ref{lem:Wp-expbound}, \ref{lem:WP-sum-product-bound}
and \ref{lem:(Grushevsky's-bound,-)}, where on the second line we
omit any $V_{0,2}$ terms from the product. Then the claim follows
from (\ref{eq:error-term-leading-contribution}) and (\ref{eq:error-bound-second}).
\end{proof}
We can now prove Proposition \ref{prop:effective-nonsimple-expansion}.
\begin{proof}[Proof of Proposition \ref{prop:effective-nonsimple-expansion}.]
Let $L\geqslant0$ and $F_{L}$ be supported in $[0,L]$. Lemma \ref{lem:integration-lemma}
says that

\begin{align*}
 & \mathbb{E}_{g}\left[\sum_{\gamma\in\mathcal{P}^{\textup{n-simp}}\left(X\right)}F_{L}\left(\ell_{\gamma}\left(X\right)\right)\right]\\
= & \sum_{\substack{\left(g_{0},n_{0}\right)\\
2g_{0}-2+n_{0}\geqslant1
}
}\frac{1}{n_{0}!}\int_{\mathbb{R}_{\geqslant0}^{n_{0}}}\frac{1}{V_{g_{0},n_{0}}\left({\bf x}\right)}\int_{\mathcal{M}_{g_{0},n_{0}}(\mathbf{x})}\sum_{\alpha\textup{ filling }S_{g_{0},n_{0}}}F_{L}\left(\ell_{\alpha}\left(Y\right)\right)\mathrm{d}\textup{Vol}_{\mathrm{WP}}\left(Y\right)\phi_{g}^{\left(g_{0},n_{0}\right)}(\mathbf{x})\mathrm{d}\mathbf{x}.
\end{align*}
 We define
\[
f_{F_{L},j}^{\textup{n-simp}}\eqdf\sum_{\substack{\left(g_{0},n_{0}\right)\\
2g_{0}-2+n_{0}\geqslant1
}
}\frac{1}{n_{0}!}\int_{\mathbb{R}_{\geqslant0}^{n_{0}}}\frac{1}{V_{g_{0},n_{0}}\left({\bf x}\right)}\int_{\mathcal{M}_{g_{0},n_{0}}(\mathbf{x})}\sum_{\alpha\textup{ filling }S_{g_{0},n_{0}}}F_{L}\left(\ell_{\alpha}\left(Y\right)\right)\mathrm{d}\textup{Vol}_{\mathrm{WP}}\left(Y\right)h_{j}^{g_{0},n_{0}}\left(\mathbf{x}\right)\mathrm{d}\mathbf{x},
\]
where $h_{j}^{g_{0},n_{0}}\left(\mathbf{x}\right)$ is given by Lemma
\ref{lem:vol-function-bound}. By Lemma \ref{lem:filling-bound},
there is a $C>0$ such that 
\[
\sum_{\alpha\textup{ filling }S_{g_{0},n_{0}}}F_{L}\left(\ell_{\alpha}(Y)\right)=0,
\]
for any $Y\in\mathcal{M}_{g_{0},n_{0}}(\mathbf{x})$ as soon as $2g_{0}+n_{0}>CL$.
For $g>CL$, it follows from Lemma \ref{lem:vol-function-bound} that
\begin{align}
 & \left|\mathbb{E}_{g}\left[\sum_{\gamma\in\mathcal{P}^{\textup{n-simp}}\left(X\right)}F_{L}\left(\ell_{\gamma}\left(X\right)\right)\right]-\sum_{j=1}^{k}\frac{f_{F_{L},j}^{\textup{n-simp}}}{g^{j}}\right|\label{eq:non-simple-expans-est}\\
= & \Bigg|\sum_{\substack{\left(g_{0},n_{0}\right)\\
1\leqslant2g_{0}-2+n_{0}\leqslant CL
}
}\frac{1}{n_{0}!}\int_{\mathbb{R}_{\geqslant0}^{n_{0}}}\frac{1}{V_{g_{0},n_{0}}\left({\bf x}\right)}\int_{\mathcal{M}_{g_{0},n_{0}}(\mathbf{x})}\sum_{\alpha\textup{ filling }S_{g_{0},n_{0}}}F_{L}\left(\ell_{\alpha}(Y)\right)\mathrm{d}\textup{Vol}_{\mathrm{WP}}\left(Y\right)\phi_{g}^{\left(g_{0},n_{0}\right)}(\mathbf{x})\mathrm{d}\mathbf{x}\nonumber \\
- & \sum_{j=1}^{k}\frac{1}{g^{j}}\sum_{\substack{\left(g_{0},n_{0}\right)\\
1\leqslant2g_{0}-2+n_{0}\leqslant CL
}
}\frac{1}{n_{0}!}\nonumber \\
 & \,\,\,\,\,\,\,\,\,\,\,\,\,\,\,\,\,\,\,\,\,\,\,\,\,\,\int_{\mathbb{R}_{\geqslant0}^{n_{0}}}\frac{1}{V_{g_{0},n_{0}}\left({\bf x}\right)}\int_{\mathcal{M}_{g_{0},n_{0}}(\mathbf{x})}\sum_{\alpha\textup{ filling }S_{g_{0},n_{0}}}F_{L}\left(\ell_{\alpha}(Y)\right)\mathrm{d}\textup{Vol}_{\mathrm{WP}}\left(Y\right)h_{j}^{g_{0},n_{0}}\left(\mathbf{x}\right)\mathrm{d}\mathbf{x}\Bigg|\nonumber \\
\leqslant & \frac{\left(ck\right)^{ck}}{g^{k+1}}\left(cL\right)^{cL}\sum_{\substack{\left(g_{0},n_{0}\right)\\
1\leqslant2g_{0}-2+n_{0}\leqslant CL
}
}\frac{1}{n_{0}!}\int_{\mathbb{R}_{\geqslant0}^{n_{0}}}\frac{1}{V_{g_{0},n_{0}}\left({\bf x}\right)}\nonumber \\
 & \,\,\,\,\,\,\,\,\,\,\,\,\,\,\,\,\,\,\,\,\,\,\,\,\,\,\,\,\,\,\,\,\,\,\,\,\int_{\mathcal{M}_{g_{0},n_{0}}(\mathbf{x})}\sum_{\alpha\textup{ filling }S_{g_{0},n_{0}}}|F_{L}\left(\ell_{\alpha}(Y)\right)|\mathrm{d}\textup{Vol}_{\mathrm{WP}}\left(Y\right)\left(1+\left|\mathbf{x}\right|\right)^{ck+n_{0}}\exp\left(|\mathbf{x}|\right)\mathrm{d}\mathbf{x}.\nonumber 
\end{align}
Now since, by Lemma \ref{lem:counting-filling}, there are at most
$C\left(g+n\right)e^{L}$ filling geodesics of length at most $L$
on a hyperbolic surface of signature $\left(g,n\right)$, and if $\gamma$
fills a surface $X\in\mathcal{M}_{g,n}\left(\mathbf{x}\right)$ then
$\ell_{\gamma}\left(X\right)\geqslant\frac{|\mathbf{x}|}{2}$, e.g.
\cite[Proposition 7]{Wu.Xu2022}, it follows that
\begin{align*}
 & \frac{1}{V_{g_{0},n_{0}}\left({\bf x}\right)}\int_{\mathcal{M}_{g_{0},n_{0}}(\mathbf{x})}\sum_{\alpha\textup{ filling }S_{g_{0},n_{0}}}F_{L}\left(\ell_{\alpha}\left(X\right)\right)\mathrm{d}\textup{Vol}_{\mathrm{WP}}\left(Y\right)\\
\leqslant & C\left(2g_{0}+n_{0}\right)\|F_{L}\|_{\infty}e^{L}\ind_{\left[\sum x_{i}\leqslant2L\right]},
\end{align*}
so that (\ref{eq:non-simple-expans-est}) is bounded by 
\begin{align*}
 & \frac{\left(ck\right)^{ck}}{g^{k+1}}\left(cL\right)^{cL}e^{L}\|F_{L}\|_{\infty}\sum_{\substack{\left(g_{0},n_{0}\right)\\
1\leqslant2g_{0}-2+n_{0}\leqslant CL
}
}\frac{\left(2g_{0}+n_{0}\right)}{n_{0}!}\int_{\mathbb{R}_{\geqslant0}^{n_{0}}}\left(|\mathbf{x}|+1\right)^{ck+n_{0}}\exp\left(|\mathbf{x}|\right)e^{L}\ind_{\left[\sum x_{i}\leqslant2L\right]}\mathrm{d}\mathbf{x}\\
\leqslant & \frac{\left(ck\right)^{ck}}{g^{k+1}}\left(cL\right)^{c(L+k)}\|F_{L}\|_{\infty},
\end{align*}
and the result follows.
\end{proof}

\subsection{Proof of Theorem \ref{thm:main-technical-result}}

We conclude this section with the proof of Theorem \ref{thm:main-technical-result}.
Recall the construction of $f$ from the start of Section \ref{sec:Geometric-estimates}
so that $\check{f}$ is supported in $[-1,1]$. 
\begin{proof}[Proof of Theorem \ref{thm:main-technical-result}.]
By Selberg's trace formula, c.f. Theorem \ref{thm:Selberg-TF}, we
have
\begin{align}
\mathbb{E}_{g}\left[\tr\left(\left(f\left(\sqrt{\Delta-\tfrac{1}{4}}\right)\right)^{t}\right)\right] & =\left(1-\frac{1}{g}\right)\int_{\frac{1}{4}}^{\infty}f\left(\sqrt{r-\tfrac{1}{4}}\right)^{t}\tanh\left(\pi\sqrt{r-\tfrac{1}{4}}\right)\mathrm{d}r\label{eq:trace-formula-pplied}\\
 & \,\,\,\,+\frac{1}{g}\mathbb{E}_{g}\left[\sum_{\gamma\in\mathcal{P}(X)}\sum_{k=1}^{\infty}\frac{\ell_{\gamma}\left(X\right)}{2\sinh\left(\frac{k\ell_{\gamma}(x)}{2}\right)}\check{f}^{*t}\left(k\ell_{\gamma}\left(X\right)\right)\right].\nonumber 
\end{align}
We separate 
\begin{align*}
\mathbb{E}_{g}\left[\sum_{\gamma\in\mathcal{P}(X)}\sum_{k=1}^{\infty}\frac{\ell_{\gamma}\left(X\right)}{2\sinh\left(\frac{k\ell_{\gamma}(x)}{2}\right)}\check{f}^{*t}\left(k\ell_{\gamma}\left(X\right)\right)\right] & =\mathbb{E}_{g}\left[\sum_{\substack{\gamma\in\mathcal{P}(X)\\
\gamma\ \textup{simple}
}
}\sum_{k=1}^{\infty}\frac{\ell_{\gamma}\left(X\right)}{2\sinh\left(\frac{k\ell_{\gamma}(x)}{2}\right)}\check{f}^{*t}\left(k\ell_{\gamma}\left(X\right)\right)\right]\\
 & +\mathbb{E}_{g}\left[\sum_{\substack{\gamma\in\mathcal{P}(X)\\
\gamma\ \textup{non-simple}
}
}\sum_{k=1}^{\infty}\frac{\ell_{\gamma}\left(X\right)}{2\sinh\left(\frac{k\ell_{\gamma}(x)}{2}\right)}\check{f}^{*t}\left(k\ell_{\gamma}\left(X\right)\right)\right].
\end{align*}
Since any non-simple geodesic has length at least $2\text{arcsinh}1$,
for any non-simple $\gamma\in\mathcal{P}\left(X\right)$ we have $\check{f}^{*t}\left(k\ell_{\gamma}\left(X\right)\right)=0$
for any $k\geqslant\frac{t}{2\text{arcsinh}1}$. Therefore 
\[
\mathbb{E}_{g}\left[\sum_{\substack{\gamma\in\mathcal{P}(X)\\
\gamma\ \textup{non-simple}
}
}\sum_{k=1}^{\infty}\frac{\ell}{2\sinh\left(\frac{k\ell}{2}\right)}\check{f}^{*t}\left(k\ell_{\gamma}\left(X\right)\right)\right]=\mathbb{E}_{g}\left[\sum_{\substack{\gamma\in\mathcal{P}(X)\\
\gamma\ \textup{non-simple}
}
}R_{t}\left(\ell_{\gamma}\left(X\right)\right)\right]
\]
where 
\[
R_{t}\left(\ell\right)\eqdf\sum_{k=1}^{\lceil\frac{t}{2\text{arcsinh}1}\rceil}\frac{\ell}{2\sinh\left(\frac{k\ell}{2}\right)}\check{f}^{*t}\left(k\ell\right),
\]
which is smooth and supported in $\left[0,t\right]$. Therefore by
Proposition \ref{prop:effective-nonsimple-expansion},

\begin{align*}
 & \left|\mathbb{E}_{g}\left[\sum_{\substack{\gamma\in\mathcal{P}(X)\\
\gamma\ \textup{non-simple}
}
}R_{t}\left(\ell_{\gamma}\left(X\right)\right)\right]-\sum_{j=1}^{q}\frac{f_{j,R_{t}}^{\textup{n-simp}}}{g^{j}}\right|\\
 & \leqslant\frac{t^{c(t+q)}\left(cq\right)^{cq}}{g^{q+1}}\|R_{t}\|_{\infty}.
\end{align*}
Now 
\[
\|R_{t}\|_{\infty}\leqslant\|\check{f}^{*t}\|_{\infty}\cdot\sum_{k=1}^{\lceil\frac{t}{2\text{arcsinh}1}\rceil}\frac{\ell}{2\sinh\left(\frac{k\ell}{2}\right)}\leqslant ct\|\check{f}^{*t}\|_{\infty}\leqslant ct^{t}\|\check{f}\|_{\infty}^{t},
\]
where in the last inequality we applied Young's convolution inequality.
In total we see that there is a $c>0$ 
\begin{equation}
\left|\mathbb{E}_{g}\left[\sum_{\substack{\gamma\in\mathcal{P}(X)\\
\gamma\ \textup{non-simple}
}
}R_{t}\left(\ell_{\gamma}\left(X\right)\right)\right]-\sum_{j=1}^{q}\frac{f_{j,R_{t}}^{\textup{n-simp}}}{g^{j}}\right|\leqslant\frac{t^{c(t+q)}\left(cq\right)^{cq}}{g^{q+1}}.\label{eq:non-simple-cont-applied}
\end{equation}
Now for the contribution of the simple curves, we can apply Proposition
\ref{prop:effective-simple-expansion} to the function 
\[
G_{t}\left(\ell\right)\eqdf\sum_{k=1}^{\infty}\frac{\ell}{2\sinh\left(\frac{k\ell}{2}\right)}\check{f}^{*t}\left(k\ell\right).
\]
Since $\ell\mapsto\ell G_{t}\left(\ell\right)$ is continuous and
supported in $[0,t]$, we apply Proposition \ref{prop:effective-simple-expansion}
to find a $c>0$ with
\begin{align}
 & \left|\mathbb{E}_{g}\left[\sum_{\substack{\gamma\in\mathcal{P}(X)\\
\gamma\ \textup{simple}
}
}G\left(\ell_{\gamma}\left(X\right)\right)\right]-\int_{0}^{\infty}\sum_{k=1}^{\infty}2\frac{\sinh\left(\frac{\ell}{2}\right)^{2}}{\sinh\left(\frac{k\ell}{2}\right)}\check{f}^{*t}(k\ell)\mathrm{d}\ell-\sum_{j=1}^{q}\int_{0}^{\infty}G_{t}\left(\ell\right)\frac{f_{j}^{\textup{simp}}}{g^{j}}(\ell)\mathrm{d}\ell\right|\label{eq:simple-cont-applied}\\
\leqslant & \frac{t^{c(t+q)}\left(cq\right)^{cq}}{g^{q+1}}\int_{0}^{t}\ell G{}_{t}\left(\ell\right)\exp\left(\ell\right)\mathrm{d}\ell\leqslant\frac{t^{c(t+q)}\left(cq\right)^{cq}}{g^{q+1}},\nonumber 
\end{align}
where on the last line we used that $\sum_{k=1}^{\infty}\frac{\ell^{2}}{2\sinh\left(\frac{k\ell}{2}\right)}\ind_{k\ell\leqslant t}$
is bounded by $t^{2}$ and that $\|\check{f}^{*t}\|_{\infty}\leqslant t^{t}\|\check{f}\|_{\infty}^{t}$
. 

Defining 
\begin{align*}
a_{0}^{t} & \eqdf\int_{\frac{1}{4}}^{\infty}f\left(\sqrt{r-\tfrac{1}{4}}\right)^{t}\tanh\left(\sqrt{r-\tfrac{1}{4}}\right)\mathrm{d}r.\\
a_{1}^{t} & \eqdf\int_{0}^{\infty}\sum_{k=1}^{\infty}2\frac{\sinh\left(\frac{\ell}{2}\right)^{2}}{\sinh\left(\frac{k\ell}{2}\right)}\check{f}^{*t}(k\ell)\mathrm{d}\ell-\int_{\frac{1}{4}}^{\infty}f\left(\sqrt{r-\tfrac{1}{4}}\right)^{t}\tanh\left(\sqrt{r-\tfrac{1}{4}}\right)\mathrm{d}r,
\end{align*}
 and for $j\geqslant2$
\[
a_{j}^{t}\eqdf\int_{0}^{\infty}G_{t}\left(\ell\right)f_{j-1}^{\textup{simp}}\left(\ell\right)\mathrm{d}\ell+f_{j-1,R_{t}}^{\textup{n-simp}},
\]
by (\ref{eq:trace-formula-pplied}), (\ref{eq:non-simple-cont-applied})
and (\ref{eq:simple-cont-applied}) there is a $c>0$ such that
\[
\left|\mathbb{E}_{g}\left[\textup{tr}\left(\left(f\left(\sqrt{\Delta-\tfrac{1}{4}}\right)\right)^{t}\right)\right]-\sum_{i=0}^{q}\frac{a_{i}^{t}}{g^{i}}\right|\leqslant\frac{\left(cqt\right)^{c(q+t)}}{g^{q+1}},
\]
and the conclusion follows for $q\geqslant t$ and $g\geqslant c'q^{c'}$.
\end{proof}

\section{Proof of Theorem \ref{thm:Main-Thm} \label{sec:Proof-of-Theorem}}

The purpose of this section is to prove Theorem \ref{thm:Main-Thm}.
Our approach is inspired by \cite{Ch.Ga.Tr.va2024,Ma.Pu.vH2025}.

\subsection{Master inequalities}

In this section we aim to adapt the arguments of \cite[Section 3.1]{Ma.Pu.vH2025}
to our setting.
\begin{lem}
\label{lem:master-ineqaulity-polynomials}Let $\tilde{h}(x)=\sum_{t=0}^{q-1}s_{t}x^{t}$
be a real-valued polynomial of degree at most $q-1$ and let $h(x)\eqdf x\tilde{h}(x)$.
There is a constant $c>0$ independent of $h$ such that 
\begin{align*}
 & \left|\mathbb{E}_{g}\left[\textup{tr}\left[h\left(f\left(\sqrt{\Delta-\tfrac{1}{4}}\right)\right)\right]\right]-\int_{\frac{1}{4}}^{\infty}h\left(f\left(\sqrt{r-\tfrac{1}{4}}\right)\right)\tanh\left(\pi\sqrt{r-\tfrac{1}{4}}\right)\mathrm{d}r-\frac{\nu_{1}(\tilde{h})}{g}\right|\\
\leqslant & \frac{cq^{c}}{g^{2}}\|\tilde{h}\|_{\left[-f\left(\frac{i}{2}\right),f\left(\frac{i}{2}\right)\right]},
\end{align*}
for all $g\geqslant2$, where 
\[
\nu_{1}\left(\tilde{h}\right)\eqdf\sum_{j=0}^{q-1}s_{j}\int_{0}^{\infty}\sum_{k=1}^{\infty}2\frac{\sinh\left(\frac{r}{2}\right)^{2}}{\sinh\left(\frac{kr}{2}\right)}\check{f}^{*j+1}(kr)\mathrm{d}r-\int_{\frac{1}{4}}^{\infty}h\left(f\left(\sqrt{r-\tfrac{1}{4}}\right)\right)\tanh\left(\pi\sqrt{r-\tfrac{1}{4}}\right)\mathrm{d}r.
\]
\end{lem}

\begin{proof}
By Theorem \ref{thm:main-technical-result}, there is a $c>0$ independent
of $h$ such that 
\[
\left|\mathbb{E}_{g}\left[\textup{tr}\left[h\left(f\left(\sqrt{\Delta-\tfrac{1}{4}}\right)\right)\right]\right]-f_{h}\left(\frac{1}{g}\right)\right|\leqslant\frac{\left(cq\right)^{cq}}{g^{q+1}}\sum_{t=0}^{q-1}|s_{t}|,
\]
for all $g>cq^{c}$. Here 
\[
f_{h}\left(x\right)=\sum_{t=0}^{q-1}\sum_{i=0}^{q}s_{t}a_{i}^{t+1}x^{i}.
\]
As noted in \cite[Section 3.2]{Ma.Pu.vH2025}, a classical result
of V. Markov \cite[Section 2.6, Eq. (9)]{Ti63} implies that 
\[
\sum_{t=0}^{q-1}|s_{t}|\leqslant e^{\frac{q}{f\left(\frac{i}{2}\right)}}\|\tilde{h}\|_{\left[-f\left(\frac{i}{2}\right),f\left(\frac{i}{2}\right)\right]}.
\]
Then 
\begin{equation}
\left|\mathbb{E}_{g}\left[\textup{tr}\left[h\left(f\left(\sqrt{\Delta-\tfrac{1}{4}}\right)\right)\right]\right]-f_{h}\left(\frac{1}{g}\right)\right|\leqslant e^{\frac{q}{f\left(\frac{i}{2}\right)}}\frac{\left(cq\right)^{cq}}{g^{q+1}}\|\tilde{h}\|_{\left[-f\left(\frac{i}{2}\right),f\left(\frac{i}{2}\right)\right]}\leqslant\frac{1}{g^{2}}\|\tilde{h}\|_{\left[-f\left(\frac{i}{2}\right),f\left(\frac{i}{2}\right)\right]}\label{eq:initial}
\end{equation}
for all $g>c'q^{c'}$ for some $c'>c$. It follows that
\begin{align}
\text{\ensuremath{\left|f_{h}\left(\frac{1}{g}\right)\right|}\ensuremath{\ensuremath{\leqslant}}} & \left|\mathbb{E}_{g}\left[\textup{tr}\left[h\left(f\left(\sqrt{\Delta-\tfrac{1}{4}}\right)\right)\right]\right]\right|+\frac{1}{g^{2}}\|\tilde{h}\|_{\left[-f\left(\frac{i}{2}\right),f\left(\frac{i}{2}\right)\right]},\label{eq:apriori-triangle-ineq}
\end{align}
for all $g>c'q^{c'}$. Since $h(x)=x\tilde{h}(x)$,
\begin{equation}
h(x)\leqslant\|\tilde{h}\|_{\left[-f\left(\frac{i}{2}\right),f\left(\frac{i}{2}\right)\right]}x\label{eq:h-bound}
\end{equation}
 for $0\leqslant x\leqslant f\left(\frac{i}{2}\right)$. By the non-negativity
of $f$ on $\mathbb{R}\cup i\mathbb{R}$ and functional calculus,
(\ref{eq:h-bound}) implies 
\begin{align*}
\mathbb{E}_{g}\left[\textup{tr}\left[h\left(f\left(\sqrt{\Delta-\tfrac{1}{4}}\right)\right)\right]\right] & =\mathbb{E}_{g}\left[\frac{1}{g}\sum_{j}h\left(f\left(\sqrt{\lambda_{j}-\tfrac{1}{4}}\right)\right)\right]\\
 & \leqslant\frac{\|\tilde{h}\|_{\left[-f\left(\frac{i}{2}\right),f\left(\frac{i}{2}\right)\right]}}{g}\mathbb{E}_{g}\left[\sum_{j}f\left(\sqrt{\lambda_{j}-\tfrac{1}{4}}\right)\right].
\end{align*}
By for example \cite[Sections 6 and 7]{Wu.Xu2022}\footnote{This follows as a special case with $T$ fixed from Propositions 24,
28, 29 and Theorem 36 of \cite{Wu.Xu2022} which we cite for convenience
although the bound we need is much less difficult. We remark that
the difficulty in the cited results from \cite{Wu.Xu2022} lies in
proving estimates which remain effective for $T\sim4\log g$. Proving
just the bound we need for $T$ fixed is much easier and follows readily
from the work of Mirzakhani and Petri \cite{Mi.Pe19}. }, there is a constant $C>0$ such that
\[
\frac{\mathbb{E}_{g}\left[\sum_{j}f\left(\sqrt{\lambda_{j}-\tfrac{1}{4}}\right)\right]}{g}\leqslant C.
\]
We see that
\begin{equation}
\mathbb{E}_{g}\left[\textup{tr}\ \left[h\left(f\left(\sqrt{\Delta-\tfrac{1}{4}}\right)\right)\right]\right]\leqslant C\|\tilde{h}\|_{\left[-f\left(\frac{i}{2}\right),f\left(\frac{i}{2}\right)\right]}.\label{eq:B-s-bound}
\end{equation}
Then by (\ref{eq:apriori-triangle-ineq}) and (\ref{eq:B-s-bound})
there is a $C>0$ with
\begin{equation}
\left|f_{h}\left(\frac{1}{g}\right)\right|\leqslant C\|\tilde{h}\|_{\left[-f\left(\frac{i}{2}\right),f\left(\frac{i}{2}\right)\right]}\label{eq:apriori-bound}
\end{equation}
for all $g>c'q^{c'}$. Here (\ref{eq:apriori-bound}) plays the role
of the a priori bound \cite[Lemma 6.4]{Ch.Ga.Tr.va2024}. By \cite[Lemma 4.2]{Magee-delaSalle},
which is a variant of the Markov brother's inequality, there is a
$C>0$ independent of $h$ such that
\begin{align}
\|f_{h}'\|_{\left[0,\frac{1}{2c'q^{c'}}\right]} & \leqslant Cq^{2c'}\|\tilde{h}\|_{\left[-f\left(\frac{i}{2}\right),f\left(\frac{i}{2}\right)\right]},\label{eq:nu-1}\\
\|f_{h}''\|_{\left[0,\frac{1}{2c'q^{c'}}\right]} & \leqslant Cq^{4c'}\|\tilde{h}\|_{\left[-f\left(\frac{i}{2}\right),f\left(\frac{i}{2}\right)\right]}.\label{eq:f''-bound}
\end{align}
Note that $f_{h}(0)=\sum_{t=0}^{q-1}s_{t}a_{0}^{t+1}=\int_{\frac{1}{4}}^{\infty}h\left(f\left(\sqrt{r-\tfrac{1}{4}}\right)\right)r\tanh\left(\pi\sqrt{r-\tfrac{1}{4}}\right)\mathrm{d}r$.
So, by Taylor expanding $f_{h}$, we obtain
\begin{align}
 & \left|f_{h}\left(\frac{1}{g}\right)-\int_{\frac{1}{4}}^{\infty}h\left(f\left(\sqrt{r-\tfrac{1}{4}}\right)\right)\tanh\left(\pi\sqrt{r-\tfrac{1}{4}}\right)\mathrm{d}r-\frac{1}{g}\nu_{1}\left(\tilde{h}\right)\right|\label{eq:taylor}\\
\leqslant & \frac{1}{g^{2}}\|f_{h}''\|_{[0,\frac{1}{g}]}\leqslant\frac{Cq^{4c'}}{g^{2}}\|\tilde{h}\|_{\left[-f\left(\frac{i}{2}\right),f\left(\frac{i}{2}\right)\right]}.\nonumber 
\end{align}
for $g>2c'q^{c'}$, where 
\begin{align}
\nu_{1}\left(\tilde{h}\right) & \eqdf\underbrace{\sum_{j=0}^{q-1}s_{j}\int_{0}^{\infty}\sum_{k=1}^{\infty}2\frac{\sinh\left(\frac{r}{2}\right)^{2}}{\sinh\left(\frac{kr}{2}\right)}\check{f}^{*j+1}(kr)\mathrm{d}r-\int_{\frac{1}{4}}^{\infty}h\left(f\left(\sqrt{r-\tfrac{1}{4}}\right)\right)\tanh\left(\pi\sqrt{r-\tfrac{1}{4}}\right)\mathrm{d}r}_{=f_{h}'(0)}.\label{eq:nu_1-def-deriv}
\end{align}
The bound for $g>2c'q^{c'}$ thus holds by the triangle inequality
and both (\ref{eq:initial}) and (\ref{eq:taylor}). Conversely, for
$2c'q^{c'}>g$ we have 
\begin{align*}
 & \left|\mathbb{E}_{g}\left[\tr\:h\left(f\left(\sqrt{\Delta-\tfrac{1}{4}}\right)\right)\right]-\int_{\frac{1}{4}}^{\infty}h\left(f\left(\sqrt{r-\tfrac{1}{4}}\right)\right)\tanh\left(\pi\sqrt{r-\tfrac{1}{4}}\right)\mathrm{d}r-\frac{\nu_{1}(\tilde{h})}{g}\right|\\
 & \leqslant2C\|\tilde{h}\|_{\left[-f\left(\frac{i}{2}\right),f\left(\frac{i}{2}\right)\right]}+\frac{1}{g}\nu_{1}\left(\tilde{h}\right)\leqslant\frac{C'q^{4c'}}{g^{2}}\|\tilde{h}\|_{\left[-f\left(\frac{i}{2}\right),f\left(\frac{i}{2}\right)\right]},
\end{align*}
where for the first inequality we applied the triangle inequality
together with (\ref{eq:B-s-bound}) to bound the first term and (\ref{eq:apriori-bound})
for the second term since it is precisely $f_{h}(0)$, and for the
second inequality we used (\ref{eq:nu_1-def-deriv}), (\ref{eq:nu-1})
and the fact that $\frac{2c'q^{c^{'}}}{g}>1$.
\end{proof}
By some minor adaptations to the proof of \cite[Theorem 7.1]{Ch.Ga.Tr.va2024},
Lemma \ref{lem:master-ineqaulity-polynomials} leads to the following
smooth master inequality.
\begin{prop}
\label{prop:smooth-master}The linear functional $\nu_{1}$ extends
to a compactly supported distribution and there exist $C,m>0$ such
that for any $\tilde{h}\in C^{\infty}\left(\mathbb{R}\right)$ and
any $g\geqslant2$, defining $h(x)=x\tilde{h}(x)$, we have
\begin{align*}
 & \left|\mathbb{E}_{g}\left[\tr\:h\left(f\left(\sqrt{\Delta-\tfrac{1}{4}}\right)\right)\right]-\int_{\frac{1}{4}}^{\infty}h\left(f\left(\sqrt{r-\tfrac{1}{4}}\right)\right)\tanh\left(\pi\sqrt{r-\tfrac{1}{4}}\right)\mathrm{d}r-\frac{\nu_{1}\left(\tilde{h}\right)}{g}\right|\\
 & \leqslant\frac{C}{g^{2}}\left(\|w^{(m)}\|_{[0,2\pi]}+\|\tilde{h}\|_{\left[-f\left(\frac{i}{2}\right),f\left(\frac{i}{2}\right)\right]}\right),
\end{align*}
where $w(\theta)\eqdf\tilde{h}\left(f\left(\frac{i}{2}\right)\cos\theta\right)$.
\end{prop}

\begin{proof}
The fact that $\nu_{1}$ extends to a compactly supported distribution
is immediate from the fact that for any polynomial $\tilde{h}$ of
degree $q\geqslant1$, we have from (\ref{eq:nu_1-def-deriv}) and
(\ref{eq:nu-1}) that
\[
\left|\nu_{1}(\tilde{h})\right|=\left|f_{h}'(0)\right|\leqslant C(q+1)^{2c'}\|\tilde{h}\|_{\left[-f\left(\frac{i}{2}\right),f\left(\frac{i}{2}\right)\right]}\leqslant C'q^{2c'}\|\tilde{h}\|_{\left[-f\left(\frac{i}{2}\right),f\left(\frac{i}{2}\right)\right]},
\]
and so \cite[Lemma 4.7]{Ch.Ga.Tr.va2024} provides the extension and
tells us that the support is contained in $\left[-f\left(\frac{i}{2}\right),f\left(\frac{i}{2}\right)\right]$.
Let $\tilde{h}_{q}$ be a sequence of polynomials of degree $q-1$
that converge to a $\tilde{h}\in C^{\infty}(\mathbb{R})$ in the $C^{c+1}$
norm where $c>0$ is as in Lemma \ref{lem:master-ineqaulity-polynomials}.
We want to apply Lemma \ref{lem:master-ineqaulity-polynomials} to
$\tilde{h}_{q}$. Each $\tilde{h}_{q}$ can be uniquely expressed
as
\[
\tilde{h}_{q}(x)=\sum_{k=0}^{q-1}a_{k}T_{k}\left(\frac{x}{f\left(\frac{i}{2}\right)}\right)
\]
for $x\in\left[-f\left(\frac{i}{2}\right),f\left(\frac{i}{2}\right)\right]$,
where $T_{k}$ is the Chebyshev polynomial of degree $k$ defined
by $T_{k}(\cos(\theta))=\cos(k\theta)$. Then 
\begin{equation}
h_{q}(x)=\sum_{k=0}^{q-1}a_{k}T_{k}\left(\frac{x}{f\left(\frac{i}{2}\right)}\right)x.\label{eq:chebyshev-expansion}
\end{equation}
By applying Lemma \ref{lem:master-ineqaulity-polynomials} individually
to each term $a_{k}T_{k}\left(\frac{x}{f\left(\frac{i}{2}\right)}\right)$
in the expansion of $\tilde{h}_{q}$,
\begin{align}
 & \left|\mathbb{E}_{g}\left[\tr\ h_{q}\left(f\left(\sqrt{\Delta-\tfrac{1}{4}}\right)\right)\right]-\int_{\frac{1}{4}}^{\infty}h_{q}\left(f\left(\sqrt{r-\tfrac{1}{4}}\right)\right)\tanh\left(\pi\sqrt{r-\tfrac{1}{4}}\right)\mathrm{d}r-\frac{\nu_{1}\left(\tilde{h}_{q}\right)}{g}\right|\nonumber \\
 & \leqslant\frac{C}{g^{2}}\sum_{k=0}^{q-1}|a_{k}|(k+1)^{c}\leqslant\frac{C2^{c}}{g^{2}}\left(|a_{0}|+\sum_{k=1}^{q-1}|a_{k}|k^{c}\right)\leqslant\frac{C}{g^{2}}\left(\|w_{q}^{(c+1)}\|_{[0,2\pi]}+\|\tilde{h}_{q}\|_{\left[-f\left(\frac{i}{2}\right),f\left(\frac{i}{2}\right)\right]}\right),\label{eq:master-inq-poly}
\end{align}
where $w_{q}(\theta)\eqdf\tilde{h}_{q}\left(f\left(\frac{i}{2}\right)\cos\theta\right)$,
and for the last inequality we used \cite[Lemma 2.3]{Ma.Pu.vH2025}
applied to $\tilde{h}_{q}$ and absorbed constants into $C$, thus
the statement holds for each $h_{q}$. We also used the fact that
because $\nu_{1}$ has support inside $\left[-f\left(\frac{i}{2}\right),f\left(\frac{i}{2}\right)\right]$,
its value on $\tilde{h}_{q}$ coincides with its value on the Chebyshev
polynomial expansion of $\tilde{h}_{q}$. The fact that $\tilde{h}_{q}$
converges to $\tilde{h}$ in the $C^{c+1}$ norm and the fact that
all of the bounds in (\ref{eq:master-inq-poly}) are independent of
$q$ then means that (\ref{eq:master-inq-poly}) extends to $\tilde{h}\in C^{\infty}(\mathbb{R})$.
\end{proof}

\subsection{Proof of Theorem \ref{thm:Main-Thm}}

We now analyse the support of the distribution $\nu_{1}$ appearing
in Proposition \ref{prop:smooth-master}. Theorem \ref{thm:Main-Thm}
will then follow from picking a suitable $h$ in Proposition \ref{prop:smooth-master}
and by Markov's inequality. 
\begin{lem}
\label{lem:support-dists}Suppose that $h$ is a smooth function with
$\textup{Supp}h\subset\left(f\left(0\right),f\left(i\sqrt{\frac{1}{4}-0.0023}\right)\right)$
and let $\tilde{h}(x)=\frac{h(x)}{x}$, then 
\[
\nu_{0}\left(\tilde{h}\right)=\nu_{1}\left(\tilde{h}\right)=0,
\]
where $\nu_{1}$ is as in Proposition \ref{prop:smooth-master} and
$\nu_{0}$ is defined by
\[
\nu_{0}\left(\tilde{h}\right)=\int_{\frac{1}{4}}^{\infty}h\left(f\left(\sqrt{r-\tfrac{1}{4}}\right)\right)\tanh\left(\pi\sqrt{r-\tfrac{1}{4}}\right)\mathrm{d}r.
\]
\end{lem}

\begin{proof}
First, we have
\[
\nu_{0}\left(\tilde{h}\right)=\int_{\frac{1}{4}}^{\infty}h\left(f\left(\sqrt{r-\tfrac{1}{4}}\right)\right)\tanh\left(\pi\sqrt{r-\tfrac{1}{4}}\right)\mathrm{d}r=0,
\]
for any $h$ that is zero on a neighbourhood of $[-f(0),f(0)]$ since
$f\left(\sqrt{r-\frac{1}{4}}\right)\in[0,f(0)]$ for any $r\geqslant\frac{1}{4}$,
so the conclusion holds for $\nu_{0}$. 

Now for $P(x)=\sum_{j=0}^{q}s_{j}x^{j}$, define

\[
\tilde{\nu}_{1}(P)=\sum_{j=0}^{q}s_{j}\int_{0}^{\infty}2\cosh\left(\frac{r}{2}\right)\check{f}^{*(j+1)}(r)\mathrm{d}r,
\]
and 
\[
\nu_{1}^{0}(P)=\sum_{j=0}^{q}s_{j}\int_{0}^{\infty}\sum_{k=1}^{\infty}2\frac{\sinh\left(\frac{r}{2}\right)^{2}}{\sinh\left(\frac{kr}{2}\right)}\check{f}^{*(j+1)}(kr)\mathrm{d}r.
\]
Then

\[
\nu_{1}(P)=(\nu_{1}^{0}-\tilde{\nu}_{1})(P)+\tilde{\nu}_{1}(P)-\nu_{0}(P).
\]
We will show that $\nu_{1}^{0}-\tilde{\nu_{1}}$ and $\tilde{\nu_{1}}$
extend to compactly supported distributions with support in $[-f(0),f(0)]$
and at $f\left(\frac{i}{2}\right)$ respectively which implies the
conclusion as $\text{Supp}\tilde{h}\subseteq\text{Supp}h$ and $\text{Supp}h\cap\left([-f(0),f(0)]\cup f\left(\frac{i}{2}\right)\right)=\emptyset$.

First we calculate

\begin{align*}
\tilde{\nu}_{1}(P) & =\sum_{j=0}^{q}s_{j}\int_{-\infty}^{\infty}\cosh\left(\frac{r}{2}\right)\check{f}^{*(j+1)}(r)\mathrm{d}r\\
 & =\sum_{j=0}^{q}s_{j}\int_{-\infty}^{\infty}e^{\frac{r}{2}}\check{f}^{*(j+1)}(r)\mathrm{d}r\\
 & =\sum_{j=0}^{q}s_{j}f^{j+1}\left(\frac{i}{2}\right)\\
 & =f\left(\frac{i}{2}\right)P\left(f\left(\frac{i}{2}\right)\right).
\end{align*}
This means that $\tilde{\nu}_{1}$ extends to a compactly supported
distribution on $C^{\infty}(\mathbb{R})$, and for $h$ with $\textup{Supp}h\subset\left(f\left(0\right),f\left(i\sqrt{\frac{1}{4}-0.0023}\right)\right)$,
we have 
\[
\tilde{\nu}_{1}\left(\tilde{h}\right)=f\left(\frac{i}{2}\right)\tilde{h}\left(f\left(\frac{i}{2}\right)\right)=h\left(f\left(\frac{i}{2}\right)\right)=0.
\]
 Now because $\nu_{0}$ and $\nu_{1}$ extend to compactly supported
distributions, so does $\nu_{1}^{0}=\nu_{0}+\nu_{1}$. Moreover, since
$\tilde{\nu}_{1}$ extends to a compactly supported distribution,
so does $\nu_{1}^{0}-\tilde{\nu_{1}}$. It remains to show that $\nu_{1}^{0}-\tilde{\nu_{1}}$
is supported in $[-f(0),f(0)]$. By \cite[Lemma 4.9]{Ch.Ga.Tr.va2024},
we just need to show that
\[
\limsup_{p\to\infty}\left|(\nu_{1}^{0}-\tilde{\nu_{1}})\left(x^{p}\right)\right|^{\frac{1}{p}}\leqslant f\left(0\right).
\]
We have 
\begin{align*}
\left|(\nu_{1}^{0}-\tilde{\nu_{1}})\left(x^{p}\right)\right| & \leqslant2\int_{0}^{\infty}\left|\sinh\left(\frac{r}{2}\right)-\cosh\left(\frac{r}{2}\right)\right|\check{f}^{*(p+1)}(r)\mathrm{d}r+\int_{0}^{\infty}2\sum_{k\geqslant2}\frac{\sinh\left(\frac{r}{2}\right)^{2}}{\sinh\left(\frac{kr}{2}\right)}\check{f}^{*(p+1)}(kr)\mathrm{d}r\\
 & \leqslant2\int_{0}^{\infty}\check{f}^{*(p+1)}(r)\mathrm{d}r+2\int_{0}^{\infty}\check{f}^{*(p+1)}(r)\sum_{k\geqslant2}\frac{1}{k}\frac{\sinh\left(\frac{r}{2k}\right)^{2}}{\sinh\left(\frac{r}{2}\right)}\mathrm{d}r.
\end{align*}
We claim that for $x\geqslant0$ and $k\geqslant2$ we have
\[
\sinh\left(\frac{x}{2k}\right)^{2}\leqslant\frac{1}{k}\sinh\left(\frac{x}{2}\right).
\]
Indeed, this is equivalent to showing that 
\[
\cosh\left(\frac{x}{k}\right)-\frac{2}{k}\sinh\left(\frac{x}{2}\right)\leqslant1,
\]
as $\sinh^{2}(\frac{x}{2})=\frac{1}{2}\cosh(x)-\frac{1}{2}$. But,
the left hand side is decreasing since its derivative is 
\[
\frac{1}{k}\left(\sinh\left(\frac{x}{k}\right)-\cosh\left(\frac{x}{2}\right)\right)<0,
\]
and so the claim is true as it holds for $x=0$. So,
\[
\left|(\nu_{1}^{0}-\tilde{\nu_{1}})\left(x^{p}\right)\right|\leqslant2\sum_{k=1}^{\infty}\frac{1}{k^{2}}\int_{0}^{\infty}\check{f}^{*(p+1)}(r)\mathrm{d}r=\sum_{k=1}^{\infty}\frac{1}{k^{2}}f(0)^{p+1}=\left(\frac{\pi^{2}}{6}f(0)\right)f(0)^{p}.
\]
Taking $\frac{1}{p}$ powers and $p\to\infty$ then gives the result.

\end{proof}
We are now ready to finish the proof of Theorem \ref{thm:Main-Thm}
\begin{proof}[Proof of Theorem \ref{thm:Main-Thm}.]
Let $\ep=\ep(g)>0$ and $h$ be a non-negative smooth function which
is equal to $1$ in $\left[f\left(i\sqrt{\ep\left(g\right)}\right),f\left(i\sqrt{\frac{1}{4}-0.0024}\right)\right]$
and equal to $0$ in $(-\infty,f(0)]$ and $\left[f\left(i\sqrt{\frac{1}{4}-0.0023}\right),f\left(\frac{i}{2}\right)\right]$.
We claim that there is a $c=c(m)>0$ such that $h$ can be picked
so that
\begin{equation}
\|w^{(m)}\|_{[0,2\pi]}\leqslant c(m)\left(\ep(g)\right)^{-\frac{m}{2}},\label{eq:derivative-bound-h}
\end{equation}
where $w(\theta)=\frac{h\left(f\left(\frac{i}{2}\right)\cos\left(\theta\right)\right)}{f\left(\frac{i}{2}\right)\cos\left(\theta\right)}$
and $m$ is as in Proposition \ref{prop:smooth-master}. Then
\begin{align*}
\mathbb{P}_{g}\left[0.0024\leqslant\lambda_{1}\left(X\right)\leqslant\frac{1}{4}-\ep(g)\right] & \leqslant\mathbb{P}_{g}\left[\text{Tr\,}h\left(f\left(\sqrt{\Delta_{X}-\frac{1}{4}}\right)\right)\geqslant1\right].\\
 & \leqslant\mathbb{E}_{g}\left[\textup{Tr\,}h\left(f\left(\sqrt{\Delta_{X}-\frac{1}{4}}\right)\right)\right]
\end{align*}
by Markov's inequality and 
\[
\mathbb{P}_{g}\left[0.0024\leqslant\lambda_{1}\left(X\right)\leqslant\frac{1}{4}-\ep(g)\right]\leqslant\frac{C\left(\ep(g)\right)^{-\frac{m}{2}}}{g}
\]
by Proposition \ref{prop:smooth-master} and Lemma \ref{lem:support-dists}.
Since by \cite[Theorem 4.8]{MirzakhaniRandom}, 
\[
\mathbb{P}_{g}\left[\lambda_{1}\left(X\right)\leqslant0.0024\right]\to0
\]
as $g\to\infty$, we see that there is a $c>0$ such that 
\[
\mathbb{P}_{g}\left[\lambda_{1}\left(X\right)\leqslant\frac{1}{4}-g^{-c}\right]\to0
\]
as $g\to\infty$.

It remains to establish (\ref{eq:derivative-bound-h}). Since $h$
is identically zero in $[-\infty,f(0)]$, 
\begin{align*}
\sup_{x\in\mathbb{R}}\frac{d^{m}}{dx^{m}}\left(\frac{h(x)}{x}\right) & =\sup_{x\in\mathbb{R}}\sum_{j=0}^{m}{m \choose i}h^{j}(x)\frac{d^{m-j}}{dx^{m-j}}\left(\frac{1}{x}\right)\\
 & \leqslant Cm^{2m}\max_{1\leqslant j\leqslant m}\sup_{x\in\mathbb{R}}h^{(j)}\left(x\right)
\end{align*}
for some $C>0$. The existence of a function $h$ is provided by \cite[Lemma 4.10]{Ch.Ga.Tr.va2024},
which concludes the proof of Theorem \ref{thm:Main-Thm}.
\end{proof}

\section{Large genus expansions of Weil-Petersson volumes\label{sec:Large-genus-expansion}}

Theorem \ref{thm:main-technical-result} relied on the following effective
large genus expansion, which will follow as a corollary of Theorem
\ref{thm:expansions}.
\begin{cor}
\label{thm:main-estimate-WP-expansion}There is a constant $c>0$
so that the following expansions hold.
\begin{enumerate}
\item For any integers $a,b$ satisfying with $2g-q=2a+b<2g$ for some $q\geq1$
with $b\leqslant3q$, there exist continuous functions $\left\{ f_{j}^{a,b}\left(\mathbf{x}\right)\right\} _{j\geq1}$
such that for any $k\geq q$, 
\[
\left|\frac{V_{a,b}\left(\mathbf{x}\right)}{V_{g}}-\sum_{j=q}^{k}\frac{f_{j}^{a,b}\left(\mathbf{x}\right)}{g^{j}}\right|\leqslant\frac{\left(ck(b+1)\right)^{ck}\left(|\mathbf{x}|+1\right)^{ck}\exp\left(\frac{1}{2}\left|\mathbf{x}\right|\right)}{g^{k+1}},
\]
for $g>(ckq)^{c}$.
\item For any integers $a,b$ such that $2a+b=2g+n$ and $g-q\leqslant a\leqslant g$
for some $q$, there exist continuous functions $\left\{ \alpha_{n,j}^{a,b}\left(\mathbf{x}\right)\right\} _{j=1}^{k}$
\[
\left|\frac{V_{a,b}\left(\mathbf{x}\right)}{V_{g,n}}-\prod_{i=1}^{b}\frac{\sinh\left(\frac{x_{i}}{2}\right)}{\left(\frac{x_{i}}{2}\right)}-\sum_{j=1}^{k}\frac{1}{g^{j}}\alpha_{n,j}^{a,b}\left(\mathbf{x}\right)\right|\leqslant\frac{\left(ck(n+1)\right)^{ck}\left(1+\left|\mathbf{x}\right|\right)^{ck}\exp\left(\frac{1}{2}\left|\mathbf{x}\right|\right)}{g^{k+1}},
\]
for $g>(cnk)^{c}$.
\end{enumerate}
\end{cor}

Let $\overline{\mathcal{M}}_{g,n}$ be the Deligne-Mumford compactification
of $\mathcal{M}_{g,n}$. There are $n$ tautological line bundles
$\mathcal{L}_{i}$ over $\overline{\mathcal{M}}_{g,n}$ whose fiber
at $X\in\overline{\mathcal{M}}_{g,n}$ is the cotangent space at the
$i$th marked point on $X$. We define the $\psi$-classes $\psi_{i}\eqdf c_{1}\left(\mathcal{L}_{i}\right)$
where $c_{1}$ denotes the first Chern class of the bundle $\mathcal{L}_{i}$.
For $\d=\left(d_{1},\dots,d_{n}\right)\in\mathbb{Z}_{\geqslant0}^{n}$
with $\left|\d\right|\eqdf\sum_{i=1}^{n}d_{i}\leqslant3g+n-3$, we
define 
\begin{equation}
\left[\tau_{d_{1}}\dots\tau_{d_{n}}\right]_{g,n}\eqdf\frac{2^{2\left|\d\right|}\prod_{i=1}^{n}\left(2d_{i}+1\right)!!}{(3g+n-3-|\d|)!}\int_{\overline{\mathcal{M}}_{g,n}}\psi_{1}^{d_{1}}\cdots\psi_{n}^{d_{n}}\omega_{WP}^{3g+n-3-\left|\d\right|}.\label{eq:def-int-numbers}
\end{equation}
If $|\d|>3g+n-3$ , $\left[\tau_{d_{1}}\dots\tau_{d_{n}}\right]_{g,n}$
is taken to be identically $0$. The main technical result of this
section is the following.
\begin{thm}
\label{thm:expansions}Let $C>0$ be a constant as in Lemma \ref{lem:(Grushevsky's-bound,-)}.
For any integer $c>\max\left\{ 600,C\right\} $, the following expansions
hold.
\begin{enumerate}
\item[A1(k)]  For any $\mathbf{d}$ and $n$, there exist $\left\{ b_{\mathbf{d},n}^{i}\right\} _{i\geqslant1}$
such that for any $g>(n+k+1)(c(k+1))^{c}$,
\[
\left|\frac{\left[\prod\tau_{d_{i}}\right]_{g,n}}{V_{g,n}}-1-\sum_{i=1}^{k}\frac{b_{\mathbf{d},n}^{i}}{g^{i}}\right|\leqslant\frac{Q_{n}^{k}(d_{1},\ldots,d_{n})}{g^{k+1}},
\]
where $Q_{n}^{k}$ is the polynomial of degree $2(k+1)$ such that
for any $t_{1},\ldots,t_{n}\in\mathbb{Z}_{\geqslant0}$ the coefficient
$[d_{1}^{t_{1}}\cdots d_{n}^{t_{n}}]Q_{n}^{k}=\frac{(c(k+1))^{c(k+1)}(n+k+1)^{k+1}}{t_{1}!\cdots t_{n}!}$,
and the $b_{\mathbf{d},n}^{i}$ are majorised by the polynomials $q_{n}^{i}$
of degree $2i$ whose coefficients satisfy $\left|[d_{1}^{t_{1}}\cdots d_{n}^{t_{n}}]q_{n}^{i}\right|=\frac{(ci)^{ci}(n+i)^{i}}{t_{1}!\cdots t_{n}!}$.
\item[A2(k)]  For any $\d$ and $n$, there exist $\left\{ e_{\d,n}^{i}\right\} _{i\geqslant1}$
and a polynomial $P_{n}^{k}$ of degree at most $2k+1$ such that
for any $g>(n+k+1)(c(k+1))^{c}$,
\[
\left|\frac{\left[\tau_{d_{1}}\tau_{d_{2}}\cdots\tau_{d_{n}}\right]_{g,n}-\left[\tau_{d_{1}+1}\tau_{d_{2}}\cdots\tau_{d_{n}}\right]_{g,n}}{V_{g,n}}-\sum_{i=1}^{k}\frac{e_{\d,n}^{i}}{g^{i}}\right|\leqslant\frac{P_{n}^{k}(d_{1},\ldots,d_{n})}{g^{k+1}},
\]
and for any $t_{1},\ldots,t_{n}\in\mathbb{Z}_{\geqslant0}$ the coefficient
$\left|[d_{1}^{t_{1}}\cdots d_{n}^{t_{n}}]P_{n}^{k}\right|\leqslant\frac{(c(k+1))^{c(k+1)-3}(n+k+1)^{k+1}}{t_{1}!\cdots t_{n}!}$,
and the $e_{\mathbf{d},n}^{i}$ are majorised by polynomials $v_{n}^{i}$
of degree at most $2(i-1)+1$ whose coefficients satisfy $\left|[d_{1}^{t_{1}}\cdots d_{n}^{t_{n}}]v_{n}^{i}\right|\leqslant\frac{(ci)^{ci-3}(n+i)^{i}}{t_{1}!\cdots t_{n}!}$
.
\item[A3(k)]  For any $n$, there exist $\left\{ h_{n}^{i}\right\} _{i\geqslant1}$
so that for any $g>500(n+k+2)(c(k+1))^{c}$,
\[
\left|\frac{4\pi^{2}\left(2g-2+n\right)V_{g,n}}{V_{g,n+1}}-1-\sum_{i=1}^{k}\frac{h_{n}^{i}}{g^{i}}\right|\leqslant\frac{500\left(c(k+1)\right)^{c(k+1)}(n+k+2)^{k+1}}{g^{k+1}},
\]
and $\left|h_{n}^{i}\right|\leqslant500(ci)^{ci}(n+i+1)^{i}$.
\item[A4(k)]  For any $n$, there exist $\left\{ p_{n}^{i}\right\} _{i\geqslant1}$
so that for any $g>(n+k+1)(c(k+1))^{c}$,
\[
\left|\frac{V_{g-1,n+2}}{V_{g,n}}-1-\sum_{i=1}^{k}\frac{p_{n}^{i}}{g^{i}}\right|\leqslant\frac{\left(c(k+1)\right)^{c(k+1)}(n+k+1)^{k+1}}{g^{k+1}},
\]
and $\left|p_{n}^{i}\right|\leqslant\left(ci\right)^{ci}(n+i)^{i}$.
\end{enumerate}
\end{thm}

Theorem \ref{thm:expansions} is a refinement of \cite[Theorem 4.1]{Mi.Zo2015}.
The difference for us is that we need to know explicit dependence
of the error terms on $k$ and $n$. Our proof relies on a careful
analysis of their argument, tracking the coefficients of the expansion
and the coefficients of the error polynomial through the induction. 

\subsection{Preliminaries}

Theorem \ref{thm:expansions}, based on refinements of the arguments
of \cite{Mi.Zo2015}, are based on analysing recursive formulae for
intersection numbers which are recalled in the next theorem, c.f.
\cite[Section 2]{Mi.Zo2015}.
\begin{thm}
\label{thm-Recurrsions}The following recursive formulae for $\left[\tau_{1}\cdots\tau_{n}\right]_{g,n}$
hold.
\end{thm}

\begin{lyxlist}{00.00.0000}
\item [{i)\label{i)recurssioni}}] 
\begin{align}
[\tau_{0}\tau_{1}\prod_{i=1}^{n}\tau_{d_{i}}]_{g,n+2} & =[\tau_{0}^{4}\prod_{i=1}^{n}\tau_{d_{i}}]_{g-1,n+4}\label{eq:recurs-1A}\\
 & +6\sum_{\substack{I\sqcup J=\{1,...,n\}\\
g_{1}+g_{2}=g
}
}[\tau_{0}^{2}\prod_{i\in I}\tau_{d_{i}}]_{g_{1},|I|+2}[\tau_{0}^{2}\prod_{i\in J}\tau_{d_{i}}]_{g_{2},|J|+2}.\nonumber 
\end{align}
\item [{}]~
\item [{ii)\label{ii)recursion-ii}}] 
\begin{equation}
\left(2g-2+n\right)\left[\prod_{i=1}^{n}\tau_{d_{i}}\right]_{g,n}=\frac{1}{2}\sum_{l=1}^{3g-2+n}\frac{\left(-1\right)^{l-1}l\pi^{2l-2}}{\left(2l+1\right)!}\left[\tau_{l}\prod_{i=1}^{n}\tau_{d_{i}}\right]_{g,n+1}\label{eq:recurssioniii}
\end{equation}
\item [{iii)\label{iii)recursion-iii}}] Let $a_{i}\eqdf\left(1-2^{1-2i}\right)\zeta\left(2i\right)$
where $\zeta$ is the Riemann zeta function and $d_{0}=3g+n-3-|\mathbf{d}|$.
Then
\begin{equation}
\left[\prod_{i=1}^{n}\tau_{d_{i}}\right]_{g,n}=A_{\d}+B_{\d}+C_{\d},\label{eq:recursion4}
\end{equation}
where
\begin{align}
A_{\d} & \eqdf8\sum_{j=2}^{n}\sum_{l=0}^{d_{0}}(2d_{j}+1)a_{l}\left[\tau_{d_{1}+d_{j}+l-1}\prod_{i\neq1,j}\tau_{d_{i}}\right]_{g,n-1},\label{eq:recurs(A)}\\
B_{\d} & \eqdf16\sum_{l=0}^{d_{0}}\sum_{k_{1}+k_{2}=l+d_{1}-2}a_{l}\left[\tau_{k_{1}}\tau_{k_{2}}\prod_{i\neq1}\tau_{d_{i}}\right]_{g-1,n+1},\label{eq:recurs(B)}\\
C_{\d} & \eqdf16\sum_{\substack{g_{1}+g_{2}=g\\
I\sqcup J=\left\{ 2,\dots,n\right\} 
}
}\sum_{l=0}^{d_{0}}\sum_{k_{1}+k_{2}=l+d_{1}-2}a_{l}\left[\tau_{k_{1}}\prod_{i\in I}\tau_{d_{i}}\right]_{g_{1},|I|+1}\left[\tau_{k_{2}}\prod_{i\in J}\tau_{d_{i}}\right]_{g_{2},|J|+1}.\label{eq:recrus(C)}
\end{align}
\end{lyxlist}
Recursion (\ref{eq:recurs-1A})  is a special case of \cite[Propositions 3.3 and 3.4]{Li.Xu2009},
recursion (\ref{eq:recurssioniii}) is proved in \cite{Do-norburry}
and \cite{Li.Xu2009}, and recursion (\ref{eq:recursion4}) is due
to Mirzakhani \cite{Mi2007}. Mirzakhani proved that $V_{g,n}\left(x_{1},\dots,x_{n}\right)$
is a polynomial in $x_{1},\dots,x_{n}$ with coefficients given by
the intersection numbers $\left[\tau_{d_{1}}\dots\tau_{d_{n}}\right]_{g,n}$.
\begin{thm}[{\cite[Theorem 1.1]{Mi07a}}]
\label{thm:Mirz-vol-exp}For $n>0$ and $\ell_{1},\dots,\ell_{n}>0$,
\[
\text{\ensuremath{V_{g,n}\left(2x_{1},\dots,2x_{n}\right)}}=\sum_{\substack{d_{1},\dots,d_{n}\\
|\mathbf{d}|\leqslant3g+n-3
}
}\left[\prod_{i=1}^{n}\tau_{d_{i}}\right]_{g,n}\frac{x_{1}^{2d_{1}}}{\left(2d_{1}+1\right)!}\cdots\frac{x_{n}^{2d_{n}}}{\left(2d_{n}+1\right)!}.
\]
\end{thm}

\subsection{Basic estimates and modifying expansions }

We first record two basic estimates that we will frequently use, the
first is an estimate on series involving the coefficients $a_{i}$
appearing in the recursion (\ref{eq:recursion4}).
\begin{lem}
\label{lem:basic estimate}For any $r\in\mathbb{N}$
\[
\sum_{i=1}^{\infty}\left(a_{i+1}-a_{i}\right)i^{r}\leqslant2r!
\]
\end{lem}

\begin{proof}
We recall that $\left|a_{i+1}-a_{i}\right|<4^{-i}$ by \cite[Equation (2.14)]{Mi.Zo2015}
 so that 

\[
\sum_{i=1}^{\infty}\left(a_{i+1}-a_{i}\right)i^{r}\leqslant\sum_{i=0}^{\infty}\frac{i^{r}}{4^{i}}\leqslant\sum_{i=0}^{\infty}\frac{i^{r}}{e^{i}}.
\]
Now recall the Abel-Plana formula valid for functions $f$ that are
holomorphic on $\mathrm{Re}(z)\geq0$ and satisfy a growth bound $|f(z)|\leq\frac{C}{|z|^{1+\varepsilon}}$
for some $C,\varepsilon>0$ that states

\[
\sum_{i=0}^{\infty}f(j)=\int_{0}^{\infty}f(x)\mathrm{d}x+\frac{1}{2}f(0)+i\int_{0}^{\infty}\frac{f(it)-f(-it)}{e^{2\pi t}-1}\mathrm{d}t.
\]
We apply this formula with $f(z)=z^{r}e^{-z}$ to obtain
\[
\sum_{i=0}^{\infty}\frac{i^{r}}{e^{i}}=\Gamma(r+1)+i^{r+1}\int_{0}^{\infty}t^{r}\frac{e^{-it}-(-1)^{r}e^{it}}{e^{2\pi t}-1}\mathrm{d}t.
\]
The second term on the right-hand side is majorised by 
\begin{align*}
2\int_{0}^{\infty}\frac{t^{r}}{e^{2\pi t}-1}\mathrm{d}t & \leqslant\frac{1}{\pi}\int_{0}^{1}t^{r-1}\mathrm{d}t+2\frac{(r+2)(r+1)}{(2\pi)^{r+2}}r!\int_{1}^{\infty}\frac{1}{t^{2}}\mathrm{d}t\\
 & \leqslant\frac{1}{\pi}+\frac{3}{2\pi^{3}}r!\leqslant r!,
\end{align*}
and so the desired bound follows immediately.
\end{proof}
Next is a simple inequality between the polynomials $q_{n}^{k}$ and
$Q_{n}^{k}$ appearing in $A1(k)$ in Theorem \ref{thm:expansions}.
\begin{lem}
\label{lem:poly-rels}For any $\d,n$ and $t\geqslant1$, we have
$|\d|q_{n}^{t}(d_{1},\ldots,d_{n})\leqslant q_{n}^{t+1}(d_{1},\ldots,d_{n})$.
In particular, $|\d|q_{n}^{t}(d_{1},\ldots,d_{n})\leqslant Q_{n}^{t}(d_{1},\ldots,d_{n})$
since $Q_{n}^{t}\equiv q_{n}^{t+1}$.
\end{lem}

\begin{proof}
By definition,
\begin{align*}
|\d|q_{n}^{t}(d_{1},\ldots,d_{n}) & =\sum_{i=1}^{n}(ct)^{ct}(n+t)^{t}\sum_{\substack{0\leqslant a_{1},\ldots,a_{n}\leqslant2t\\
\sum a_{i}\leqslant2t
}
}\frac{d_{1}^{a_{1}}\cdots d_{i}^{a_{i}+1}\cdots d_{n}^{a_{n}}}{a_{1}!\cdots(a_{i}+1)!\cdots a_{n}!}(a_{i}+1)\\
 & \leqslant\sum_{i=1}^{n}(2t+1)(ct)^{ct}(n+t)^{t}\sum_{\substack{0\leqslant a_{1},\ldots,a_{n}\leqslant2t+1\\
\sum a_{i}\leqslant2t+1
}
}\frac{d_{1}^{a_{1}}\cdots d_{n}^{a_{n}}}{a_{1}!\cdots a_{n}!}\\
 & \leqslant(c(t+1))^{c(t+1)}(n+t+1)^{t+1}\sum_{\substack{0\leqslant a_{1},\ldots,a_{n}\leqslant2(t+1)\\
\sum a_{i}\leqslant2(t+1)
}
}\frac{d_{1}^{a_{1}}\cdots d_{n}^{a_{n}}}{a_{1}!\cdots a_{n}!}=q_{n}^{t+1}(d_{1},\ldots d_{n}).
\end{align*}
\end{proof}
Now we record some basic manipulations of asymptotic expansions. We
will frequently wish to obtain an asymptotic expansion of a product
of multiple expressions for which we have asymptotic expansions of
and the following is a general formula for doing this.
\begin{lem}
\label{lem:product-exp}Suppose that $A_{1},\ldots,A_{n}$ have expansions
\[
\left|A_{i}-\sum_{t=0}^{k}\frac{a_{t}^{(i)}}{g^{t}}\right|\leqslant\frac{C_{k}^{(i)}}{g^{k+1}}.
\]
Then
\begin{align*}
\left|\prod_{i=1}^{n}A_{i}-\sum_{t=0}^{k}\sum_{\substack{m_{1},\ldots,m_{n}\\
\sum m_{i}=t\\
m_{i}\geqslant0
}
}\frac{\prod_{i=1}^{n}a_{m_{i}}^{(i)}}{g^{t}}\right| & \leqslant\frac{1}{g^{k+1}}\left(\sum_{i=1}^{n}C_{k}^{(i)}\prod_{\substack{j=1\\
j\neq i
}
}^{n}a_{0}^{(j)}+\sum_{\substack{m_{1},\ldots,m_{n}\\
\sum m_{i}=k+1\\
0\leqslant m_{i}\leqslant k
}
}\prod_{i=1}^{n}a_{m_{i}}^{(i)}\right)\\
 & +\sum_{\substack{I\subseteq\left\{ 1,\ldots,n\right\} \\
|I|\leqslant n-2
}
}\prod_{i\in I}\left(\sum_{t=0}^{k}\frac{a_{t}^{(i)}}{g^{t}}\right)\prod_{j\notin I}\frac{C_{k}^{(j)}}{g^{k+1}}\\
 & +\frac{1}{g^{k+1}}\sum_{i=1}^{n}C_{k}^{(i)}\sum_{t=1}^{k(n-1)}\sum_{\substack{m_{j}:j\neq i\\
\sum m_{j}=t\\
0\leqslant m_{i}\leqslant k
}
}\frac{\prod_{j\neq i}a_{m_{j}}^{(j)}}{g^{t}}\\
 & +\sum_{t=k+2}^{nk}\sum_{\substack{m_{1},\ldots,m_{n}\\
\sum m_{i}=t\\
0\leqslant m_{i}\leqslant k
}
}\frac{\prod_{i=1}^{n}a_{m_{i}}^{(i)}}{g^{t}}.
\end{align*}
\end{lem}

\begin{proof}
The proof relies on the following identity 
\[
\prod_{i=1}^{n}x_{i}=\sum_{I\subseteq\left\{ 1,\ldots,n\right\} }\prod_{i\in I}\tilde{x}_{i}\prod_{j\notin I}(x_{j}-\tilde{x}_{j}),
\]
which can be proven by induction. We apply it to $x_{i}=A_{i}$ and
$\tilde{x}_{i}=\sum_{t=0}^{k}\frac{a_{t}^{(i)}}{g^{t}}$ to obtain
\begin{align*}
 & \left|\prod_{i=1}^{n}A_{i}-\sum_{t=0}^{k}\sum_{\substack{m_{1},\ldots,m_{n}\\
\sum m_{i}=t\\
m_{i}\geqslant0
}
}\frac{\prod_{i=1}^{n}a_{m_{i}}^{(i)}}{g^{t}}\right|\\
= & \left|\sum_{\substack{I\subseteq\left\{ 1,\ldots,n\right\} }
}\prod_{i\in I}\left(\sum_{t=0}^{k}\frac{a_{t}^{(i)}}{g^{t}}\right)\prod_{j\notin I}\left(A_{j}-\sum_{t=0}^{k}\frac{a_{t}^{(j)}}{g^{t}}\right)-\sum_{t=0}^{k}\sum_{\substack{m_{1},\ldots,m_{n}\\
\sum m_{i}=t\\
m_{i}\geqslant0
}
}\frac{\prod_{i=1}^{n}a_{m_{i}}^{(i)}}{g^{t}}\right|.
\end{align*}
The term where $I=\left\{ 1,\ldots,n\right\} $ corresponds to 
\begin{align*}
\prod_{i\in I}\left(\sum_{t=0}^{k}\frac{a_{t}^{(i)}}{g^{t}}\right) & =\sum_{t=0}^{nk}\sum_{\substack{m_{1},\ldots,m_{n}\\
\sum m_{i}=t\\
0\leqslant m_{i}\leqslant k
}
}\frac{\prod_{i=1}^{n}a_{m_{i}}^{(i)}}{g^{t}}\\
 & =\sum_{t=0}^{k}\sum_{\substack{m_{1},\ldots,m_{n}\\
\sum m_{i}=t\\
m_{i}\geqslant0
}
}\frac{\prod_{i=1}^{n}a_{m_{i}}^{(i)}}{g^{t}}+\sum_{\substack{m_{1},\ldots,m_{n}\\
\sum m_{i}=k+1\\
0\leqslant m_{i}\leqslant k
}
}\frac{\prod_{i=1}^{n}a_{m_{i}}^{(i)}}{g^{k+1}}+\sum_{t=k+2}^{nk}\sum_{\substack{m_{1},\ldots,m_{n}\\
\sum m_{i}=t\\
0\leqslant m_{i}\leqslant k
}
}\frac{\prod_{i=1}^{n}a_{m_{i}}^{(i)}}{g^{t}}.
\end{align*}
We also distinguish the terms when $|I|=n-1$ as these will also give
leading order contributions
\begin{align*}
 & \left|\sum_{\substack{I\subseteq\left\{ 1,\ldots,n\right\} \\
|I|=n-1
}
}\prod_{i\in I}\left(\sum_{t=0}^{k}\frac{a_{t}^{(i)}}{g^{t}}\right)\prod_{j\notin I}\left(A_{j}-\sum_{t=0}^{k}\frac{a_{t}^{(j)}}{g^{t}}\right)\right|\leqslant\sum_{i=1}^{n}\left|A_{i}-\sum_{t=0}^{k}\frac{a_{t}^{(i)}}{g^{t}}\right|\prod_{j\neq i}\left(\sum_{t=0}^{k}\frac{a_{t}^{(j)}}{g^{t}}\right)\\
 & \leqslant\sum_{i=1}^{n}\frac{C_{k}^{(i)}}{g^{k+1}}\prod_{j\neq i}\left(\sum_{t=0}^{k}\frac{a_{t}^{(j)}}{g^{t}}\right)=\sum_{i=1}^{n}\frac{C_{k}^{(i)}}{g^{k+1}}\prod_{j\neq i}a_{0}^{(j)}+\sum_{i=1}^{n}\frac{C_{k}^{(i)}}{g^{k+1}}\sum_{t=1}^{k(n-1)}\sum_{\substack{m_{j}:j\neq i\\
\sum m_{j}=t\\
0\leqslant m_{j}\leqslant k
}
}\frac{\prod_{j\neq i}a_{m_{j}}^{(j)}}{g^{t}}.
\end{align*}
The result then follows immediately from putting all of these together.

\end{proof}

An important corollary of the expansion for products that we will
frequently use is when we have an expansion of $A_{1}$ whose first
$r$ coefficients are zero. Then for the product $\prod_{i=1}^{n}A_{i}$,
an expansion up to order $g^{-k}$ relies only on knowing the expansion
up to order $g^{-(k-r-1)}$ for the $A_{i}$ with $i\geqslant2$. 
\begin{cor}
\label{cor:product-with-zeroterms}Suppose that $A_{2},\ldots,A_{n}$
have expansions
\[
\left|A_{i}-\sum_{t=0}^{k}\frac{a_{t}^{(i)}}{g^{t}}\right|\leqslant\frac{C_{k}^{(i)}}{g^{k+1}},
\]
and $A_{1}$ has an expansion for some $s$ satisfying $k-(r+1)\leqslant s\leqslant k$,
\[
\left|A_{1}-\sum_{t=r+1}^{s+r+1}\frac{a_{t}^{(1)}}{g^{t}}\right|\leqslant\frac{C_{s+r+1}^{(1)}}{g^{s+r+2}}
\]
Then
\begin{align*}
 & \left|\prod_{i=1}^{n}A_{i}-\sum_{t=r+1}^{s+r+1}\sum_{\substack{m_{1},\ldots,m_{n}\\
\sum m_{i}=t\\
0\leqslant m_{i}\leqslant k,\ i\geqslant2\\
r+1\leqslant m_{1}\leqslant s+r+1
}
}\frac{\prod_{i=1}^{n}a_{m_{i}}^{(i)}}{g^{t}}\right|\leqslant\frac{C_{s+r+1}^{(1)}}{g^{s+r+2}}\prod_{j=2}^{n}a_{0}^{(j)}+\frac{a_{r+1}^{(1)}}{g^{s+r+2}}\sum_{i=2}^{n}C_{s}^{(i)}\prod_{\substack{j=2\\
j\neq i
}
}^{n}a_{0}^{(j)}\\
 & +\frac{1}{g^{s+r+2}}\sum_{\substack{m_{1},\ldots,m_{n}\\
\sum m_{i}=s+1\\
0\leqslant m_{i}\leqslant k,\ i\geq2
}
}a_{m_{1}+r+1}^{(1)}\prod_{i=2}^{n}a_{m_{i}}^{(i)}+\frac{1}{g^{r+1}}\sum_{\substack{I\subseteq\left\{ 1,\ldots,n\right\} \\
|I|\leqslant n-2\\
1\in I
}
}\left(\sum_{t=0}^{s}\frac{a_{t+r+1}^{(1)}}{g^{t}}\right)\prod_{i\in I\setminus\left\{ 1\right\} }\left(\sum_{t=0}^{s}\frac{a_{t}^{(i)}}{g^{t}}\right)\prod_{j\notin I}\frac{C_{s}^{(j)}}{g^{s+1}}\\
 & +\frac{C_{s+r+1}^{(1)}}{g^{s+r+2}}\sum_{\substack{I\subseteq\left\{ 1,\ldots,n\right\} \\
|I|\leqslant n-2\\
1\notin I
}
}\prod_{i\in I}\left(\sum_{t=0}^{s}\frac{a_{t}^{(i)}}{g^{t}}\right)\prod_{j\notin I\sqcup\left\{ 1\right\} }\frac{C_{s}^{(j)}}{g^{s+1}}+\frac{C_{s+r+1}^{(1)}}{g^{s+r+2}}\sum_{t=1}^{s(n-1)}\sum_{\substack{m_{j}:j\neq1\\
\sum m_{j}=t\\
0\leqslant m_{j}\leqslant s
}
}\frac{\prod_{j\neq i}a_{m_{j}}^{(j)}}{g^{t}}\\
 & +\frac{1}{g^{s+r+2}}\sum_{i=2}^{n}C_{s}^{(i)}\sum_{t=1}^{s(n-1)}\sum_{\substack{m_{j}:j\neq i\\
\sum m_{j}=t\\
0\leqslant m\leqslant s
}
}\frac{a_{m_{1}+r+1}^{(1)}\prod_{j\neq i}a_{m_{j}}^{(j)}}{g^{t}}+\frac{1}{g^{r+1}}\sum_{t=s+2}^{ns}\sum_{\substack{m_{1},\ldots,m_{n}\\
\sum m_{i}=t\\
0\leqslant m_{i}\leqslant s
}
}\frac{a_{m_{1}+r+1}^{(1)}\prod_{i=2}^{n}a_{m_{i}}^{(i)}}{g^{t}}.
\end{align*}
\end{cor}

\begin{proof}
We rewrite the expansion for $A_{1}$ as 
\[
\left|A_{1}-\sum_{t=0}^{k}\frac{\tilde{a}_{t}^{(1)}}{g^{t}}\right|\leqslant\frac{\tilde{C}_{k}^{(1)}}{g^{k+1}},
\]
where 
\[
\tilde{a}_{t}^{(1)}\eqdf\frac{a_{t+r+1}^{(1)}}{g^{r+1}},\qquad\tilde{C}_{k}^{(1)}\eqdf\frac{C_{k+r+1}^{(1)}}{g^{r+1}}.
\]
The proof is then immediate from Lemma \ref{lem:product-exp} after
re-indexing summations. For example, 
\begin{align*}
\left|\prod_{i=1}^{n}A_{i}-\sum_{t=0}^{s}\sum_{\substack{m_{1},\ldots,m_{n}\\
\sum m_{i}=t\\
m_{i}\geq0
}
}\frac{\tilde{a}_{m_{1}}^{(1)}\prod_{i=2}^{n}a_{m_{i}}^{(i)}}{g^{t}}\right| & =\left|\prod_{i=1}^{n}A_{i}-\sum_{t=0}^{s}\sum_{\substack{m_{1},\ldots,m_{n}\\
\sum m_{i}=t\\
0\leqslant m_{i}\leqslant s
}
}\frac{a_{m_{1}+r+1}^{(1)}\prod_{i=2}^{n}a_{m_{i}}^{(i)}}{g^{t+r+1}}\right|\\
 & =\left|\prod_{i=1}^{n}A_{i}-\sum_{t=r+1}^{s+r+1}\sum_{\substack{m_{1},\ldots,m_{n}\\
\sum m_{i}=t-(r+1)\\
0\leqslant m\leqslant s
}
}\frac{a_{m_{1}+r+1}^{(1)}\prod_{i=2}^{n}a_{m_{i}}^{(i)}}{g^{t}}\right|\\
 & =\left|\prod_{i=1}^{n}A_{i}-\sum_{t=r+1}^{s+r+1}\sum_{\substack{m_{1},\ldots,m_{n}\\
\sum m_{i}=t\\
0\leqslant m_{i}\leqslant s,\ i\geq2\\
r+1\leqslant m_{1}\leqslant s+r+1
}
}\frac{a_{m_{1}}^{(1)}\prod_{i=2}^{n}a_{m_{i}}^{(i)}}{g^{t}}\right|.
\end{align*}
Similar computation yields the error term.
\end{proof}
All of the expansions that we wish to obtain will be in inverse powers
of $g$, however at intermediary steps towards this, we will deal
with expansions in terms of inverse powers of $g-m$ for $m<g$. The
following result records a way to pass between these expansions.
\begin{lem}
\label{lem:shift}Let $m\in\mathbb{N}$ with $m\leqslant k+1\leqslant g^{\frac{1}{3}}$.
Suppose that $A$ has an expansion of the form
\[
\left|A-\sum_{i=0}^{k}\frac{a_{i}}{(g-m)^{i}}\right|\leqslant\frac{C_{k}}{(g-m)^{k+1}},
\]
Then $A$ also has the expansion 
\[
\left|A-\sum_{i=0}^{k}\frac{m^{i}b_{i}}{g^{i}}\right|\leqslant\frac{C_{k}'}{g^{k+1}},
\]
where $b_{0}=a_{0}$ and 
\[
b_{t}=\sum_{i=1}^{t}{t-1 \choose i-1}m^{-i}a_{i},
\]
and 
\[
C_{k}'=3C_{k}+3k^{2}e^{m}\tilde{a}_{k},
\]
where $\tilde{a}_{i}$ are bounds on the $a_{i}$ such that $\frac{\tilde{a}_{i}}{(i-1)!}$
are increasing for $i\geqslant1.$ 
\end{lem}

\begin{proof}
We note that 
\begin{align*}
\sum_{i=0}^{k}\frac{a_{i}}{(g-m)^{i}} & =\sum_{i=0}^{k}\frac{a_{i}}{g^{i}}\frac{1}{(1-\frac{m}{g})^{i}}=a_{0}+\sum_{i=1}^{k}\frac{a_{i}}{g^{i}}\sum_{j=0}^{\infty}{i+j-1 \choose i-1}\left(\frac{m}{g}\right)^{j}\\
 & =a_{0}+\sum_{j=0}^{\infty}\sum_{i=1}^{k}{i+j-1 \choose i-1}m^{j}a_{i}\frac{1}{g^{i+j}}=a_{0}+\sum_{t=1}^{\infty}\sum_{\substack{1\leqslant i\leqslant k\\
j\geqslant0\\
i+j=t
}
}{t-1 \choose i-1}m^{j}a_{i}\frac{1}{g^{t}}.
\end{align*}
For $t\leqslant k$, the inner summation becomes 
\[
\sum_{\substack{1\leqslant i\leqslant k\\
j\geqslant0\\
i+j=t
}
}{t-1 \choose i-1}m^{j}a_{i}=m^{t}\sum_{i=1}^{t}{t-1 \choose i-1}m^{-i}a_{i}.
\]
For $t\geqslant k+1$, we have 
\begin{align*}
\sum_{t=k+1}^{\infty}\sum_{\substack{1\leqslant i\leqslant k\\
j\geqslant0\\
i+j=t
}
}{t-1 \choose i-1}m^{j}a_{i}\frac{1}{g^{t}} & =\sum_{i=1}^{k}\frac{a_{i}}{m^{i}}\sum_{t=k+1}^{\infty}{t-1 \choose i-1}\frac{m^{t}}{g^{t}}\\
 & \leqslant\sum_{i=1}^{k}\frac{a_{i}}{(i-1)!}\sum_{t=k+1}^{\infty}\frac{(t-1)!}{(t-i)!}\frac{m^{t-i}}{g^{t}}\\
 & \leqslant\frac{k\tilde{a}_{k}}{k!}\sum_{i=1}^{k}\sum_{t=k+1}^{\infty}\frac{(t-1)!}{(t-i)!}\frac{m^{t-i}}{g^{t}}
\end{align*}
Note that $\frac{m^{t-i}}{(t-i)!}$ is increasing in $i$ when $i\leq t-m$
and decreasing in $i$ when $i\geqslant t-m$. It follows that in
our range of $i$, $\frac{m^{t-i}}{(t-i)!}$ is maximized when $i=\min\left\{ k,t-m\right\} $.
It is thus instructive to split the end summation
\[
\sum_{t=k+1}^{\infty}\frac{(t-1)!}{(t-i)!}\frac{m^{t-i}}{g^{t}}\leqslant\sum_{t=k+1}^{k+m-1}\frac{(t-1)!}{m!}\frac{m^{m}}{g^{t}}+\sum_{t=k+m}^{\infty}\frac{(t-1)!}{(t-k)!}\frac{m^{t-k}}{g^{t}},
\]
where the first summation is zero if $m=1$. Now
\begin{align*}
\frac{k\tilde{a}_{k}}{k!}\sum_{i=1}^{k}\sum_{t=k+1}^{k+m-1}\frac{(t-1)!}{m!}\frac{m^{m}}{g^{t}} & \leqslant\frac{k^{2}\tilde{a}_{k}}{k!g^{k+1}}\sum_{t=0}^{m-2}\frac{(t+k)!}{m!}\frac{m^{m}}{g^{t}}\\
 & =\frac{k^{2}\tilde{a}_{k}}{k!g^{k+1}}\sum_{t=0}^{m-2}\frac{(t+k)!}{t!}\frac{m^{t}}{g^{t}}\frac{t!m^{m-t}}{m!}\\
 & \leqslant\frac{k^{2}\tilde{a}_{k}}{g^{k+1}}\frac{m^{m}}{m!}\sum_{t=0}^{m-2}\frac{(t+k)!}{t!k!}\frac{m^{t}}{g^{t}}.
\end{align*}
where on the last line we used that $t!m^{m-t}$ is decreasing in
$t$ for $t\leqslant m-2$ so is maximized at $t=0$. We moreover
have
\begin{align*}
 & \frac{k\tilde{a}_{k}}{k!}\sum_{i=1}^{k}\sum_{t=k+m}^{\infty}\frac{(t-1)!}{(t-k)!}\frac{m^{t-k}}{g^{t}}\leqslant k^{2}\tilde{a}_{k}\sum_{t=k+m-1}^{\infty}\frac{(t-1)!}{k!(t-k)!}\frac{m^{t-k}}{g^{t}}\\
 & =\frac{mk^{2}\tilde{a}_{k}}{g^{k+1}}\sum_{t=m-1}^{\infty}\frac{(t+k)!}{k!(t+1)!}\frac{m^{t}}{g^{t}}\leqslant\frac{mk^{2}\tilde{a}_{k}}{g^{k+1}}\sum_{t=m-1}^{\infty}\frac{(t+k)!}{k!t!}\frac{m^{t}}{g^{t}}.
\end{align*}
Putting these together, we estimate
\begin{align*}
 & \sum_{t=k+1}^{\infty}\sum_{\substack{1\leq i\leq k\\
j\geq0\\
i+j=t
}
}{t-1 \choose i-1}m^{j}a_{i}\frac{1}{g^{t}}\leqslant\frac{e^{m}k^{2}\tilde{a}_{k}}{g^{k+1}}\sum_{t=0}^{\infty}\frac{(t+k)!}{k!t!}\frac{m^{t}}{g^{t}}\\
 & =\frac{e^{m}k^{2}\tilde{a}_{k}}{g^{k+1}}\left(\frac{1}{1-\frac{m}{g}}\right)^{k}\leqslant\frac{e^{m}k^{2}\tilde{a}_{k}}{g^{k+1}}\exp\left(\frac{mk}{g-m}\right)\leqslant\frac{3e^{m}k^{2}\tilde{a}_{k}}{g^{k+1}},
\end{align*}
where on the last line we us that $mk\leqslant g-m$ since $m\leqslant k\leqslant g^{\frac{1}{3}}$.
It follows that

\[
\left|A-\sum_{t=0}^{k}\frac{m^{t}b_{t}}{g^{t}}\right|\leqslant\frac{C_{k}}{(g-m)^{k+1}}+C\left(\frac{m}{g}\right)^{k+1}\sum_{i=1}^{k}\frac{a_{i}k^{i}}{m^{i}}
\]
with $b_{0}=a_{0}$ and
\[
b_{t}\eqdf\sum_{i=1}^{t}{t-1 \choose i-1}m^{-i}a_{i}.
\]
The result then follows again from $m\leqslant k+1\leqslant g^{\frac{1}{3}}$
so that
\begin{align*}
 & \frac{C_{k}}{(g-m)^{k+1}}=\frac{C_{k}}{g^{k+1}}\left(1+\frac{m}{g-m}\right)^{k+1}\\
 & \leqslant\frac{C_{k}}{g^{k+1}}\left(1+\frac{m}{k+1}\right)^{k+1}\leqslant\frac{3C_{k}}{g^{k+1}}.
\end{align*}
\end{proof}
Often we will both shift an asymptotic expansion base and then take
products of the resulting expansions. When the factor with which one
shifts by is bounded by some common term, and the original asymptotic
expansion coefficients take a form akin to those seen in Theorem \ref{thm:expansions},
then we can obtain a better bound on the resulting asymptotic expansion
coefficients by simultaneously treating each shifted coefficient in
the product. The following lemma is a blackbox for this procedure
that we will apply in our proofs.
\begin{lem}
\label{lem:coeff-prod-error}For $t_{1},\ldots,t_{k}\in\mathbb{N}$
and $b,c>1$, we have
\[
\sum_{p_{1}=1}^{t_{1}}{t_{1}-1 \choose p_{1}-1}b^{t_{1}-p_{1}}(ct_{1})^{cp_{1}}\cdots\sum_{p_{k}=1}^{t_{k}}{t_{k}-1 \choose p_{k}-1}b^{t_{k}-p_{k}}(ct_{k})^{cp_{k}}\leq(c(t_{1}+\ldots+t_{k})+b)^{c(t_{1}+\ldots+t_{k})}.
\]
\end{lem}

\begin{proof}
By shifting the summations and using the bound ${n_{1} \choose m_{1}}{n_{2} \choose m_{2}}\leq{n_{1}+n_{2} \choose m_{1}+m_{2}}$
we readily obtain
\begin{align*}
 & \sum_{p_{1}=1}^{t_{1}}{t_{1}-1 \choose p_{1}-1}b^{t_{1}-p_{1}}(ct_{1})^{cp_{1}}\cdots\sum_{p_{k}=1}^{t_{k}}{t_{k}-1 \choose p_{k}-1}b^{t_{k}-p_{k}}(ct_{k})^{cp_{k}}\\
 & \leqslant(ct_{1})^{c}\cdots(ct_{k})^{c}\sum_{p_{1}=0}^{t_{1}-1}\cdots\sum_{p_{k}=0}^{t_{k}-1}{\sum_{q=1}^{k}t_{q}-k \choose \sum_{q=1}^{k}p_{q}}b^{\sum_{q=1}^{k}t_{q}-\sum_{q=1}^{k}p_{q}-k}\prod_{q=1}^{k}(ct_{q})^{cp_{q}}.
\end{align*}
We now make the change of variable $\ell=p_{k-1}+p_{k}$ to obtain
\begin{align*}
 & \sum_{p_{1}=0}^{t_{1}-1}\cdots\sum_{p_{k}=0}^{t_{k}-1}{\sum_{q=1}^{k}t_{q}-k \choose \sum_{q=1}^{k}p_{q}}b^{\sum_{q=1}^{k}t_{q}-\sum_{q=1}^{k}p_{q}-k}\prod_{q=1}^{k}(ct_{q})^{cp_{q}}\\
 & =\sum_{p_{1}=0}^{t_{1}-1}\cdots\sum_{p_{k-1}=0}^{t_{k-1}-1}\sum_{\ell=p_{k-1}}^{p_{k-1}+t_{k}-1}{\sum_{q=1}^{k}t_{q}-k \choose \sum_{q=1}^{k-2}p_{q}+\ell}b^{\sum_{q=1}^{k}t_{q}-\sum_{q=1}^{k-2}p_{q}-\ell-k}(ct_{k})^{c(\ell-p_{k-1})}\prod_{q=1}^{k-1}(ct_{q})^{cp_{q}}\\
 & =\sum_{p_{1}=0}^{t_{1}-1}\cdots\sum_{\ell=0}^{t_{k-1}+t_{k}-2}\sum_{p_{k-1}=0}^{\min\left\{ \ell,t_{k-1}-1\right\} }{\sum_{q=1}^{k}t_{q}-k \choose \sum_{q=1}^{k-2}p_{q}+\ell}b^{\sum_{q=1}^{k}t_{q}-\sum_{q=1}^{k-2}p_{q}-\ell-k}(ct_{k})^{c(\ell-p_{k-1})}\prod_{q=1}^{k-1}(ct_{q})^{cp_{q}}\\
 & \leqslant\sum_{p_{1}=0}^{t_{1}-1}\cdots\sum_{\ell=0}^{t_{k-1}+t_{k}-2}{\sum_{q=1}^{k}t_{q}-k \choose \sum_{q=1}^{k-2}p_{q}+\ell}b^{\sum_{q=1}^{k}t_{q}-\sum_{q=1}^{k-2}p_{q}-\ell-k}\\
 & \,\,\,\,\,\,\,\,\,\,\,\,\,\,\,\,\,\,\,\,\,\,\,\,\,\,\,\,\,\,\,\,\,\,\,\,\,\,\,\,\,\,\,\,\,\,\,\,\cdot\sum_{p_{k-1}=0}^{\ell}{\ell \choose p_{k-1}}(ct_{k-1})^{cp_{k-1}}(ct_{k})^{c(\ell-p_{k-1})}\prod_{q=1}^{k-2}(ct_{q})^{cp_{q}}\\
 & \leqslant\sum_{p_{1}=0}^{t_{1}-1}\cdots\sum_{\ell=0}^{t_{k-1}+t_{k}-2}{\sum_{q=1}^{k}t_{q}-k \choose \sum_{q=1}^{k-2}p_{q}+\ell}b^{\sum_{q=1}^{k}t_{q}-\sum_{q=1}^{k-2}p_{q}-\ell-k}(c(t_{k-1}+t_{k}))^{c\ell}\prod_{q=1}^{k-2}(ct_{q})^{cp_{q}}.
\end{align*}
We now repeat this computation with the change of variables $\ell'=\ell+p_{k-2}$.
Continuing in this manner, we reduce to 
\[
\sum_{p_{1}=0}^{t_{1}-1}\sum_{\ell=0}^{t_{2}+\ldots+t_{k}-(k-1)}{\sum_{q=1}^{k}t_{q}-k \choose p_{1}+\ell}b^{\sum_{q=1}^{k}t_{q}-(p_{1}+\ell)-k}(ct_{1})^{cp_{1}}(c(t_{2}+\ldots+t_{k}))^{c\ell}.
\]
Once again, making the change of variables $q=p_{1}+\ell$ we obtain
\begin{align*}
 & \sum_{p_{1}=0}^{t_{1}-1}\sum_{q=p_{1}}^{p_{1}+t_{2}+\ldots+t_{k}-(k-1)}{\sum_{q=1}^{k}t_{q}-k \choose q}b^{\sum_{q=1}^{k}t_{q}-q-k}(ct_{1})^{cp_{1}}(c(t_{2}+\ldots+t_{k}))^{c(q-p_{1})}\\
 & =\sum_{q=0}^{t_{1}+t_{2}+\ldots+t_{k}-k}\sum_{p_{1}=0}^{\min\left\{ q,t_{1}-1\right\} }{\sum_{q=1}^{k}t_{q}-k \choose q}b^{\sum_{q=1}^{k}t_{q}-q-k}(ct_{1})^{cp_{1}}(c(t_{2}+\ldots+t_{k}))^{c(q-p_{1})}\\
 & \leqslant\sum_{q=0}^{t_{1}+t_{2}+\ldots+t_{k}-k}{\sum_{q=1}^{k}t_{q}-k \choose q}b^{\sum_{q=1}^{k}t_{q}-q-k}\sum_{p_{1}=0}^{q}{q \choose p_{1}}(ct_{1})^{cp_{1}}(c(t_{2}+\ldots+t_{k}))^{c(q-p_{1})}\\
 & \leqslant\sum_{q=0}^{t_{1}+t_{2}+\ldots+t_{k}-k}{\sum_{q=1}^{k}t_{q}-k \choose q}b^{\sum_{q=1}^{k}t_{q}-k-q}(c(t_{1}+\ldots+t_{k}))^{cq}\\
 & =\left(c(t_{1}+\ldots+t_{k})^{c}+g_{1}\right)^{t_{1}+\ldots+t_{k}-k}\\
 & \leqslant\left(c(t_{1}+\ldots+t_{k})+g_{1}\right)^{c(t_{1}+\ldots+t_{k}-k)}.
\end{align*}
Now multiplying by $(ct_{1})^{c}\cdots(ct_{k})^{c}\leq\left(c(t_{1}+\ldots+t_{k})+g_{1}\right)^{ck}$
gives the desired result.
\end{proof}

\subsection{Proof of Theorem \ref{thm:expansions}}

The method of proof of Theorem \ref{thm:expansions} follows that
of Mirzakhani and Zograf \cite[Theorem 4.1]{Mi.Zo2015} which is an
inductive argument according to the following schematic. 

\paragraph{Schematic for the proof of Theorem \ref{thm:expansions}.}
\begin{enumerate}
\item Each of $A1(0)$, $A2(0)$, $A3(0)$, $A4(0)$ hold. 
\item For all $k\geqslant0$, if $A2(k)$ holds, then $A1(k)$ holds.
\item For all $k\geqslant0$, if $A1(r)$, $A2(r)$, $A3(r),$and $A4(r)$
hold for all $r\leqslant k$ then $A2(k+1)$ holds.
\item For all $k\geqslant0$, if $A1(k)$ holds, then statements $A3(k)$
and $A4(k)$ hold.
\end{enumerate}
We now prove each step in this schematic, starting with the most complicated.
We note importantly that in this Section, we must track constants
very carefully to prove the induction statements and so the constants
$c,C$ appearing in the proof \textbf{do not }change from line to
line (they are fixed once and for all as defined in Theorem \ref{thm:expansions}).
\begin{prop}
\label{prop:A2}Suppose that for some $k\geqslant1$, each of $A1(r)$,
$A2(r)$, $A3(r),$ and $A4(r)$ hold for all $r\leqslant k$ then
$A2(k+1)$ holds.
\end{prop}

\begin{proof}
We write
\[
\frac{\left[\tau_{d_{1}}\dots\tau_{d_{n}}\right]_{g,n}-\left[\tau_{d_{1}+1}\dots\tau_{d_{n}}\right]_{g,n}}{V_{g,n}}=S_{1}+S_{2}+S_{3},
\]
where 
\begin{align*}
S_{1} & =\frac{1}{4\pi^{2}\left(2g-2+n\right)}\frac{4\pi^{2}\left(2g-2+n\right)V_{g,n-1}}{V_{g,n}}\frac{\tilde{A}_{\d,g,n}}{V_{g,n-1}},\\
S_{2} & =\frac{1}{4\pi^{2}\left(2g-2+n\right)}\frac{4\pi^{2}\left(2g-2+n\right)V_{g,n-1}}{V_{g,n}}\frac{V_{g-1,n+1}}{V_{g,n-1}}\frac{\tilde{B}_{\d,g,n}}{V_{g-1,n+1}},\\
S_{3} & =\frac{\tilde{C}_{\d,g,n}}{V_{g,n}},
\end{align*}
and 
\[
\tilde{A}_{\d,g,n}+\tilde{B}_{\d,g,n}+\tilde{C}_{\d,g,n}=\left[\tau_{d_{1}}\dots\tau_{d_{n}}\right]_{g,n}-\left[\tau_{d_{1}+1}\dots\tau_{d_{n}}\right]_{g,n}
\]
are the terms corresponding to (\ref{eq:recurs(A)}), (\ref{eq:recurs(B)})
and (\ref{eq:recrus(C)}) respectively. In Lemmas \ref{lem:S_1},
\ref{lem:S_2} and \ref{lem:S_3} we obtain asymptotic expansions
for each of $S_{1},S_{2}$ and $S_{3}$ which when combined prove
$A2(k+1)$.
\end{proof}
\begin{lem}
\label{lem:S_1}Suppose that for some $k\geqslant1$, each of $A1(r)$,
$A2(r)$, $A3(r),$ and $A4(r)$ hold for all $r\leqslant k$, then
there exist $\left\{ w_{\mathbf{d},n}^{i}\right\} _{i\geqslant1}$
and a polynomial $P_{n}^{k+1}$ of degree at most $2(k+1)+1$ such
that for any $g>(n+k+1)(c(k+1))^{c}$,
\[
\left|S_{1}-\sum_{i=1}^{k+1}\frac{w_{\d,n}^{i}}{g^{i}}\right|\leqslant\frac{P_{n}^{k+1}(d_{1},\ldots,d_{n})}{g^{k+2}},
\]
and for any $t_{1},\ldots,t_{n}\in\mathbb{Z}_{\geqslant0}$ the coefficient
$\left|[d_{1}^{t_{1}}\cdots d_{n}^{t_{n}}]P_{n}^{k+1}\right|\leqslant\frac{(c(k+1))^{c(k+1)-4}(n+k+2)^{k+2}}{t_{1}!\cdots t_{n}!}$,
and the $w_{\mathbf{d},n}^{i}$ are majorised by polynomials $v_{n}^{i}$
of degree at most $2(i-1)+1$ whose coefficients satisfy $\left|[d_{1}^{t_{1}}\cdots d_{n}^{t_{n}}]v_{n}^{i}\right|\leqslant\frac{(ci)^{ci-4}(n+i)^{i}}{t_{1}!\cdots t_{n}!}$
.
\end{lem}

\begin{proof}
By (\ref{eq:recurs(A)}), $\tilde{A}_{\d,g,n}$ is given by 
\[
8\sum_{j=2}^{n}\sum_{i=0}^{d_{0}}(2d_{j}+1)(a_{i}-a_{i-1})\left[\tau_{d_{1}+d_{j}+i-1}\tau_{d_{2}}\dots\hat{\tau}_{d_{j}}\dots\tau_{d_{n}}\right]_{g,n-1},
\]
where we write $a_{-1}\equiv0$ and a hat on an index indicates that
it is removed from the product. We first deal with the case that $|\d|\leqslant g$.
With the coefficients as in $A1(k)$, we have the bound 
\begin{align*}
 & \left|\frac{\tilde{A}_{\d,g,n}}{V_{g,n-1}}-\sum_{t=0}^{k}\sum_{j=2}^{n}\sum_{i=0}^{\infty}8(a_{i}-a_{i-1})(2d_{j}+1)\frac{b_{\mathbf{d}(i,j),n-1}^{t}}{g^{t}}\right|\leqslant\\
 & \underbrace{\left|8\sum_{j=2}^{n}\sum_{i=0}^{\infty}(a_{i}-a_{i-1})(2d_{j}+1)\left(\frac{\left[\tau_{d_{1}+d_{j}+i-1}\tau_{d_{2}}\dots\hat{\tau}_{d_{j}}\dots\tau_{d_{n}}\right]_{g,n-1}}{V_{g,n-1}}-\sum_{t=0}^{k}\frac{b_{\mathbf{d}(i,j),n-1}^{t}}{g^{t}}\right)\right|}_{(1)}\\
 & +\underbrace{\left|8\sum_{j=2}^{n}\sum_{i=d_{0}+1}^{\infty}(2d_{j}+1)(a_{i}-a_{i-1})\frac{\left[\tau_{d_{1}+d_{j}+i-1}\tau_{d_{2}}\dots\hat{\tau}_{d_{j}}\dots\tau_{d_{n}}\right]_{g,n-1}}{V_{g,n-1}}\right|}_{(2)},
\end{align*}
where we have written $\mathbf{d}(i,j)\eqdf(d_{1}+d_{j}+i-1,d_{2},\ldots\hat{d_{j}},\ldots,d_{n})$
and $b_{\mathbf{d}(i,j)}^{0}=1$. We first will deal with the tail
term (2) which can be bounded using the triangle inequality by 
\begin{align*}
(2) & \leqslant8\sum_{j=2}^{n}\sum_{i=d_{0}+1}^{\infty}(2d_{j}+1)(a_{i}-a_{i-1})\leqslant16n\left(|\mathbf{d}|+1\right)\sum_{i=d_{0}+1}^{\infty}4^{-i}\leqslant\frac{1}{g^{k+2}},
\end{align*}
where we use $d_{0}>g\geqslant|\mathbf{d}|$ and $g>cnk^{2}.$  Let
$Q_{n-1}^{k}(x_{1},\ldots,x_{n-1})$ be the polynomial as in the error
term of $A1(k)$ and denote its coefficient of $x_{1}^{a_{1}}\cdots x_{n-1}^{a_{n-1}}$
by $p_{a_{1},\ldots,a_{n-1}}^{k,n-1}$. For $(1)$, we use $A1(k)$
to obtain an upper bound of the form
\begin{align}
(1) & \leqslant\left|8\sum_{j=2}^{n}\sum_{i=0}^{\infty}(a_{i}-a_{i-1})(2d_{j}+1)\left(\frac{\left[\tau_{d_{1}+d_{j}+i-1}\tau_{d_{2}}\dots\hat{\tau}_{d_{j}}\dots\tau_{d_{n}}\right]_{g,n-1}}{V_{g,n-1}}-\sum_{t=0}^{k}\frac{b_{\mathbf{d}(i,j),n-1}^{t}}{g^{t}}\right)\right|\label{eq:S1-error-bound}\\
 & \leqslant\frac{1}{g^{k+1}}8\sum_{j=2}^{n}(2d_{j}+1)\sum_{i=0}^{\infty}\left(a_{i}-a_{i-1}\right)Q_{n-1}^{k}\left(d_{1}+d_{j}+i-1,d_{2},\dots,\hat{d_{j}},\dots,d_{n}\right)\nonumber \\
 & =\frac{1}{g^{k+1}}8\sum_{j=2}^{n}(2d_{j}+1)\sum_{i=0}^{\infty}\left(a_{i}-a_{i-1}\right)\nonumber \\
 & \,\,\,\,\,\,\,\,\,\,\,\,\,\,\,\,\,\,\,\,\,\,\,\,\,\,\,\,\,\,\,\,\,\,\,\cdot\sum_{\ell=0}^{2(k+1)}\sum_{\substack{t_{2},\dots,t_{n-1}\\
\sum t_{j}\leqslant2(k+1)-\ell
}
}p_{\left(\ell,t_{2},\dots,t_{n-1}\right)}^{k,n-1}\left(d_{1}+d_{j}+i-1\right)^{\ell}d_{2}^{t_{2}}\dots d_{n}^{t_{n-1}}\nonumber \\
 & =\frac{1}{g^{k+1}}8\sum_{j=2}^{n}(2d_{j}+1)\sum_{\ell=0}^{2(k+1)}\nonumber \\
 & \,\,\,\,\,\,\,\,\,\,\,\,\,\,\,\,\,\,\,\,\cdot\sum_{\substack{t_{2},\dots,t_{n-1}\\
\sum t_{j}\leqslant2(k+1)-\ell
}
}p_{\left(\ell,t_{2},\dots,t_{n-1}\right)}^{k,n-1}d_{2}^{t_{2}}\dots d_{n}^{t_{n-1}}\sum_{i=0}^{\infty}\left(a_{i}-a_{i-1}\right)\left(d_{1}+d_{j}+i-1\right)^{\ell}.\nonumber 
\end{align}
Now, we have by Lemma \ref{lem:basic estimate}
\begin{align*}
 & \sum_{i=0}^{\infty}\left(a_{i}-a_{i-1}\right)\left(d_{1}+d_{j}+i-1\right)^{\ell}=\sum_{m=0}^{\ell}(d_{1}+d_{j}-1)^{\ell-m}{\ell \choose m}\sum_{i=0}^{\infty}\left(a_{i}-a_{i-1}\right)i^{m}\\
 & \leqslant2\sum_{m=0}^{\ell}(d_{1}+d_{j})^{\ell-m}{\ell \choose m}m!=2\sum_{m=0}^{\ell}\frac{\ell!}{(\ell-m)!}\sum_{s=0}^{\ell-m}{\ell-m \choose s}d_{1}^{s}d_{j}^{\ell-m-s}\\
 & =2\sum_{m=0}^{\ell}\sum_{s=0}^{\ell-m}\frac{\ell!}{(\ell-m-s)!s!}d_{1}^{s}d_{j}^{\ell-m-s}.
\end{align*}
We thus obtain
\begin{align}
(1) & \leqslant\frac{32}{g^{k+1}}\sum_{j=2}^{n}\sum_{\ell=0}^{2(k+1)}\sum_{\substack{t_{2},\dots,t_{n-1}\\
\sum t_{j}\leqslant2(k+1)-\ell
}
}\sum_{m=0}^{\ell}\sum_{s=0}^{\ell-m}\frac{\ell!}{(\ell-m-s)!s!}p_{\left(\ell,t_{2},\dots,t_{n-1}\right)}^{k,n-1}d_{2}^{t_{2}}\dots d_{n}^{t_{n-1}}d_{1}^{s}d_{j}^{\ell-m-s+1}\label{eq:bound-on-A}\\
 & +\frac{16(n-1)}{g^{k+1}}\sum_{\ell=0}^{2(k+1)}\sum_{\substack{t_{2},\dots,t_{n-1}\\
\sum t_{j}\leqslant2(k+1)-\ell
}
}\sum_{m=0}^{\ell}\sum_{s=0}^{\ell-m}\frac{\ell!}{(\ell-m-s)!s!}p_{\left(\ell,t_{2},\dots,t_{n-1}\right)}^{k,n-1}d_{2}^{t_{2}}\dots d_{n}^{t_{n-1}}d_{1}^{s}d_{j}^{\ell-m-s}.\nonumber 
\end{align}
The right hand-side is a polynomial in $d_{1},\ldots,d_{n}$ with
degree bounded by $2(k+1)+1$ for the first term and bounded by $2(k+1)$
in the second term. We now show that the right hand-side is majorised
by a polynomial in $d_{1},\ldots,d_{n}$ of degree at most $2(k+1)+1$
whose coefficient of $d_{1}^{a_{1}}\cdots d_{n}^{a_{n}}$ is bounded
by (note that $a_{j}\geqslant1)$
\[
\frac{(c(k+2))^{c(k+1)+3}(n+k)^{k+2}}{a_{1}!\cdots a_{n}!}.
\]
To this end, notice that in the first term on the right-hand side
(\ref{eq:bound-on-A}), the coefficient of $d_{1}^{a_{1}}\cdots d_{n}^{a_{n}}$
is bounded by 
\[
\frac{32}{g^{k+1}}\sum_{j=2}^{n}\sum_{\substack{t_{1}\leqslant2(k+1)\\
t_{1}\geqslant a_{1}+a_{j}-1
}
}\frac{t_{1}!}{(a_{j}-1)!a_{1}!}p_{\left(t_{1},a_{2},\dots,a_{j-1},a_{j+1},\ldots,a_{n}\right)}^{k,n-1}.
\]
By Statement $A1(k)$,
\[
\left|p_{\left(t_{1},a_{2},\dots,a_{j-1},a_{j+1},\ldots,a_{n}\right)}^{k,n-1}\right|=\frac{(n+k)^{k+1}(c(k+1))^{c(k+1)}}{t_{1}!a_{1}!\cdots a_{j-1}!a_{j+1}!\cdots a_{n}!},
\]
so this coefficient is majorised by
\begin{align*}
 & \frac{32}{g^{k+1}}\sum_{j=2}^{n}\sum_{\substack{t_{1}\leqslant2(k+1)\\
t_{1}\geq a_{1}+a_{j}-1
}
}\frac{(n+k)^{k+1}(c(k+1))^{c(k+1)}}{a_{1}!\cdots a_{j-1}!a_{j+1}!\cdots a_{n}!(a_{j}-1)!}\\
 & \leqslant\frac{32}{g^{k+1}}\frac{(n+k)^{k+1}(2(k+1))(c(k+1))^{c(k+1)}\sum_{j=2}^{n}a_{j}}{a_{1}!\cdots a_{n}!}\\
 & \leqslant\frac{32}{g^{k+1}}\frac{(n+k)^{k+1}(2(k+1)+1)^{2}(c(k+1))^{c(k+1)}}{a_{1}!\cdots a_{n}!}.
\end{align*}
Similarly, the second term has coefficient majorised by
\[
\frac{16}{g^{k+1}}\frac{(n+k)^{k+2}(2(k+1))(c(k+1))^{c(k+1)}}{a_{1}!\cdots a_{n}!}.
\]
Combining the two gives the claimed bound when recalling that $c>600$.

Now define
\[
\tilde{a}_{\mathbf{d},n}^{r}\eqdf8\sum_{j=2}^{n}\sum_{i=0}^{\infty}(a_{i}-a_{i-1})(2d_{j}+1)b_{\mathbf{d}(i,j),n-1}^{r},
\]
so that $\left|\tilde{a}_{\mathbf{d},n}^{0}\right|\leqslant4(n-1)(|\mathbf{d}|+1)$
by Lemma \ref{lem:basic estimate} and $b_{\mathbf{d}(i,j),n-1}^{0}=1$.
Moreover, in much the same way as the error term, we see that $\tilde{a}_{\mathbf{d},n}^{r}$
is a polynomial in $d_{1},\ldots,d_{n}$ of degree at most $2r+1$
bounded by
\begin{align*}
\left|\tilde{a}_{\mathbf{d},n}^{r}\right| & \leqslant16\sum_{j=2}^{n}\sum_{\substack{t_{1},\ldots,t_{n-1}\\
\sum t_{i}\leqslant2r
}
}t_{1}!q_{(t_{1},\ldots,t_{n-1})}^{r,n-1}(2d_{j}+1)\\
 & \,\,\,\,\,\,\,\,\,\,\,\,\,\,\,\,\,\,\,\,\,\,\,\,\,\,\,\,\cdot\sum_{s=0}^{t_{1}}\sum_{p=0}^{t_{1}-s}\frac{1}{p!(t_{1}-s-p)!}d_{1}^{p}d_{j}^{t_{1}-s-p}d_{2}^{t_{2}}\cdots d_{j-1}^{t_{j-1}}d_{j+1}^{t_{j+1}}d_{n}^{t_{n-1}}.
\end{align*}
where $q_{(t_{1},\ldots,t_{n-1})}^{r,n-1}$ are the polynomials majorising
the $b_{\mathbf{d}(i,j),n-1}^{r}$ as in Statement $A1(k)$. By the
inductive hypothesis on $q_{(t_{1},\ldots,t_{n-1})}^{r,n-1}$, the
coefficient of $d_{1}^{a_{1}}\cdots d_{n}^{a_{n}}$ is thus bounded
by 
\begin{align*}
 & \frac{32(n+r-1)^{r}(cr)^{cr}}{a_{1}!a_{2}\cdots a_{j-1}!a_{j+1}!\cdots a_{n}!}\sum_{j=2}^{n}\sum_{\substack{t_{1}\leqslant2r\\
t_{1}\geqslant a_{1}+a_{j}-1
}
}\frac{a_{j}}{a_{j}!}+\frac{16(n+r-1)^{cr+1}(cr)^{cr}}{a_{1}!\cdots a_{n}!}\sum_{\substack{t_{1}\leqslant2r\\
t_{1}\geqslant a_{1}+a_{j}-1
}
}1\\
 & \leqslant\frac{48(n+r-1)^{r+1}(2r+1)^{2}(cr)^{cr}}{a_{1}!\cdots a_{n}!}\leqslant\frac{(n+r-1)^{r+1}(c(r+1))^{cr+3}}{a_{1}!\cdots a_{n}!}.
\end{align*}
In summary, using only $A1(k)$ we have obtained that whenever $g>(n+k+1)(c(k+1))^{c}$
and $g>|\d|$
\[
\left|\frac{\tilde{A}_{\d,g,n}}{V_{g,n-1}}-\sum_{t=0}^{k}\frac{\tilde{a}_{\mathbf{d},n}^{t}}{g^{t}}\right|\leqslant\frac{P_{n,k}(d_{1},\ldots,d_{n})}{g^{k+1}},
\]
where $P_{n,k}(d_{1},\ldots,d_{n})$ is a polynomial in $d_{1},\ldots,d_{n}$
of degree at most $2(k+1)+1$ with coefficient of $d_{1}^{a_{1}}\cdots d_{n}^{a_{n}}$
bounded by 
\[
\frac{(n+k)^{k+2}(c(k+2))^{c(k+1)+3}}{a_{1}!\cdots a_{n}!},
\]
and each $\tilde{a}_{\mathbf{d},n}^{t}$ is majorised by a polynomial
of degree at most $2t+1$ in $d_{1},\ldots,d_{n}$ with coefficient
of $d_{1}^{a_{1}}\cdots d_{n}^{a_{n}}$ bounded by the same thing
but with $k$ replaced by $t-1$.

When $g<|\d|$, we obtain the same result because for any $i\geqslant0$
we have that $|\d(i,j)|=i+|\d|-1\geqslant g$ and so
\begin{align*}
 & \left|\frac{\tilde{A}_{\d,g,n}}{V_{g,n-1}}-\sum_{t=0}^{k}\sum_{j=2}^{n}\sum_{i=0}^{\infty}8(a_{i}-a_{i-1})(2d_{j}+1)\frac{b_{\mathbf{d}(i,j),n-1}^{t}}{g^{t}}\right|\\
 & \leqslant\left|\frac{\tilde{A}_{\d,g,n}}{V_{g,n-1}}-\sum_{j=2}^{n}\sum_{i=0}^{\infty}8(a_{i}-a_{i-1})(2d_{j}+1)\right|+\sum_{j=2}^{n}\sum_{i=0}^{\infty}8(a_{i}-a_{i-1})(2d_{j}+1)\sum_{t=1}^{k}\frac{q_{n-1}^{t}(\d(i,j))}{g^{t}}\\
 & \leqslant\sum_{j=2}^{n}\sum_{i=0}^{\infty}8(a_{i}-a_{i-1})(2d_{j}+1)\frac{Q_{n}^{k}(\d(i,j))}{g^{k+1}}+\sum_{j=2}^{n}\sum_{i=0}^{\infty}8(a_{i}-a_{i-1})(2d_{j}+1)\frac{(k-1)Q_{n}^{k}(\d(i,j))}{g^{k+1}},
\end{align*}
where the second inequality follows from applying $A1(0)$ for the
first term along with Lemma \ref{lem:poly-rels} for $t=1$ and Lemma
\ref{lem:poly-rels} for $1\leqslant t\leqslant k$ for the second
part since 
\[
g^{-t}q_{n-1}^{t}(\d(i,j))\leqslant|\d(i,j)|^{k+1-t}g^{-(k+1)}q_{n-1}^{t}(\d(i,j))\leqslant g^{-(k+1)}Q_{n-1}^{k}(\d(i,j)).
\]
Thus the error is bounded by $k$ multiplied by the error obtained
in the $|\d|<g$ case and so the estimates hold by an identical analysis.

From $A3(k)$ we also have that there exist $\left\{ h_{n}^{i}\right\} _{i\geqslant1}$
so that for any $g>500\left(c(k+1)(n+k+2)\right)^{c}$,
\[
\left|\frac{4\pi^{2}\left(2g-2+n\right)V_{g,n}}{V_{g,n+1}}-1-\sum_{i=1}^{k}\frac{h_{n}^{i}}{g^{i}}\right|\leqslant\frac{500\left(c(k+1)\right)^{c(k+1)}(n+k+2)^{k+1}}{g^{k+1}},
\]
and $h_{n}^{i}\leqslant500(ci)^{ci}(n+i+1)^{i}$. Moreover, 
\begin{align*}
\frac{1}{4\pi^{2}\left(2g-2+n\right)} & =\frac{1}{8\pi^{2}g}\frac{1}{1-\frac{1-\frac{1}{2}n}{g}}\\
= & \begin{cases}
\sum_{p=1}^{\infty}\frac{(4\pi^{2})^{-1}2^{-p}}{g^{p}} & \text{if }n=1,\\
\frac{(8\pi^{2})^{-1}}{g} & \text{if }n=2,\\
\sum_{p=1}^{\infty}\frac{(8\pi^{2})^{-1}\left(\frac{1}{2}n-1\right)^{p-1}}{g^{p}} & \text{if }n>2.
\end{cases}
\end{align*}
So we have the expansion 
\[
\left|\frac{1}{4\pi^{2}\left(2g-2+n\right)}-\sum_{t=1}^{k+1}\frac{\theta_{n}^{t}}{g^{t}}\right|\leqslant\frac{n^{k+1}}{g^{k+2}},
\]
and the coefficients satisfy $\left|\theta_{n}^{t}\right|\leqslant n^{t-1}$.
Setting $\theta_{n}^{0}\eqdf0$, we can use Corollary \ref{cor:product-with-zeroterms}
(and similar analysis to the above for the error term) to obtain the
expansion of $S_{1}$ up to terms of order of $g^{-(k+1)}$ as 
\begin{align*}
\left|S_{1}-\sum_{t=1}^{k+1}\sum_{\substack{m_{1},m_{2},m_{3}\\
\sum m_{i}=t\\
0\leqslant m_{2},m_{3}\leqslant k\\
1\leqslant m_{1}\leqslant k+1
}
}\frac{\theta_{n}^{m_{1}}\tilde{a}_{\mathbf{d},n}^{m_{2}}h_{n}^{m_{3}}}{g^{t}}\right| & \leqslant\frac{P_{k+1,n}^{(1)}(d_{1}\ldots,d_{n})}{g^{k+2}},
\end{align*}
where $\text{\ensuremath{P_{k+1,n}^{(1)}}(\ensuremath{d_{1}\ldots},\ensuremath{d_{n}})}$
is a polynomial of degree at most $2(k+1)+1$ whose coefficient of
$d_{1}^{a_{1}}\cdots d_{n}^{a_{n}}$ is bounded by 
\[
\frac{(n+k+2)^{k+2}(c(k+2))^{c(k+1)+5}}{a_{1}!\cdots a_{n}!}.
\]
The coefficients of the asymptotic expansion are majorised by polynomial
of degrees at most $2(t-1)+1$ in $d_{1},\ldots,d_{n}$ whose coefficients
of $d_{1}^{a_{1}}\cdots d_{n}^{a_{n}}$ are bounded similarly. Indeed,
the coefficient of $d_{1}^{a_{1}}\cdots d_{n}^{a_{n}}$ is bounded
by 
\begin{align*}
 & \sum_{\substack{m_{1},m_{2},m_{3}}
}\frac{500n^{m_{1}-1}(c(m_{2}+1))^{cm_{2}+3}(n+m_{2}-1)^{m_{2}+1}(n+m_{3}+1)^{m_{3}}(cm_{3})^{cm_{3}}}{a_{1}!\cdots a_{n}!}\\
 & \leqslant500(n+t)^{t}\sum_{\substack{m_{1},m_{2},m_{3}}
}\frac{(c(m_{2}+m_{3}))^{c(m_{2}+m_{3})+3}}{a_{1}!\cdots a_{n}!}\\
 & \leqslant\frac{(n+t)^{t}}{a_{1}!\cdots a_{n}!}(c(t-1))^{c(t-1)+6},
\end{align*}
where the summations are over $m_{1},m_{2},m_{3}$ for which $\sum_{i=1}^{3}m_{i}=t$,
$0\leqslant m_{2},m_{3}\leqslant k$ and $1\leqslant m_{1}\leqslant k+1$,
and we use the fact that $m_{2}+m_{3}\leqslant t-1$ as $m_{1}\geqslant1$.
The conclusion now follows since $c>600$. 
\end{proof}
\begin{lem}
\label{lem:S_2}Suppose that for some $k\geqslant1$, each of $A1(r)$,
$A2(r)$, $A3(r),$ and $A4(r)$ hold for all $r\leqslant k$, then
there exist $\left\{ w_{\mathbf{d},n}^{i}\right\} _{i\geqslant1}$
and a polynomial $P_{n}^{k+1}$ of degree at most $2(k+1)+1$ such
that for any $g>(n+k+2)(c(k+2))^{c}$,
\[
\left|S_{2}-\sum_{i=1}^{k+1}\frac{w_{\d,n}^{i}}{g^{i}}\right|\leqslant\frac{P_{n}^{k+1}(d_{1},\ldots,d_{n})}{g^{k+2}},
\]
and for any $t_{1},\ldots,t_{n}\in\mathbb{Z}_{\geq0}$ the coefficient
$\left|[d_{1}^{t_{1}}\cdots d_{n}^{t_{n}}]P_{n}^{k+1}\right|\leqslant\frac{(c(k+2))^{c(k+2)-4}(n+k+2)^{k+1}}{t_{1}!\cdots t_{n}!}$,
and the $w_{\mathbf{d},n}^{i}$ are majorised by polynomials $v_{n}^{i}$
of degree at most $2(i-1)+1$ whose coefficients satisfy $\left|[d_{1}^{t_{1}}\cdots d_{n}^{t_{n}}]v_{n}^{i}\right|\leqslant\frac{(ci)^{ci-4}(n+i)^{i-1}}{t_{1}!\cdots t_{n}!}$
.
\end{lem}

\begin{proof}
By (\ref{eq:recurs(B)}), $\tilde{B}_{\d,g,n}$ is given by 
\[
16\sum_{\ell=0}^{d_{0}}\sum_{k_{1}+k_{2}=\ell+d_{1}-2}(a_{\ell}-a_{\ell-1})\left[\tau_{k_{1}}\tau_{k_{2}}\prod_{i\neq1}\tau_{d_{i}}\right]_{g-1,n+1},
\]
where we write $a_{-1}\equiv0$. We deal with the case when $|\d|\leqslant g$,
with the other case following in a very similar manner to that in
the proof of Lemma \ref{lem:S_1}. From $A1(k)$ and Lemma \ref{lem:shift}
we write 
\begin{align*}
 & \left|\frac{\tilde{B}_{\d,g,n}}{V_{g-1,n+1}}-16\sum_{t=0}^{k}\sum_{\ell=0}^{\infty}(a_{\ell}-a_{\ell-1})\sum_{i+j=d_{1}+\ell-2}\frac{\tilde{b}_{\mathbf{d}(i,j),n+1}^{t}}{g^{t}}\right|\\
 & \leqslant\underbrace{\left|16\sum_{\ell=0}^{\infty}\sum_{i+j=\ell+d_{1}-2}(a_{\ell}-a_{\ell-1})\left(\frac{\left[\tau_{i}\tau_{j}\prod_{i\neq1}\tau_{d_{i}}\right]_{g-1,n+1}}{V_{g-1,n+1}}-\sum_{t=0}^{k}\frac{\tilde{b}_{\mathbf{d}(i,j),n+1}^{t}}{g^{t}}\right)\right|}_{(1)}\\
 & +\underbrace{\left|16\sum_{\ell=d_{0}+1}^{\infty}\sum_{i+j=\ell+d_{1}-2}(a_{\ell}-a_{\ell-1})\frac{\left[\tau_{i}\tau_{j}\prod_{i\neq1}\tau_{d_{i}}\right]_{g-1,n+1}}{V_{g-1,n+1}}\right|}_{(2)},
\end{align*}
where $\mathbf{d}(i,j)=(i,j,d_{2},\ldots,d_{n})$ and $\tilde{b}_{\mathbf{d}(i,j),n+1}^{t}=\sum_{p=1}^{t}{t-1 \choose p-1}b_{\mathbf{d}(i,j),n+1}^{p}.$

We begin with the tail term (2) noting that $\left[\tau_{i}\tau_{j}\prod_{i\neq1}\tau_{d_{i}}\right]_{g-1,n+1}\leqslant V_{g-1,n+1}$,
$a_{\ell}-a_{\ell-1}\leqslant4^{-\ell}$ and the number of non-negative
$i$ and $j$ whose sum is equal to $d_{1}+\ell-2$ is precisely $d_{1}+\ell-1$.
Thus,
\[
(2)\leqslant16\sum_{\ell=d_{0}+1}^{\infty}(d_{1}-1+\ell)4^{-\ell}\leqslant\frac{1}{g^{k+2}},
\]
where we use $d_{0}>g\geqslant|\mathbf{d}|$ and $g>ck^{2}$. For
(1), we use $A1(k)$ and Lemma \ref{lem:shift} to obtain 
\begin{align*}
(1) & \leqslant\frac{42e}{g^{k+1}}\sum_{\ell=0}^{\infty}\sum_{i+j=\ell+d_{1}-2}(a_{\ell}-a_{\ell-1})\left(Q_{n+1}^{k}(i,j,d_{2},\ldots,d_{n})+k^{2}b_{\mathbf{d}(i,j),n+1}^{k}\right).
\end{align*}
We first deal with the term on the right corresponding to the polynomial
$Q_{n+1}^{k}(i,j,d_{2},\ldots,d_{n})$ where we again denote its coefficients
by $p_{(t_{1},\ldots,t_{n+1})}^{k,n+1}$. We have 
\begin{align*}
 & \sum_{\ell=0}^{\infty}\sum_{i+j=\ell+d_{1}-2}(a_{\ell}-a_{\ell-1})Q_{n+1}^{k}(i,j,d_{2},\ldots,d_{n})\\
 & =\sum_{\substack{0\leqslant\tilde{t},t_{1},\ldots,t_{n}\leqslant2(k+1)\\
\tilde{t}_{1}+\sum_{p=1}^{n}t_{p}\leqslant2(k+1)
}
}p_{(\tilde{t}_{1},t_{1},\ldots,t_{n})}^{k,n+1}\sum_{\ell=0}^{\infty}\sum_{i=0}^{\ell+d_{1}-2}(a_{\ell}-a_{\ell-1})i^{\tilde{t}_{1}}(\ell+d_{1}-2-i)^{t_{1}}d_{2}^{t_{2}}\cdots d_{n}^{t_{n}}\\
 & =\sum_{\substack{0\leqslant\tilde{t}_{1},t_{1},\ldots,t_{n}\leqslant2(k+1)\\
\tilde{t}_{1}+\sum_{p=1}^{n}t_{p}\leqslant2(k+1)
}
}p_{(\tilde{t}_{1},t_{1},\ldots,t_{n})}^{k,n+1}\sum_{i=0}^{\infty}\sum_{\ell=\max\left\{ 0,i-(d_{1}-2)\right\} }^{\infty}(a_{\ell}-a_{\ell-1})i^{\tilde{t}_{1}}(\ell+d_{1}-2-i)^{t_{1}}d_{2}^{t_{2}}\cdots d_{n}^{t_{n}}.
\end{align*}
The inner double sum splits into 
\begin{align*}
 & \underbrace{\sum_{i=0}^{d_{1}-2}\sum_{\ell=0}^{\infty}(a_{\ell}-a_{\ell-1})i^{\tilde{t}_{1}}(\ell+d_{1}-2-i)^{t_{1}}d_{2}^{t_{2}}\cdots d_{n}^{t_{n}}}_{(a)}\\
 & +\underbrace{\sum_{i=d_{1}-1}^{\infty}\sum_{\ell=i-(d_{1}-2)}^{\infty}(a_{\ell}-a_{\ell-1})i^{\tilde{t}_{1}}(\ell+d_{1}-2-i)^{t_{1}}d_{2}^{t_{2}}\cdots d_{n}^{t_{n}}.}_{(b)}
\end{align*}
For (a), since $i\leq d_{1}-2$ and $\ell\geq0$, we have 
\begin{align*}
\left(\ell+d_{1}-2\right)^{t_{1}+\tilde{t}_{1}} & =\left(\ell+d_{1}-2-i+i\right)^{t_{1}+\tilde{t}_{1}}\geqslant{t_{1}+\tilde{t}_{1} \choose t_{1}}\left(\ell+d_{1}-2-i\right)^{\tilde{t}_{1}}i^{t_{1}}.
\end{align*}
Thus, using Lemma \ref{lem:basic estimate},
\begin{align*}
(a) & \leqslant\frac{t_{1}!\tilde{t}_{1}!}{(t_{1}+\tilde{t}_{1})!}\sum_{i=0}^{d_{1}-2}\sum_{\ell=0}^{\infty}(a_{\ell}-a_{\ell-1})(\ell+d_{1})^{t_{1}+\tilde{t}_{1}}d_{2}^{t_{2}}\cdots d_{n}^{t_{n}}\\
 & \leqslant\frac{t_{1}!\tilde{t}_{1}!}{(t_{1}+\tilde{t}_{1})!}\sum_{p=0}^{t_{1}+\tilde{t}_{1}}{t_{1}+\tilde{t}_{1} \choose p}d_{1}^{t_{1}+\tilde{t}_{1}-p+1}\sum_{\ell=0}^{\infty}(a_{\ell}-a_{\ell-1})\ell^{p}d_{2}^{t_{2}}\cdots d_{n}^{t_{n}}\\
 & \leqslant2t_{1}!\tilde{t}_{1}!\sum_{p=0}^{t_{1}+\tilde{t}_{1}}\frac{1}{(t_{1}+\tilde{t}_{1}-p)!}d_{1}^{t_{1}+\tilde{t}_{1}-p+1}d_{2}^{t_{2}}\cdots d_{n}^{t_{n}}.
\end{align*}
Thus, the contribution from (a) to the upper bound of (1) is majorised
by a polynomial in $d_{1},\ldots,d_{n}$ of degree at most $2(k+1)+1$
whose coefficient of $d_{1}^{a_{1}}\cdots d_{n}^{a_{n}}$ for $\sum_{i=1}^{n}a_{i}\leqslant2(k+1)+1$
is bounded by 
\[
84e\sum_{a_{1}-1\leqslant t_{1}+\tilde{t}_{1}\leqslant2(k+1)-\sum_{i=2}^{n}a_{i}}\frac{t_{1}!\tilde{t}_{1}!}{(a_{1}-1)!}p_{(\tilde{t}_{1},t_{1},a_{2},\ldots,a_{n})}^{k,n+1}.
\]
By the inductive hypothesis of $A1(k)$, we have 
\[
p_{(\tilde{t}_{1},t_{1},a_{2},\ldots,a_{n})}^{k,n+1}\leqslant\frac{(c(k+1))^{c(k+1)}(n+k+2)^{k+1}}{t_{1}!\tilde{t}_{1}!a_{2}!\cdots a_{n}!},
\]
and so the coefficient is bounded by 
\begin{align*}
 & 84e\frac{a_{1}(c(k+1))^{c(k+1)}(n+k+2)^{k+1}}{a_{1}!\cdots a_{n}!}\sum_{a_{1}-1\leqslant t_{1}+\tilde{t}_{1}\leqslant2(k+1)-\sum_{i=2}^{n}a_{i}}1\\
 & \leqslant84e\frac{(2(k+1)+1)^{3}(c(k+1))^{c(k+1)}(n+k+2)^{k+1}}{a_{1}!\cdots a_{n}!}\\
 & \leqslant\frac{(c(k+2))^{c(k+1)+4}(n+k+2)^{k+1}}{a_{1}!\cdots a_{n}!},
\end{align*}
where on the last line we use that $c>600>84e$. We now turn to the
contribution of (b). We have 
\begin{align*}
(b) & =\sum_{i=d_{1}-1}^{\infty}\sum_{\ell=0}^{\infty}(a_{\ell+i-(d_{1}-2)}-a_{\ell+i-(d_{1}-2)-1})i^{\tilde{t}_{1}}\ell^{t_{1}}d_{2}^{t_{2}}\cdots d_{n}^{t_{n}}\\
 & \leqslant\sum_{i=d_{1}-1}^{\infty}4^{-i+(d_{1}-2)}i^{\tilde{t}_{1}}\sum_{\ell=0}^{\infty}4^{-\ell}\ell^{t_{1}}d_{2}^{t_{2}}\cdots d_{n}^{t_{n}}\leqslant2t_{1}!\sum_{i=d_{1}-1}^{\infty}4^{-i+(d_{1}-2)}i^{\tilde{t}_{1}}d_{2}^{t_{2}}\cdots d_{n}^{t_{n}}\\
 & =2t_{1}!\sum_{q=1}^{\infty}4^{-q}(q+d_{1}-2)^{\tilde{t}_{1}}d_{2}^{t_{2}}\cdots d_{n}^{t_{n}}\leqslant2t_{1}!\sum_{s=0}^{\tilde{t}_{1}}{\tilde{t}_{1} \choose s}d_{1}^{\tilde{t}_{1}-s}\sum_{q=1}^{\infty}4^{-q}q^{s}d_{2}^{t_{2}}\cdots d_{n}^{t_{n}}\\
 & \leqslant4t_{1}!\tilde{t}_{1}!\sum_{s=0}^{\tilde{t}_{1}}\frac{1}{(\tilde{t}_{1}-s)!}d_{1}^{\tilde{t}_{1}-s}d_{2}^{t_{2}}\cdots d_{n}^{t_{n}}.
\end{align*}
Thus once again, the contribution from (b) to the upper bound on (1)
is majorised by a polynomial in $d_{1},\ldots,d_{n}$ of degree at
most $2(k+1)$ whose coefficient of $d_{1}^{a_{1}}\cdots d_{n}^{a_{n}}$
for $\sum_{i=1}^{n}a_{i}\leqslant2(k+1)$ is bounded by
\[
168e\sum_{\substack{a_{1}\leqslant\tilde{t}_{1}\leqslant2(k+1)-\sum_{i=2}^{n}a_{i}\\
0\leqslant t_{1}\leqslant2(k+1)
}
}\frac{t_{1}!\tilde{t}_{1}!}{a_{1}!}p_{(\tilde{t}_{1},t_{1},a_{2}\ldots,a_{n})}^{k,n+1}.
\]
Again, by using the inductive hypothesis of $A1(k)$, we can bound
this by 
\begin{align*}
 & 168e\frac{(c(k+1))^{c(k+1)}(n+k+2)^{k+1}}{a_{1}!\cdots a_{n}!}\sum_{\substack{a_{1}\leqslant\tilde{t}_{1}\leqslant2(k+1)-\sum_{i=2}^{n}a_{i}\\
0\leqslant t_{1}\leqslant2(k+1)
}
}1\\
\leqslant & 168e\frac{(2(k+1))^{2}(c(k+1))^{c(k+1)}(n+k+2)^{k+1}}{a_{1}!\cdots a_{n}!}\\
\leqslant & \frac{(c(k+1))^{c(k+1)+3}(n+k+2)^{k+1}}{a_{1}!\cdots a_{n}!},
\end{align*}
where we use $c>600>168e$. The bound on (1) from the contribution
of the $k^{2}b_{\mathbf{d}(i,j),n+1}^{k}$ is entirely similar since
by the induction hypothesis for $A1(k)$ we may majorize $b_{\mathbf{d}(i,j),n+1}^{k}$
by a polynomial of degree at most $2k$ in the variables $i,j,d_{2},\ldots,d_{n}$.
In total, we can majorize (1) by a polynomial of degree at most $2(k+1)+1$
in the variables $d_{1},\ldots,d_{n}$ whose coefficient of $d_{1}^{a_{1}}\cdots d_{n}^{a_{n}}$
is bounded by 
\[
\frac{(c(k+2))^{c(k+1)+5}(n+k+2)^{k+1}}{a_{1}!\cdots a_{n}!}.
\]
An estimate for $S_{2}$ now holds by using Lemma \ref{lem:product-exp}
with $A3(k)$ and $A4(k)$. Define 
\[
B_{\mathbf{d},n}^{r}\eqdf16\sum_{\ell=0}^{\infty}(a_{\ell}-a_{\ell-1})\sum_{i+j=d_{1}+\ell-2}\tilde{b}_{\mathbf{d}(i,j),n+1}^{r},
\]
so that 
\[
B_{\mathbf{d},n}^{0}=16\sum_{\ell=0}^{\infty}(a_{\ell}-a_{\ell-1})(d_{1}+\ell-1)\leqslant32(d_{1}+1),
\]
Moreover, since $\tilde{b}_{\mathbf{d}(i,j),n+1}^{t}=\sum_{p=1}^{t}{t-1 \choose p-1}b_{\mathbf{d}(i,j),n+1}^{p}$
and $b_{\mathbf{d}(i,j),n+1}^{p}$ is majorised by a polynomial of
degree at most $2p$ in $i,j,d_{2},\ldots,d_{n}$ whose coefficients
of $i^{\tilde{t}_{1}}j^{t_{1}}d_{2}^{t_{2}}\cdots d_{n}^{t_{n}}$
are bounded by
\[
\frac{(cp)^{cp}(n+p+1)^{p}}{t_{1}!\tilde{t}_{1}!t_{2}!\cdots t_{n}!},
\]
 we have that $B_{\mathbf{d},n}^{r}$ is also majorised by a polynomial
of degree at most $2r+1$ in $d_{1},\ldots,d_{n}$ whose coefficients
of $d_{1}^{a_{1}}\cdots d_{n}^{a_{n}}$ are bounded by 
\[
\frac{(n+r+1)^{r}(cr)^{cr+4}}{a_{1}!\cdots a_{n}!}.
\]
The proof of this is identical to the bound for the contribution of
the $Q_{n+1}^{k}(i,j,d_{2},\ldots,d_{n})$ in (1) above by first splitting
into two contributions similar to (a) and (b) and then obtaining analogous
bounds. 

We now recall the expansion of $\frac{1}{4\pi^{2}\left(2g-2+n\right)}$
with coefficients of $\theta_{n}^{t}$ obtained when computing the
expansion of $S_{1}$ in Lemma \ref{lem:S_1}. Then, from Corollary
\ref{cor:product-with-zeroterms} using the estimates on the coefficients
$B_{\mathbf{d},n}^{m_{1}}$, $h_{n-1}^{m_{2}}$ (from $A3(k)$), $p_{n-1}^{m_{3}}$(from
$A4(k)$) and $\theta_{n}^{m_{4}}$,
\begin{align*}
\left|S_{2}-\sum_{t=1}^{k+1}\sum_{\substack{m_{1},m_{2},m_{3},m_{4}}
}\frac{B_{\mathbf{d},n}^{m_{1}}h_{n-1}^{m_{2}}p_{n-1}^{m_{3}}\theta_{n}^{m_{4}}}{g^{t}}\right| & \leqslant\frac{P_{k+1,n}^{(2)}(d_{1},\ldots,d_{n})}{g^{k+2}},
\end{align*}
where the summation is over $m_{1},m_{2},m_{3},m_{4}$ such that $\sum_{i=1}^{4}m_{i}=t$,
$0\leq m_{1},m_{2},m_{3}\leqslant k,$ and $1\leqslant m_{4}\leqslant k+1$,
and $P_{k+1,n}^{(2)}(d_{1},\ldots,d_{n})$ is a polynomial in $d_{1},\ldots,d_{n}$
of degree at most $2(k+1)+1$ and whose coefficient of $d_{1}^{a_{1}}\cdots d_{n}^{a_{n}}$
is bounded by 
\[
\frac{(c(k+1))^{c(k+1)+5}(n+k+2)^{k+1}}{a_{1}!\cdots a_{n}!}.
\]
 We also note that the coefficient of $g^{-t}$ for $t\geqslant1$
is majorised by a polynomial of degree at most $2(t-1)+1$ in $d_{1},\ldots,d_{n}$
since $m_{4}\geqslant1$ so that $m_{1},m_{2},m_{3}\leqslant t-1$,
whose coefficient of $d_{1}^{a_{1}}\cdots d_{n}^{a_{n}}$ is bounded
by 
\begin{align*}
 & \sum_{\substack{m_{1},m_{2},m_{3},m_{4}}
}\frac{500(n+m_{1}+1)^{m_{1}}(cm_{1})^{cm_{1}+4}}{a_{1}!\cdots a_{n}!}(n+m_{2})^{m_{2}}(cm_{2})^{cm_{2}}(n+m_{3}-1)^{m_{3}}(cm_{3})^{cm_{3}}n^{m_{4}-1}\\
 & \leqslant\frac{500(n+t)^{t-1}(c(t-1))^{3}}{a_{1}!\cdots a_{n}!}\sum_{\substack{m_{1},m_{2},m_{3},m_{4}}
}\left(c(m_{1}+m_{2}+m_{3})\right)^{c(m_{1}+m_{2}+m_{3})}\\
 & \leqslant\frac{(n+t)^{t-1}(c(t-1))^{c(t-1)+7}}{a_{1}!\cdots a_{n}!},
\end{align*}
where again the summation is over $m_{1},m_{2},m_{3},m_{4}$ as previously.
The stated expansion now holds since $c>600$.
\end{proof}
\begin{lem}
\label{lem:S_3}Suppose that for some $k\geqslant1$, each of $A1(r)$,
$A2(r)$, $A3(r),$ and $A4(r)$ hold for all $r\leqslant k$, then
there exist $\left\{ w_{\mathbf{d},n}^{i}\right\} _{i\geqslant1}$
and a polynomial $P_{n}^{k+1}$ of degree at most $2(k+1)+1$ such
that for any $g>(c(k+2))^{c}(n+k+2)$,
\[
\left|S_{3}-\sum_{i=1}^{k+1}\frac{w_{\d,n}^{i}}{g^{i}}\right|\leqslant\frac{P_{n}^{k+1}(d_{1},\ldots,d_{n})}{g^{k+2}},
\]
and for any $t_{1},\ldots,t_{n}\in\mathbb{Z}_{\geqslant0}$ the coefficient
$\left|[d_{1}^{t_{1}}\cdots d_{n}^{t_{n}}]P_{n}^{k+1}\right|\leqslant\frac{(c(k+2))^{c(k+2)-4}(n+k+2)^{k+2}}{t_{1}!\cdots t_{n}!}$,
and the $w_{\mathbf{d},n}^{i}$ are majorised by polynomials $v_{n}^{i}$
of degree at most $2(i-1)+1$ whose coefficients satisfy $\left|[d_{1}^{t_{1}}\cdots d_{n}^{t_{n}}]v_{n}^{i}\right|\leqslant\frac{(ci)^{ci-4}(n+i)^{i}}{t_{1}!\cdots t_{n}!}$
.
\end{lem}

\begin{proof}
By (\ref{eq:recrus(C)}) we have
\[
\tilde{C}_{\d,g,n}=16\sum_{\substack{g_{1}+g_{2}=g\\
I\sqcup J=\left\{ 2,\ldots,n\right\} 
}
}\sum_{\ell=0}^{d_{0}}\sum_{k_{1}+k_{2}=\ell+d_{1}-2}(a_{\ell}-a_{\ell-1})\left[\tau_{k_{1}}\prod_{i\in I}\tau_{d_{i}}\right]_{g_{1},|I|+1}\left[\tau_{k_{2}}\prod_{j\in J}\tau_{d_{j}}\right]_{g_{2},|J|+1}.
\]
We deal with the case when $|\d|\leqslant g$, with the other case
following in a very similar manner to that in the proof of Lemma \ref{lem:S_1}.
After dividing by $V_{g,n}$, the associated tail of this term (extending
$\ell$ up to $\infty$) is bounded above by 
\[
16\sum_{\ell=d_{0}+1}^{\infty}\sum_{k_{1}+k_{2}=\ell+d_{1}-2}4^{-\ell}\sum_{\substack{g_{1}+g_{2}=g\\
I\sqcup J=\left\{ 2,\ldots,n\right\} 
}
}\frac{V_{g_{1},|I|+1}V_{g_{2},|J|+1}}{V_{g,n}}.
\]
To deal with the innermost summation, we use \cite[Lemma 3.2, part 3]{MirzakhaniRandom}
followed by (\ref{eq:recurs-1A}), and then \cite[Lemma 5.2(i) and (iii)]{Mi.Zo2015}
to obtain

\begin{align*}
\sum_{\substack{g_{1}+g_{2}=g\\
I\sqcup J=\left\{ 2,\ldots,n\right\} 
}
}\frac{V_{g_{1},|I|+1}V_{g_{2},|J|+1}}{V_{g,n}} & \leqslant\frac{C}{g}\frac{V_{g,n+1}}{V_{g,n}}\sum_{\substack{g_{1}+g_{2}=g\\
I\sqcup J=\left\{ 2,\ldots,n\right\} 
}
}\frac{V_{g_{1},|I|+2}V_{g_{2},|J|+2}}{V_{g,n+1}}\\
 & =\frac{C}{6g}\frac{V_{g,n+1}}{V_{g,n}}\left(\frac{\left[\tau_{1}\tau_{0}^{n}\right]_{g,n+1}}{V_{g,n+1}}-\frac{V_{g-1,n+3}}{V_{g,n+1}}\right)\\
 & \leqslant\frac{C}{6g}\frac{V_{g,n+1}}{V_{g,n}}\left(\left|1-\frac{\left[\tau_{1}\tau_{0}^{n}\right]_{g,n+1}}{V_{g,n+1}}\right|-\left|\frac{V_{g-1,n+3}}{V_{g,n+1}}-1\right|\right)\\
 & \leqslant\frac{C'n}{g(2g+n-2)}\frac{V_{g,n+1}}{V_{g,n}}\leqslant\frac{C'n}{g},
\end{align*}
for some universal constant $C'$. Thus the tail bounds by 
\[
\frac{C'n}{g}\sum_{\ell=d_{0}+1}^{\infty}\sum_{k_{1}+k_{2}=\ell+d_{1}-2}4^{-\ell}\leqslant\frac{C'n}{g}\sum_{\ell=d_{0}+1}^{\infty}(\ell+d_{1}-1)4^{-\ell}\leqslant\frac{1}{g^{k+2}},
\]

since $d_{0}>g\geqslant|\mathbf{d}|$ and $g>ck^{2}$. For the main
term, we note that since by Lemma \ref{lem:WP-sum-product-bound}
and Remark \ref{rem:volume-product-cylinder} we have 
\[
\sum_{\substack{g_{1}+g_{2}=g\\
2g_{1}+|I|\geqslant k+2\\
2g_{2}+|J|\geqslant k+2
}
}\frac{\left[\tau_{k_{1}}\prod_{i\in I}\tau_{d_{i}}\right]_{g_{1},|I|+1}\left[\tau_{k_{2}}\prod_{j\in J}\tau_{d_{j}}\right]_{g_{2},|J|+1}}{V_{g,n}}\leqslant C\frac{V_{g,n+1}}{g^{k+3}V_{g,n}}\leqslant\frac{C}{g^{k+2}},
\]
the only terms that will contribute to an expansion up to an order
$g^{-(k+2)}$ are when either $2g_{1}+|I|<k+2$ or $2g_{2}+|J|<k+2$.
The main term hence becomes

\begin{align}
 & \underbrace{\sum_{\substack{g_{1},i\\
2\leqslant2g_{1}+i\leqslant k+1\\
i\leqslant n-1
}
}\sum_{\substack{I\subseteq\left\{ 2,\ldots,n\right\} \\
|I|=i
}
}\sum_{\ell=0}^{\infty}\sum_{k_{1}+k_{2}=\ell+d_{1}-2}(a_{\ell}-a_{\ell-1})\frac{\left[\tau_{k_{1}}\prod_{i\in I}\tau_{d_{i}}\right]_{g_{1},i+1}\left[\tau_{k_{2}}\prod_{j\notin I}\tau_{d_{j}}\right]_{g-g_{1},n-i}}{V_{g,n}}}_{(a)}\label{eq:C-term}\\
 & +\underbrace{\sum_{\substack{g_{2},j\\
2\leqslant2g_{2}+j\leqslant k+1\\
2g_{2}+j\leqslant2g+n-2-(k+1)\\
0\leqslant j\leqslant n-1
}
}\sum_{\substack{J\subseteq\left\{ 2,\ldots,n\right\} \\
|J|=j
}
}\sum_{\ell=0}^{\infty}\sum_{k_{1}+k_{2}=\ell+d_{1}-2}(a_{\ell}-a_{\ell-1})\frac{\left[\tau_{k_{1}}\prod_{i\notin J}\tau_{d_{i}}\right]_{g-g_{2},n-j}\left[\tau_{k_{2}}\prod_{j\in J}\tau_{d_{j}}\right]_{g_{2},j+1}}{V_{g,n}}}_{(b)}.\nonumber 
\end{align}
In term (a), the asymptotic expansion in $g$ comes from $\frac{\left[\tau_{k_{2}}\prod_{j\notin I}\tau_{d_{j}}\right]_{g-g_{1},n-i}}{V_{g,n}}$
and correspond to when $2g_{1}+|I|<k+2$, whereas in (b) they come
from $\frac{\left[\tau_{k_{1}}\prod_{i\in I}\tau_{d_{i}}\right]_{g-g_{2},n-j}}{V_{g,n}}$
and correspond to when $2g_{2}+|J|<k+2$ and $2g_{1}+|I|\geqslant k+2$.
We begin with (a) and let $m\eqdf m(g_{1},i)$ be the smallest integer
such that $(m-1)\left\lfloor \frac{1}{2}(n-i)\right\rfloor <g_{1}\leqslant m\left\lfloor \frac{1}{2}(n-i)\right\rfloor $.
Then, denote by $\rho_{p}\eqdf\rho(p,g_{1},i)=g_{1}-(p-1)\lfloor\frac{1}{2}(n-i)\rfloor$
so we have the expansion
\begin{align}
 & \frac{\left[\tau_{k_{2}}\prod_{j\notin I}\tau_{d_{j}}\right]_{g-g_{1},n-i}}{V_{g,n}}=\frac{\left[\tau_{k_{2}}\prod_{j\notin I}\tau_{d_{j}}\right]_{g-g_{1},n-i}}{V_{g-g_{1},n-i}}\prod_{p=1}^{m-1}\prod_{s=0}^{\left\lfloor \frac{1}{2}(n-i)\right\rfloor -1}\frac{V_{g-\rho_{p}+s,n-i-2s}}{V_{g-\rho_{p}+s+1,n-i-2s-2}}\label{eq:ginc1}\\
 & \cdot\prod_{s=0}^{\rho_{m}-1}\frac{V_{g-\rho_{m}+s,n-i-2s}}{V_{g-\rho_{m}+s+1,n-i-2s-2}}\label{eq:ginc2}\\
 & \cdot\prod_{p=0}^{m-2}\prod_{s=0}^{2\left\lfloor \frac{1}{2}(n-i)\right\rfloor -1}\frac{4\pi^{2}(2(g-\rho_{p+1})+n-i+2+s)V_{g-\rho_{p+2},n-i-2\lfloor\frac{1}{2}(n-i)\rfloor+s}}{V_{g-\rho_{p+2},n-i-2\lfloor\frac{1}{2}(n-i)\rfloor+s+1}}\label{eq:ninc1}\\
 & \cdot\prod_{s=0}^{2\rho_{m}+i-1}\frac{4\pi^{2}(2(g-\rho_{m})+n-i+2+s)V_{g,n-i-2\rho_{m}+s}}{V_{g,n-i-2\rho_{m}+s+1}}\label{eq:ninc2}\\
 & \cdot\prod_{s=0}^{2g_{1}-1+i}\frac{1}{4\pi^{2}(2(g-g_{1})+n-i+2+s)}.\label{correction}
\end{align}
This expansion is similar to that used in \cite[Equation (4.4)]{Mi.Zo2015}
but we take care to ensure that the number of cusps in any of the
volumes never exceeds $n$ at any step, leading to a more complicated
formula. Let $J=\left\{ 2,\ldots,n\right\} \setminus I\eqdf\left\{ j_{1},\ldots,j_{n-i-1}\right\} $
and $\mathbf{d}(k_{2},J)\eqdf(d_{j_{1}},\ldots,d_{j_{n-i-1}},k_{2})$.
Using Lemma \ref{lem:shift} and $A1(r)$ for any $r\leqslant k$,
we have for the first term on the right-hand side in (\ref{eq:ginc1})
that
\[
\left|\frac{\left[\tau_{k_{2}}\prod_{j\notin I}\tau_{d_{j}}\right]_{g-g_{1},n-i}}{V_{g-g_{1},n-i}}-\sum_{t=0}^{r}\frac{g_{1}^{t}\tilde{b}_{\mathbf{d}(k_{2},J),n-i}^{t}}{g^{t}}\right|\leqslant\frac{1}{g^{r+1}}\left(3Q_{n-i}^{r}(\mathbf{d}(k_{2},J))+3e^{g_{1}}r^{2}q_{n-i}^{r}(\mathbf{d}(k_{2},J))\right),
\]
where $\tilde{b}_{\mathbf{d}(k_{2},J),n-i}^{t}=\sum_{p=1}^{t}{t-1 \choose p-1}g_{1}^{-p}b_{\mathbf{d}(k_{2},J),n-i}^{p}$
and $Q_{n-i}^{r}$ and $q_{n-i}^{r}$ are the polynomials from $A1(r)$
of degrees at most $2(r+1)$ and $2r$ respectively. Due to the bounds
on the coefficients of $Q_{n-i}^{r}$ and $q_{n-i}^{r}$, the error
\[
E_{r,\mathbf{d}(k_{2},J),i}^{(1)}\eqdf3Q_{n-i}^{r}(\mathbf{d}(k_{2},J))+3e^{g_{1}}r^{2}q_{n-i}^{r}(\mathbf{d}(k_{2},J))
\]
is a polynomial of degree at most $2(r+1)$ in $d_{j_{1}},\ldots,d_{j_{n-i-1}},k_{2}$
with coefficient of $\left(\prod_{p=1}^{n-i-1}d_{j_{p}}^{a_{p}}\right)k_{2}^{a_{n-i}}$
being bounded by
\[
6e^{g_{1}}\frac{(c(r+1))^{c(r+1)}(n+r+1)^{r+1}}{a_{1}!\cdots a_{n-i}!}.
\]
 Likewise, for fixed 
\[
(p,s)\in\mathcal{B}\eqdf\left\{ (p,s):0\leqslant p\leqslant m-2,0\leqslant s\leqslant2\left\lfloor \frac{1}{2}(n-i)\right\rfloor -1\right\} ,
\]
 from $A3(r)$ for $r\leqslant k$ and Lemma \ref{lem:shift}, we
get for terms appearing in (\ref{eq:ninc1}) that
\begin{align*}
 & \left|\frac{4\pi^{2}(2(g-\rho_{p+1})+n-i+2+s)V_{g-\rho_{p+2},n-i-2\lfloor\frac{1}{2}(n-i)\rfloor+s}}{V_{g-\rho_{p+2},n-i-2\lfloor\frac{1}{2}(n-i)\rfloor+s+1}}-\sum_{t=0}^{r}\frac{\rho_{p+2}^{t}\tilde{h}_{n-i-2\lfloor\frac{1}{2}(n-i)\rfloor+s+1}^{t}}{g^{t}}\right|\\
 & \leqslant\frac{1}{g^{r+1}}E_{r,p,s,i}^{(2)},
\end{align*}
where
\[
E_{r,p,s,i}^{(2)}\eqdf3000e^{g_{1}}\left(c(r+1)\right)^{c(r+1)}(n+r+2)^{r+1},
\]
and $\tilde{h}_{n-i-2\lfloor\frac{1}{2}(n-i)\rfloor+s+1}^{t}=\sum_{p=1}^{t}{t-1 \choose p-1}\rho_{p+2}^{-p}h_{n-i-2\lfloor\frac{1}{2}(n-i)\rfloor+s+1}^{p}$.
For terms appearing in (\ref{eq:ninc2}), no asymptotic expansion
base shifting is needed and so we obtain directly from $A3(r)$ that
for $s=0,\ldots,2\rho_{m}+i-1$,
\[
\left|\frac{4\pi^{2}(2(g-\rho_{m})+n-i+2+s)V_{g,n-i-2\rho_{m}+s}}{V_{g,n-i-2\rho_{m}+s+1}}-\sum_{t=0}^{r}\frac{h_{n-i-2\rho_{m}+s+1}^{t}}{g^{t}}\right|\leqslant\frac{500(c(r+1))^{c(r+1)}(n+r+2)^{r+1}}{g^{r+1}}.
\]
 For terms in the indexed product in (\ref{eq:ginc1}) and (\ref{eq:ginc2}),
we set 
\[
\mathcal{C}\eqdf\left\{ (p.s):\left(1\leqslant p\leqslant m-1,0\leqslant s\leqslant\left\lfloor \frac{1}{2}(n-i)\right\rfloor -1\right)\vee\left(p=m,0\leqslant s\leqslant\rho_{m}-1\right)\right\} .
\]
Then by $A4(r)$ for $r\leqslant k$ and Lemma \ref{lem:shift}, for
any $(p,s)\in\mathcal{C}$, 
\begin{align*}
 & \left|\frac{V_{g-\rho_{p}+s,n-i-2s}}{V_{g-\rho_{p}+s+1,n-i-2s-2}}-\sum_{t=0}^{r}\frac{(\rho_{p}-s-1)^{t}\tilde{p}_{n-i-2s-2}^{t}}{g^{t}}\right|\\
 & \leqslant\frac{1}{g^{r+1}}\left(3(c(r+1))^{c(r+1)}(n+r+1)^{r+1}+3e^{g_{1}}r^{2}(cr)^{cr}(n+r)^{r}\right)\\
 & \leqslant\frac{1}{g^{r+1}}E_{r,p,s,i}^{(3)},
\end{align*}
where 
\[
E_{r,p,s,i}^{(3)}\eqdf6e^{g_{1}}(c(r+1))^{c(r+1)}(n+r+1-i)^{r+1},
\]
and $\tilde{p}_{n-i-2s-2}^{t}=\sum_{q=1}^{t}{t-1 \choose q-1}(\rho_{p}-s-1)^{-q}p_{n-i-2s-2}^{q}.$ 

Finally, by Taylor expansion,
\begin{align*}
 & \prod_{s=0}^{2g_{1}-1+i}\frac{1}{4\pi^{2}(2(g-g_{1})+n-i+2+s)}=\frac{1}{(8\pi^{2}g)^{2g_{1}+i}}\prod_{s=1}^{2g_{1}+i}\frac{1}{1-\frac{g_{1}-\frac{1}{2}(n-i+1+s)}{g}}\\
 & =\frac{1}{(8\pi^{2})^{2g_{1}+i}}\sum_{p=2g_{1}+i}^{\infty}\frac{1}{g^{p}}\sum_{\substack{m_{1},\ldots,m_{2g_{1}+1}\\
\sum m_{q}=p-(2g_{1}+i)\\
m_{q}\geqslant0
}
}\prod_{s=1}^{2g_{1}+i}\left(g_{1}-\frac{1}{2}(n-i+1+s)\right)^{m_{s}}.
\end{align*}
So for $t\geqslant2g_{1}+i-1$, if we set $\theta_{n}^{t}$
\[
\theta_{n}^{t}\eqdf\frac{1}{(8\pi^{2})^{2g_{1}+i}}\sum_{\substack{m_{1},\ldots,m_{2g_{1}+1}\\
\sum m_{q}=t-(2g_{1}+i)\\
m_{q}\geqslant0
}
}\prod_{s=1}^{2g_{1}+i}\left(g_{1}-\frac{1}{2}(n-i+1+s)\right)^{m_{s}},
\]
then for any $2g_{1}+i\leqslant s\leqslant k+1$, we have the asymptotic
expansion 
\begin{align*}
\left|\prod_{s=0}^{2g_{1}+i}\frac{1}{4\pi^{2}(2(g-g_{1})+n-i+2+s)}-\sum_{t=2g_{1}+i}^{s}\frac{\theta_{n}^{t}}{g^{t}}\right| & \leqslant\sum_{t=s+1}^{\infty}\frac{\theta_{n}^{t}}{g^{t}}.
\end{align*}
The individual coefficients $\theta_{n}^{t}$ satisfy the bound 
\begin{align*}
|\theta_{n}^{t}| & \leqslant\frac{1}{(8\pi^{2})^{2g_{1}+i}}\sum_{\substack{m_{1},\ldots,m_{2g_{1}+1}\\
\sum m_{q}=t-(2g_{1}+i)\\
m_{q}\geq0
}
}\left|g_{1}-1-\frac{1}{2}(n-i)\right|^{t-(2g_{1}+i)}\\
 & \leqslant\frac{1}{(8\pi^{2})^{2g_{1}+i}}{t-1 \choose 2g_{1}+i-1}|g_{1}-1|^{t-(2g_{1}+i)}(n-i)^{t-(2g_{1}+i)}.
\end{align*}
Similarly, the error satisfies the bound
\begin{align*}
g^{s+1}\sum_{t=s+1}^{\infty}\frac{\theta_{n}^{t}}{g^{t}} & \leqslant\frac{g^{s+1}}{(8\pi^{2})^{2g_{1}+i}}\sum_{t=s+1}^{\infty}\frac{1}{g^{t}}\sum_{\substack{m_{1},\ldots,m_{2g_{1}+1}\\
\sum m_{q}=t-(2g_{1}+i)\\
m_{q}\geqslant0
}
}\left|g_{1}-1-\frac{1}{2}(n-i)\right|^{t-(2g_{1}+i)}\\
 & =\frac{g^{s+1}}{(8\pi^{2})^{2g_{1}+i-1}}\sum_{t=s+1}^{\infty}\frac{1}{g^{t}}{t-1 \choose 2g_{1}+i-1}\left|g_{1}-1-\frac{1}{2}(n-i)\right|^{t-(2g_{1}+i)}\\
 & =\frac{\left|g_{1}-1-\frac{1}{2}(n-i)\right|^{s+1-(2g_{1}+i)}}{(8\pi^{2})^{2g_{1}+i}}\sum_{t=0}^{\infty}{t+s \choose 2g_{1}+i-1}\frac{\left|g_{1}-1-\frac{1}{2}(n-i)\right|^{t}}{g^{t}}\\
 & \leqslant\frac{\left|g_{1}-1-\frac{1}{2}(n-i)\right|^{s+1-(2g_{1}+i)}}{\left(8\pi^{2}\right)^{2g_{1}+i}}{s \choose 2g_{1}+i-1}\sum_{t=0}^{\infty}{t+s \choose s}\left(\frac{\left|g_{1}-1-\frac{1}{2}(n-i)\right|}{g}\right)^{t}\\
 & \leqslant\frac{\left|g_{1}-1-\frac{1}{2}(n-i)\right|^{s+1-(2g_{1}+i)}}{\left(8\pi^{2}\right)^{2g_{1}+i-1}}{s \choose 2g_{1}+i-1}\exp\left(\frac{s\left|g_{1}-1-\frac{1}{2}(n-i)\right|}{g-\left|g_{1}-1-\frac{1}{2}(n-i)\right|}\right)\\
 & \leqslant\frac{3\left|g_{1}-1-\frac{1}{2}(n-i)\right|^{s+1-(2g_{1}+i)}}{\left(8\pi^{2}\right)^{2g_{1}+i-1}}{s \choose 2g_{1}+i-1}\eqdf E_{s,i}^{(4)},
\end{align*}
where the second inequality follows from
\begin{align*}
{t+s \choose 2g_{1}+i-1} & =\frac{(t+s)!}{t!s!}\frac{s!}{(2g_{1}+i-1)!}\frac{t!}{(t+s-(2g_{1}+i-1))!}\\
 & =\frac{(t+s)!}{t!s!}\frac{s!}{(2g_{1}+i-1)!}\frac{1}{(s-(2g_{1}+i-1))!}\frac{t!}{\prod_{q=1}^{t}(q+s-(2g_{1}+i-1))}\\
 & \leqslant\frac{(t+s)!}{t!s!}\frac{s!}{(2g_{1}+i-1)!(s-(2g_{1}+i-1))!},
\end{align*}
and on the penultimate line we have used $\sum_{t=0}^{\infty}{t+a \choose a}\frac{1}{b^{t}}=\left(1+\frac{1}{b-1}\right)^{a}$,
and on the final line we use that $g>(s+1)\left|g_{1}-1-\frac{1}{2}(n-i)\right|$
since $g_{1}(s+1)\leqslant(k+2)^{2}$, $(n-i)(s+1)\leq n(k+2)$ and
$(ck)^{c}(n+k)<g$ with $c>600$ from the induction hypothesis. 

Treating $\prod_{s=0}^{2g_{1}-1+i}\frac{1}{4\pi^{2}(2(g-g_{1})+n-i+2+s)}$
as a single term since we have a full asymptotic expansion for it,
the number of terms in (\ref{eq:ginc1})-(\ref{correction}) is $3g_{1}+i+2$.
By Corollary \ref{cor:product-with-zeroterms}, taking $r=2g_{1}+i-1$
and $s=k-r$ we obtain (where we recall the definitions of $m=m(g_{1},i)$
and $\rho_{p}$ in equations (\ref{eq:ginc1})-(\ref{correction}))
\begin{align*}
 & \left|\frac{\left[\tau_{k_{2}}\prod_{j\notin I}\tau_{d_{j}}\right]_{g-g_{1},n-i}}{V_{g,n}}-\sum_{t=2g_{1}+i}^{k+1}\frac{1}{g^{t}}\sum_{\mathrm{indices}}A_{\alpha_{1}}\prod_{(p,s)\in\mathcal{B}}B_{p,s}^{\beta_{p,s}}\prod_{(p,s)\in\mathcal{C}}C_{p,s}^{\gamma_{p,s}}\prod_{j=0}^{2\rho_{m}+i-1}h_{n-i-2\rho_{m}+j+1}^{t_{j}}\theta_{n}^{\alpha_{2}}\right|\\
 & \leqslant\frac{E_{\mathbf{d}(k_{2},J),n,g_{1}}^{k}}{g^{k+2}}.
\end{align*}
where the summation is over all indices $\alpha_{1},\beta_{p,s}:(p,s)\in\mathcal{B},\gamma_{p,s}:(p,s)\in\mathcal{C},t_{0},\ldots,t_{2\rho_{m}+i-1},\alpha_{2}$
such that
\[
\alpha_{1}+\alpha_{2}+\sum_{j}t_{j}+\sum_{(p,s)}\beta_{p,s}+\sum_{(p,s)}\gamma_{p,s}=t,
\]
and 
\begin{align*}
0 & \leqslant\alpha_{1},\beta_{p,s}\gamma_{p,s},t_{j}\leqslant t-(2g_{1}+i),
\end{align*}
\[
2g_{1}+i\leqslant\alpha_{2}\leqslant k+1.
\]
\begin{align*}
A_{\alpha_{1}} & =g_{1}^{\alpha_{1}}\tilde{b}_{\mathbf{d}(k_{2},J),n-i}^{\alpha_{1}},\\
B_{p,s}^{\beta} & =\rho_{p+2}^{\beta}\tilde{h}_{n-i-2\lfloor\frac{1}{2}(n-i)\rfloor+s+1}^{\beta},\\
C_{p,s}^{\gamma} & =(\rho_{p}-s-1)^{\gamma}\tilde{p}_{n-i-2s-2}^{\gamma}.
\end{align*}
Note that $A_{0}=B_{p,s}^{0}=C_{p,s}^{0}=h_{n-i-2\rho_{m}+j+1}^{0}=1$
by definition. The error $E_{\mathbf{d}(k_{2},J),n,g_{1}}^{k}$ is
determined by Corollary \ref{cor:product-with-zeroterms}. It follows
that if for $0\leqslant t\leqslant2g_{1}+i-1$, we set $\Xi_{I,k_{2},g_{1},n}^{t}\eqdf0$
and for $2g_{1}+i\leqslant t\leqslant k+1$,
\[
\Xi_{I,k_{2},g_{1},n}^{t}\eqdf\sum_{\mathrm{indices}}A_{\alpha_{1}}\prod_{(p,s)\in\mathcal{B}}B_{p,s}^{\beta_{p,s}}\prod_{(p,s)\in\mathcal{C}}C_{p,s}^{\gamma_{p,s}}\prod_{j=0}^{2\rho_{m}+i-1}h_{n-i-2\rho_{m}+j+1}^{t_{j}}\theta_{n}^{\alpha_{2}},
\]
where the summation is over all indices as above, then
\begin{equation}
\left|(a)-\sum_{t=0}^{k+1}\frac{1}{g^{t}}\sum_{\substack{g_{1},i\\
2\leqslant2g_{1}+i\leqslant k+1\\
i\leqslant n-1
}
}\sum_{\substack{I\subseteq\left\{ 2,\ldots n\right\} \\
|I|=i
}
}\sum_{\ell=0}^{\infty}\sum_{k_{1}+k_{2}=\ell+d_{1}-2}(a_{\ell}-a_{\ell-1})\left[\tau_{k_{1}}\prod_{i\in I}\tau_{d_{i}}\right]_{g_{1},i+1}\Xi_{I,k_{2},g_{1},n}^{t}\right|\leq\frac{P_{k,n}^{(a)}(d_{1},\ldots,d_{n})}{g^{k+2}},\label{eq:coeff-(a)}
\end{equation}
where $P_{k,n}^{(a)}(d_{1},\ldots,d_{n})$ is a polynomial in $d_{1},\ldots,d_{n}$
of degree at most $2(k+1)$ whose coefficient of $d_{1}^{a_{1}}\cdots d_{n}^{a_{n}}$
is bounded by
\[
\frac{(n+k+2)^{k+2}}{a_{1}!\cdots a_{n}!}(c(k+2))^{c(k+2)-5},
\]
and the coefficient of $g^{-t}$ in the asymptotic expansion of (a)
is zero for $t=0,1$ and a polynomial in $d_{1},\ldots,d_{n}$ of
degree at most $2(t-2)$ whose coefficient of $d_{1}^{a_{1}}\cdots d_{n}^{a_{n}}$
is bounded by
\[
\frac{(ct)^{ct-5}(n+t)^{t}}{a_{1}!\cdots a_{n}!},
\]
for $t\geqslant2$. Note that for fixed $t$, the summation over the
$g_{i},i$ is such that $2g_{1}+i\leqslant t$ because $\Xi_{I,k_{2},g_{1},n}^{t}\equiv0$
when $2g_{1}+i>t$, this is also the reason why the coefficients of
$g^{0}$ and $g^{-1}$ are zero. 

The proof of the coefficient bounds for $P_{k,n}^{(a)}(d_{1},\ldots,d_{n})$
follows from an estimate on the coefficient bounds of the polynomial
$E_{\mathbf{d}(k_{2},J),n,g_{1}}^{k}$ that we will prove later. Namely,
that \textbf{$E_{\mathbf{d}(k_{2},J),n,g_{1}}^{k}$ }is a polynomial
in $d_{j_{1}},\ldots,d_{j_{n-i-1}},k_{2}$ of degree at most $2(k+1)$
whose coefficient of $\left(\prod_{p=1}^{n-i-1}d_{j_{p}}^{a_{p}}\right)k_{2}^{a_{n-i}}$
is bounded by 
\[
\frac{(n+k+2)^{k+2-(2g_{1}+i)}}{a_{1}!\cdots a_{n-i}!}\left(c(k+2)\right)^{c(k+2)-(c-4)(2g_{1}+i)}.
\]
Indeed, let $\xi_{(t_{1},\ldots,t_{n-i})}^{t,I,k_{2},g_{1},n}$ be
the coefficient of $\left(\prod_{p=1}^{n-i-1}d_{j_{p}}^{t_{p}}\right)k_{2}^{t_{n-i}}$
in $E_{\mathbf{d}(k_{2},J),n,g_{1}}^{k}$. Then, 
\begin{align*}
 & \left|(a)-\sum_{t=0}^{k+1}\frac{1}{g^{t}}\sum_{\substack{g_{1},i\\
2\leqslant2g_{1}+i\leqslant k+1\\
i\leqslant n-1
}
}\sum_{\substack{I\subseteq\left\{ 2,\ldots n\right\} \\
|I|=i
}
}\sum_{\ell=0}^{\infty}\sum_{k_{1}+k_{2}=\ell+d_{1}-2}(a_{\ell}-a_{\ell-1})\left[\tau_{k_{1}}\prod_{i\in I}\tau_{d_{i}}\right]_{g_{1},i+1}\Xi_{I,k_{2},g_{1},n}^{t}\right|\\
 & \leqslant\sum_{\substack{g_{1},i\\
2\leqslant2g_{1}+i\leqslant k+1\\
i\leqslant n-1
}
}\sum_{\substack{I\subseteq\left\{ 2,\ldots n\right\} \\
|I|=i
}
}\sum_{\ell=0}^{\infty}\sum_{k_{1}+k_{2}=\ell+d_{1}-2}(a_{\ell}-a_{\ell-1})\left[\tau_{k_{1}}\prod_{i\in I}\tau_{d_{i}}\right]_{g_{1},i+1}E_{\mathbf{d}(k_{2},J),n,g_{1}}^{k}\\
 & \leqslant\sum_{\substack{g_{1},i\\
2\leqslant2g_{1}+i\leqslant k+1\\
i\leqslant n-1
}
}C^{g_{1}+\frac{i+1}{2}}(2g_{1}+i+1)!\sum_{\substack{I\subseteq\left\{ 2,\ldots n\right\} \\
|I|=i
}
}\sum_{\ell=0}^{\infty}\sum_{k_{2}=0}^{\ell+d_{1}-2}(a_{\ell}-a_{\ell-1})E_{\mathbf{d}(k_{2},J),n,g_{1}}^{k}\\
 & =\sum_{\substack{g_{1},i\\
2\leqslant2g_{1}+i\leqslant k+1\\
i\leqslant n-1
}
}\sum_{\substack{t_{1},\ldots,t_{n-i}\\
\sum_{j}t_{j}\leqslant2(k+1)\\
t_{j}\geqslant0
}
}C^{g_{1}+\frac{i+1}{2}}(2g_{1}+i+1)!\sum_{\substack{I\subseteq\left\{ 2,\ldots n\right\} \\
|I|=i
}
}\sum_{\ell=0}^{\infty}\sum_{k_{2}=0}^{\ell+d_{1}-2}(a_{\ell}-a_{\ell-1})\xi_{(t_{1},\ldots,t_{n-i})}^{t,I,k_{2},g_{1},n}d_{j_{1}}^{t_{1}}\cdots d_{j_{n-i-1}}^{t_{n-i-1}}k_{2}^{t_{n-i}}.
\end{align*}
where the final inequality follows from $\left[\tau_{k_{1}}\prod_{i\in I}\tau_{d_{i}}\right]_{g_{1},i+1}\leqslant V_{g_{1},i+1}$
and then Lemma \ref{lem:(Grushevsky's-bound,-)}. Then, by Lemma \ref{lem:basic estimate}
we have
\begin{align*}
\sum_{\ell=0}^{\infty}(a_{\ell}-a_{\ell-1})\sum_{k_{2}=0}^{\ell+d_{1}-2}k_{2}^{t_{n-i}} & \leqslant\sum_{\ell=0}^{\infty}(a_{\ell}-a_{\ell-1})(\ell+d_{1})^{t_{n-i}+1}\\
 & =\sum_{\ell=0}^{\infty}(a_{\ell}-a_{\ell-1})\sum_{q=0}^{t_{n-i}+1}{t_{n-i}+1 \choose q}\ell^{q}d_{1}^{t_{n-i}+1-q}\\
 & \leqslant2(t_{n-i}+1)!\sum_{q=0}^{t_{n-i}+1}\frac{d_{1}^{t_{n-i}+1-q}}{(t_{n-i}+1-q)!}.
\end{align*}
We further bound 
\begin{align*}
C^{g_{1}+\frac{i+1}{2}}(2g_{1}+i+1)! & \leqslant C^{g_{1}+\frac{i+1}{2}}(k+2)^{2g_{1}+i+1}\\
 & \leqslant(c(k+2))^{2g_{1}+i+1}.
\end{align*}
The terms that contribute to the coefficient of $d_{1}^{a_{1}}\cdots d_{n}^{a_{n}}$
are those $J$ which contain the set $J_{1}=\left\{ j\neq1:a_{j}\neq0\right\} $
hence the summation over $i$ is restricted to those $i$ for which
$n-i-1\geqslant|J_{1}|$ and given such an $i$ there are at most
${n \choose i}\leqslant n^{i}$ sets $I$ that don't contain any element
of $J_{1}$. For such a $J$ we set the $t_{j}$ for $j\in J\setminus J_{1}$
equal to 0 and for $j\in J_{1}$ we set $t_{j}=a_{j}$ and lastly
$t_{n-i}+1-q=a_{1}$ so that we only sum over $t_{n-i}\geqslant a_{1}-1$.
With these considerations and the bound on $\xi_{(t_{1},\ldots,t_{n-i})}^{t,I,k_{2},g_{1},n}$,
the coefficient of $d_{1}^{a_{1}}\cdots d_{n}^{a_{n}}$ is bounded
by (noting that for $j\not\notin J_{1}$ we have $a_{j}!=1)$
\begin{align*}
 & \frac{2(n+k+2)^{k+2}}{a_{1}!\cdots a_{n}!}\sum_{\substack{g_{1},i\\
3\leqslant2g_{1}+i\leqslant k+1\\
i\leqslant n-1
}
}\sum_{\substack{a_{1}-1\leqslant t_{n-i}\leqslant2(k+1)-\sum_{j=2}^{n}a_{j}\\
t_{n-i}\leqslant2(k+1)-(2g_{1}+i)\\
t_{j}\geqslant0
}
}(t_{n-i}+1)(c(k+2))^{c(k+2-(2g_{1}+i))+6(2g_{1}+i)+2}\\
 & \leqslant\frac{(n+k+2)^{k+2}}{a_{1}!\cdots a_{n}!}(c(k+2))^{c(k+2)-5},
\end{align*}
where in the last line, we have used that the number of terms in
the interior sum is bounded by $2(k+1)$, the number of terms in the
exterior sum is at most $(k+1)^{2}$ and the fact that $c>600$ and
$2g_{1}+i\geqslant2$ to absorb the constant $2$ and bound the exponent
by $-c(2g_{1}+i)+6(2g_{1}+i+1)+1<-6$. This proves the estimate for
$P_{k,n}^{(a)}(d_{1},\ldots,d_{n})$.

We now prove the claim about the coefficient of $g^{-t}$.

\textbf{Claim. }The coefficient of $g^{-t}$ in (\ref{eq:coeff-(a)})
is a polynomial in $d_{1},\ldots,d_{n}$ of degree at most $2(t-2)$
whose coefficient of $d_{1}^{a_{1}}\cdots d_{n}^{a_{n}}$ is bounded
by
\[
\frac{(ct)^{ct-5}(n+t)^{t}}{a_{1}!\cdots a_{n}!}.
\]
The coefficient of $g^{0}$ and $g^{-1}$ is zero.

\textbf{Proof of claim. }The coefficient of $g^{-t}$ is given by
\begin{equation}
\sum_{\substack{g_{1},i\\
2\leqslant2g_{1}+i\leqslant t\\
i\leqslant n-1
}
}\sum_{\substack{I\subseteq\left\{ 2,\ldots n\right\} \\
|I|=i
}
}\sum_{\ell=0}^{\infty}\sum_{k_{1}+k_{2}=\ell+d_{1}-2}(a_{\ell}-a_{\ell-1})\left[\tau_{k_{1}}\prod_{i\in I}\tau_{d_{i}}\right]_{g_{1},i+1}\Xi_{I,k_{2},g_{1},n}^{t},\label{eq:Cg-tcoeff}
\end{equation}
which is a polynomial in $d_{1},\ldots,d_{n}$, and the summation
of $g_{1},i$ is restricted to $2g_{1}+i\leqslant t$ because $\Xi_{I,k_{2},g_{1},n}^{t}\equiv0$
whenever $t>2g_{1}+i$ . Since $2\leqslant2g_{1}+i\leqslant t$, the
coefficients of $g^{0}$ and $g^{-1}$ are zero. To determine the
coefficient of $d_{1}^{a_{1}}\cdots d_{n}^{a_{n}}$, as with the error
term, we note that the only terms that have contribution are those
with $I$ such that $J=\left\{ 2,\ldots,n\right\} \setminus I$ contains
$J_{1}=\left\{ j\neq1:a_{j}\neq0\right\} $. The summation over $i$
is restricted to those $i$ for which $n-i-1\geqslant|J_{1}|$ and
given such an $i$ there are at most ${n \choose i}\leqslant n^{i}$
sets $I$ that don't contain any element of $J_{1}$. For such an
$I$, the $\Xi_{I,k_{2},g_{1},n}^{t}$ is a polynomial in $d_{j_{1}},\ldots,d_{j_{n-i-1}},k_{2}$
of degree at most $2(t-2)$ and with coefficient of $\left(\prod_{r=1}^{n-i-1}d_{j_{r}}^{a_{r}}\right)k_{2}^{a_{n-i}}$
bounded by 

\[
\frac{(n+t)^{t-(2g_{1}+i)}}{a_{1}!\cdots a_{n-i}!}\left(ct\right)^{ct-(c-4)(2g_{1}+i)}.
\]
Indeed, for $g_{1}>0$, the coefficient is by definition bounded by
the following summed over all choices of $\alpha_{1},\beta_{p,s},\gamma_{p,s},t_{j},\alpha_{2}$
\begin{align}
 & \frac{500^{2g_{1}+i}}{a_{1}!\cdots a_{n-i}!}\sum_{u_{1}=1}^{\alpha_{1}}\left(\prod_{(p,s)\in\mathcal{B}}\sum_{u_{p,s}=1}^{\beta_{p,s}}\right)\left(\prod_{(p,s)\in\mathcal{C}}\sum_{v_{p,s}=1}^{\gamma_{p,s}}\right){\alpha_{1}+\sum_{(p,s)\in\mathcal{B}}\beta_{p,s}+\sum_{(p,s)\in\mathcal{C}}\gamma_{p,s}-(1+|\mathcal{B}|+|\mathcal{C}|) \choose u_{1}+\sum_{(p,s)\in\mathcal{B}}u_{p,s}+\sum_{(p,s)\in\mathcal{C}}v_{p,s}-(1+|\mathcal{B}|+|\mathcal{C}|)}\nonumber \\
 & \cdot g_{1}^{\alpha_{1}+\sum_{(p,s)\in\mathcal{B}}\beta_{p,s}+\sum_{(p,s)\in\mathcal{C}}\gamma_{p,s}-(u_{1}+\sum_{(p,s)\in\mathcal{B}}u_{p,s}+\sum_{(p,s)\in\mathcal{C}}v_{p,s})}\label{eq:Ca-coeff}\\
 & \cdot(c\alpha_{1})^{cu_{1}}\prod_{(p,s)\in\mathcal{B}}(c\beta_{p,s})^{cu_{p,s}}\prod_{(p,s)\in\mathcal{C}}(c\gamma_{p,s})^{cv_{p,s}}\nonumber \\
 & \cdot\alpha_{2}^{2g_{1}+i}g_{1}^{\alpha_{2}-(2g_{1}+i)}\left(\prod_{j=0}^{2\rho_{m}+i-1}(ct_{j})^{ct_{j}}\right)(n+t)^{t-(2g_{1}+i)},\nonumber 
\end{align}
where we use the fact that ${n_{1} \choose m_{1}}{n_{2} \choose m_{2}}\leqslant{n_{1}+n_{2} \choose m_{1}+m_{2}}$,
the inductive bound on the coefficient of $\left(\prod_{r=1}^{n-i-1}d_{j_{r}}^{a_{r}}\right)k_{2}^{a_{n-i}}$
in $b_{\mathbf{d}(k_{2,}J),n-i}^{u_{1}}$ to obtain that the coefficient
of $\left(\prod_{r=1}^{n-i-1}d_{j_{r}}^{a_{r}}\right)k_{2}^{a_{n-i}}$
in $A_{\alpha_{1}}$ (which is a polynomial of degree at most $2\alpha_{1}\leqslant2(t-2)$
since $\alpha_{2}\geqslant2g_{1}+i\geqslant2\implies\alpha_{1}\leqslant t-2$)
is bounded by
\begin{align*}
\frac{(n+t)^{\alpha_{1}}}{\alpha_{1}!\cdots\alpha_{n-i}!}\sum_{u_{1}=1}^{\alpha_{1}}{\alpha_{1}-1 \choose u_{1}-1}g_{1}^{\alpha_{1}-u_{1}}(c\alpha_{1})^{cu_{1}},
\end{align*}
and the inductive bounds on the other coefficients to obtain (noting
that $\beta_{p,s},\gamma_{p,s},t_{j}\leqslant t-2$)
\begin{align*}
B_{p,s}^{\beta_{p,s}} & \leqslant500(n+t)^{\beta_{p,s}}\sum_{u_{p,s}=1}^{\beta_{p,s}}{\beta_{p,s}-1 \choose u_{p,s}-1}g_{1}^{\beta_{p,s}-u_{p,s}}(c\beta_{p,s})^{cu_{p,s}},\\
C_{p,s}^{\gamma_{p,s}} & \leqslant(n+t)^{\gamma_{p,s}}\sum_{v_{p,s}=1}^{\gamma_{p,s}}{\gamma_{p,s}-1 \choose v_{p,s}-1}g_{1}^{\gamma_{p,s}-v_{p,s}}(c\gamma_{p,s})^{cv_{p,s}},\\
h_{n-i-2\rho_{m}+j+1}^{t_{j}} & \leqslant500(n+t)^{t_{j}}(ct_{j})^{ct_{j}},\\
|\theta_{n}^{\alpha_{2}}| & \leqslant\alpha_{2}^{2g_{1}+i-1}g_{1}^{\alpha_{2}-(2g_{1}+i)}(n-i)^{\alpha_{2}-(2g_{1}+i)}.
\end{align*}
The remaining summations in (\ref{eq:Ca-coeff}) are bounded using
Lemma \ref{lem:coeff-prod-error} to obtain 
\begin{align*}
(\ref{eq:Ca-coeff}) & \leqslant\frac{500^{2g_{1}+i}(n+t)^{t-(2g_{1}+i)}}{a_{1}!\cdots a_{n-i}!}\alpha_{2}^{2g_{1}+i}g_{1}^{\alpha_{2}-(2g_{1}+i)}\left(\prod_{j=0}^{2\rho_{m}+i-1}(ct_{j})^{ct_{j}}\right)\cdot\\
 & \cdot\left(c\left(\alpha_{1}+\sum_{(p,s)\in\mathcal{B}}\beta_{p,s}+\sum_{(p,s)\in\mathcal{C}}\gamma_{p,s}\right)+g_{1}\right)^{c\left(\alpha_{1}+\sum_{(p,s)\in\mathcal{B}}\beta_{p,s}+\sum_{(p,s)\in\mathcal{C}}\gamma_{p,s}\right)}\\
 & \leqslant\frac{500^{2g_{1}+i}(n+t)^{t-(2g_{1}+i)}}{a_{1}!\cdots a_{n-i}!}\left(c\left(t-\alpha_{2}\right)+g_{1}\right)^{c\left(t-\alpha_{2}\right)}\alpha_{2}^{2g_{1}+i}g_{1}^{\alpha_{2}-(2g_{1}+i)}\\
 & \leqslant\frac{500^{2g_{1}+i}(n+t)^{t-(2g_{1}+i)}}{a_{1}!\cdots a_{n-i}!}\left(c\left(t-\alpha_{2}\right)+2g_{1}+\alpha_{2}\right)^{c\left(t-\alpha_{2}\right)+\alpha_{2}}\\
 & \leqslant\frac{(n+t)^{t-(2g_{1}+i)}}{a_{1}!\cdots a_{n-i}!}\left(ct\right)^{ct-(c-1)(2g_{1}+i)}.
\end{align*}
The conclusion is identical but easier when $g_{1}=0$ since then
there are less asymptotic expansion base changes required. Summing
over the choices of $\alpha_{1},\beta_{p,s},\gamma_{p,s},t_{j},\alpha_{2}$
introduces a factor bounded by $t^{3g_{1}+i+2}$ since there are precisely
$3g_{1}+i+2$ different coefficients each of whose maximal value is
$t$. But, $c>600$ and $3(2g_{1}+i)>3g_{1}+i+2$ as $2g_{1}+i\geqslant2$
so we obtain the bound 
\[
\frac{(n+t)^{t-(2g_{1}+i)}}{a_{1}!\cdots a_{n-i}!}\left(ct\right)^{ct-(c-4)(2g_{1}+i)}
\]
for the coefficient of $\left(\prod_{r=1}^{n-i-1}d_{j_{r}}^{a_{r}}\right)k_{2}^{a_{n-i}}$
in $\Xi_{I,k_{2},g_{1},n}^{t}$. 

To obtain the coefficient of $g^{-t}$ in $(a)$ we now insert this
bound on the coefficients into (\ref{eq:Cg-tcoeff}). Then identically
to the computation for the bound on the coefficients of the error
term $P_{k,n}^{(a)}(d_{1},\ldots,d_{n})$, we obtain the following
bound on the coefficient of $d_{1}^{a_{1}}\cdots d_{n}^{a_{n}}$ 
\begin{align*}
 & \frac{2(n+t)^{t}}{a_{1}!\cdots a_{n}!}\sum_{\substack{g_{1},i\\
2\leqslant2g_{1}+i\leqslant t\\
i\leqslant n-1
}
}\sum_{\substack{a_{1}-1\leq t_{n-i}\leqslant2t-\sum_{j=2}^{n}a_{j}\\
t_{n-i}\leqslant2t-(2g_{1}+i)\\
t_{j}\geqslant0
}
}(t_{n-i}+1)(t+1)^{2g_{1}+i+1}(ct)^{ct-(c-4)(2g_{1}+i)}\\
 & \leqslant\frac{2(n+t)^{t}}{a_{1}!\cdots a_{n}!}\sum_{\substack{g_{1},i\\
2\leqslant2g_{1}+i\leqslant t\\
i\leqslant n-1
}
}(ct)^{ct-(c-4)(2g_{1}+i)+2g_{1}+i+4}\leqslant\frac{(n+t)^{t}}{a_{1}!\cdots a_{n}!}(ct)^{ct-5},
\end{align*}
since $c>600$. This conclude the proof of the claim. 

We finally prove the claim about the coefficients of the polynomial
\textbf{$E_{\mathbf{d}(k_{2},J),n,g_{1}}^{k}$}.

\noindent\textbf{Claim. $E_{\mathbf{d}(k_{2},J),n,g_{1}}^{k}$ }is
a polynomial in $d_{j_{1}},\ldots,d_{j_{n-i-1}},k_{2}$ of degree
at most $2(k+1)$ whose coefficient of $\left(\prod_{p=1}^{n-i-1}d_{j_{p}}^{a_{p}}\right)k_{2}^{a_{n-i}}$
is bounded by 
\[
\frac{(n+k+2)^{k+2-(2g_{1}+i)}}{a_{1}!\cdots a_{n-i}!}\left(c(k+2)\right)^{c(k+2))-(c-4)(2g_{1}+i)}.
\]

\noindent\textbf{Proof of claim.} The proof of this claim is very
similar in flavor to the way that we obtained bounds for the individual
coefficients above using the induction hypotheses on coefficients
and error terms and using Lemma \ref{lem:coeff-prod-error} and so
we omit the details. 

We note by near identical computation and arguments, a similar bound
holds for the (b) term in (\ref{eq:C-term}). Putting everything together,
we obtain an asymptotic expansion for $S_{3}$ of the form 
\[
\left|S_{3}-\sum_{t=2}^{k+1}\frac{\sigma_{\mathbf{d},n}^{t}}{g^{t}}\right|\leqslant\frac{P_{k+1,n}^{(3)}(d_{1},\ldots,d_{n})}{g^{k+2}},
\]
where $P_{k+1,n}^{(3)}(d_{1},\ldots,d_{n})$ is a polynomial in $d_{1},\ldots,d_{n}$
of degree at most $2(k+1)$ whose coefficient of $d_{1}^{a_{1}}\cdots d_{n}^{a_{n}}$
is bounded by
\[
\frac{(n+k+2)^{k+2}}{a_{1}!\cdots a_{n}!}(c(k+2))^{c(k+2)-3},
\]
and the $\sigma_{\mathbf{d},n}^{t}$ are some polynomials in $d_{1},\ldots,d_{n}$
of degree at most $2(t-2)$ whose coefficient of $d_{1}^{a_{1}}\cdots d_{n}^{a_{n}}$
is bounded by 
\[
\frac{(n+t)^{t}}{a_{1}!\cdots a_{n}!}(ct)^{ct-3}.
\]
 
\end{proof}
\begin{lem}
\label{lem:A3A4}Suppose that for $k\geqslant0$, $A1(k)$ holds,
then $A3(k)$ and $A4(k)$ hold.
\end{lem}

\begin{proof}
We recall by Theorem \ref{thm-Recurrsions} statement \ref{eq:recurssioniii},
that

\[
\frac{4\pi^{2}\left(2g-2+n\right)V_{g,n}}{V_{g,n+1}}-1=2\sum_{\ell=1}^{3g+n-2}\dfrac{(-1)^{\ell-1}\ell\pi^{2\ell}}{(2\ell+1)!}\Bigg(\dfrac{[\tau_{\ell}\tau_{0}^{n}]_{g,n+1}}{V_{g,n+1}}-1\Bigg)+2\sum_{\ell=3g+n-1}^{\infty}\dfrac{(-1)^{\ell-1}\ell\pi^{2\ell}}{(2\ell+1)!},
\]
since 
\[
2\sum_{\ell=1}^{\infty}\dfrac{(-1)^{\ell-1}\ell\pi^{2\ell}}{(2\ell+1)!}=1.
\]
By $A1(k)$, there exist functions $b_{(\ell,0,\ldots,0),n}^{i}$
and a polynomial $Q_{n+1}^{k}$ such that 

\[
\Bigg|\dfrac{[\tau_{\ell}\tau_{0}^{n}]_{g,n+1}}{V_{g,n+1}}-1-\sum_{i=1}^{k}\frac{b_{(\ell,0,\ldots,0),n+1}^{i}}{g^{i}}\Bigg|\leqslant\frac{Q_{n+1}^{k}(\ell,0,\ldots,0)}{g^{k+1}}
\]
where $Q_{n+1}^{k}(\ell,0,\ldots,0)$ is a polynomial in $\ell$ of
degree at most $2(k+1)$ such that the coefficient of $\ell^{t_{1}}$
is 
\[
\frac{(c(k+1))^{c(k+1)}\left(n+k+2\right)^{k+1}}{t_{1}!}.
\]
Then defining
\[
h_{n}^{i}\eqdf2\sum_{\ell=1}^{\infty}\dfrac{(-1)^{\ell-1}\ell\pi^{2\ell}}{(2\ell+1)!}b_{(\ell,0,\ldots,0),n+1}^{i},
\]
we have that
\begin{align*}
\left|\frac{4\pi^{2}\left(2g-2+n\right)V_{g,n}}{V_{g,n+1}}-1-\sum_{i=1}^{k}\frac{h_{n}^{i}}{g^{i}}\right|\leqslant & \left|\frac{2}{g^{k+1}}\sum_{\ell=1}^{\infty}\dfrac{(-1)^{\ell-1}\ell\pi^{2\ell}}{(2\ell+1)!}Q_{n+1}^{k}(\ell,0,\ldots,0)\right|\\
 & +\left|2\sum_{i=1}^{k}\frac{1}{g^{k}}\sum_{\ell=3g+n-1}^{\infty}\dfrac{(-1)^{\ell-1}\ell\pi^{2\ell}}{(2\ell+1)!}b_{(\ell,0,\ldots,0),n+1}^{i}\right|\\
 & +2\sum_{\ell=3g+n-1}^{\infty}\dfrac{(-1)^{\ell-1}\ell\pi^{2\ell}}{(2\ell+1)!}.
\end{align*}
Since for any $\ell\geqslant1$, 
\[
\frac{\ell^{3}\pi^{2\ell}}{\left(2\ell+1\right)!}\leqslant\frac{110}{e^{\ell}},
\]
we can use Lemma \ref{lem:basic estimate} to obtain

\begin{align*}
\left|2\sum_{\ell=1}^{\infty}\dfrac{(-1)^{\ell-1}\ell\pi^{2\ell}}{(2\ell+1)!}Q_{n+1}^{k}(\ell,0,\ldots,0)\right| & \leqslant\sum_{t_{1}=0}^{2(k+1)}\frac{(c(k+1))^{c(k+1)}\left(n+k+2\right)^{k+1}}{t_{1}!}\sum_{\ell=1}^{\infty}\dfrac{2\ell^{t_{1}+1}\pi^{2\ell}}{(2\ell+1)!}\\
 & \leqslant(c(k+1))^{c(k+1)}\left(n+k+2\right)^{k+1}\left(23+\sum_{t_{1}=2}^{2(k+1)}\frac{1}{t_{1}!}\sum_{\ell=1}^{\infty}\dfrac{2\ell^{t_{1}+1}\pi^{2\ell}}{(2\ell+1)!}\right)\\
 & \leqslant(c(k+1))^{c(k+1)}\left(n+k+2\right)^{k+1}\left(23+220\sum_{t_{1}=2}^{2(k+1)}\frac{1}{t_{1}!}\sum_{\ell=1}^{\infty}\dfrac{\ell^{t_{1}-2}}{e^{\ell}}\right)\\
 & \leqslant(c(k+1))^{c(k+1)}\left(n+k+2\right)^{k+1}\left(23+440\sum_{t_{1}=2}^{2(k+1)}\frac{1}{t_{1}(t_{1}-1)}\right)\\
 & \leqslant463(c(k+1))^{c(k+1)}\left(n+k+2\right)^{k+1}.
\end{align*}
By an identical argument, the coefficients satisfy the bound 
\[
h_{n}^{i}\leqslant463\left(ci\right)^{ci}\left(n+i+1\right)^{i}.
\]
Moreover, using Lemma \ref{lem:poly-rels} gives
\begin{align*}
\left|2\sum_{i=1}^{k}\frac{1}{g^{k}}\sum_{\ell=3g+n-1}^{\infty}\dfrac{(-1)^{\ell-1}\ell\pi^{2\ell}}{(2\ell+1)!}b_{(\ell,0,\ldots,0),n+1}^{i}\right| & \leqslant\sum_{t_{1}=0}^{2(k+1)}\frac{(c(k+1))^{c(k+1)}\left(n+k+2\right)^{k+1}}{t_{1}!}\sum_{\ell=3g+n-1}^{\infty}\dfrac{2\ell^{t_{1}+1}\pi^{2\ell}}{(2\ell+1)!}.
\end{align*}
Since $t_{1}\leqslant2(k+1)<g<\ell$, we have
\begin{align*}
\dfrac{\ell^{t_{1}+1}\pi^{2\ell}}{(2\ell+1)!} & =\frac{\ell^{t_{1}+1}}{\prod_{q=0}^{t_{1}}(2\ell+1-q)}\frac{\pi^{2\ell}}{(2\ell-t_{1})!}\\
 & \leqslant\frac{\pi^{2\ell}}{\ell!}.
\end{align*}
Thus, 
\begin{align*}
\sum_{\ell=3g+n-1}^{\infty}\dfrac{2\ell^{t_{1}+1}\pi^{2\ell}}{(2\ell+1)!} & \leqslant2\sum_{\ell=3g+n-1}^{\infty}\frac{(\pi^{2})^{\ell}}{\ell!}\\
 & \leqslant2e^{\pi^{2}}\frac{(e\pi^{2})^{3g+n-1}}{(3g+n-1)^{3g+n-1}}\\
 & \leqslant2e^{\pi^{2}}g^{-2g},
\end{align*}
where the last line holds because

\[
\frac{(e\pi^{2})^{3g+n-1}g^{2g}}{(g+1)^{3g+n-1}}\leqslant\left(\frac{\left(e\pi^{2}\right)^{3}}{g}\right)^{g}<1,
\]
because $g>\left(e\pi^{2}\right)^{3}$. It follows that 
\begin{align*}
\left|2\sum_{i=1}^{k}\frac{1}{g^{k}}\sum_{\ell=3g+n-1}^{\infty}\dfrac{(-1)^{\ell-1}\ell\pi^{2\ell}}{(2\ell+1)!}b_{(\ell,0,\ldots,0),n+1}^{i}\right| & \leqslant\frac{2e^{\pi^{2}+1}(c(k+1))^{c(k+1)}\left(n+k+2\right)^{k+1}}{g^{2g}}\\
 & =2e^{\pi^{2}+1}\left(\frac{(c(k+1))^{c}(n+k+2)}{g}\right)^{k+1}\frac{1}{g^{2g-k-1}}\\
 & \leqslant\frac{1}{g^{k+2}},
\end{align*}
since $(c(k+1))^{c}(n+k+2)<g$. An easier and similar bound holds
for the series tail. Combining the error terms then gives the conclusion
for $A3(k)$.

The proof of $A4(k)$ uses the identity (\ref{eq:recurs-1A})

\[
\frac{V_{g-1,n+2}}{V_{g,n}}=\frac{\left[\tau_{1}\tau_{0}^{n-1}\right]_{g,n}}{V_{g,n}}-6\sum_{\substack{I\sqcup J=\left\{ 1,\dots,n-2\right\} \\
g_{1}+g_{2}=g
}
}\frac{V_{g_{1},|I|+2}V_{g_{2},|J|+2}}{V_{g,n}}.
\]
The expansion of the intersection number term follows immediately
from $A1(k)$ and the expansion of the second term on the right-hand
side follows from an easier argument used in the expansion of $S_{3}$
in the proof of Lemma \ref{lem:S_3} since it is a slightly modified
version of (\ref{eq:C-term}) when $k_{1}=k_{2}=d_{1}=\ldots=d_{n}=0$
and $\ell=2$.
\end{proof}
\begin{prop}
Suppose that for some $k\geqslant0$, $A2(k)$ holds, then $A1(k)$
also holds.
\end{prop}

\begin{proof}
For $A1(k)$, observe first that
\begin{align*}
 & 1-\frac{\left[\tau_{d_{1}}\cdots\tau_{d_{n}}\right]_{g,n}}{V_{g,n}}=\frac{[\tau_{0}^{n}]_{g,n}-\left[\tau_{d_{1}}\cdots\tau_{d_{n}}\right]_{g,n}}{V_{g,n}}\\
 & =\frac{\left(\sum_{i=0}^{d_{1}-1}[\tau_{i}\tau_{0}^{n-1}]_{g,n}-[\tau_{i+1}\tau_{0}^{n-1}]_{g,n}\right)+\dots+\left(\sum_{i=0}^{d_{n}-1}[\tau_{d_{1}}\cdots\tau_{d_{n-1}}\tau_{i}]_{g,n}-[\tau_{d_{1}}\cdots\tau_{d_{n-1}}\tau_{i+1}]_{g,n}\right)}{V_{g,n}}.
\end{align*}
Now, for $e_{\mathbf{d},n}^{t}$ and $P_{n}^{k}(\mathbf{d})$ the
coefficients and error polynomial from $A2(k)$, if we denote by $p_{(t_{1},\ldots,t_{n})}^{k,n}$
the coefficient of $d_{1}^{t_{1}}\cdots d_{n}^{t_{n}}$ in $P_{n}^{k}(\mathbf{d})$,
then we have for each $j=1,\ldots,n$
\begin{align*}
 & \left|\sum_{i=0}^{d_{j}-1}\left(\frac{\left[\tau_{d_{1}}\cdots\tau_{d_{j-1}}\tau_{i}\tau_{0}^{n-j}\right]_{g,n}-\left[\tau_{d_{1}}\cdots\tau_{d_{j-1}}\tau_{i+1}\tau_{0}^{n-j}\right]_{g,n}}{V_{g,n}}-\sum_{t=1}^{k}\frac{e_{(d_{1},\ldots,d_{j},i,0,\ldots,0)}^{t}}{g^{t}}\right)\right|\\
 & \leqslant\sum_{i=0}^{d_{j}-1}\left|\frac{\left[\tau_{d_{1}}\cdots\tau_{d_{j-1}}\tau_{i}\tau_{0}^{n-j}\right]_{g,n}-\left[\tau_{d_{1}}\cdots\tau_{d_{j-1}}\tau_{i+1}\tau_{0}^{n-j}\right]_{g,n}}{V_{g,n}}-\sum_{t=1}^{k}\frac{e_{(d_{1},\ldots,d_{j},i,0,\ldots,0)}^{t}}{g^{t}}\right|\\
 & \leqslant\frac{1}{g^{k+1}}\sum_{i=0}^{d_{j}-1}\sum_{p=0}^{ck+1}\sum_{\substack{0\leq t_{1},\dots,t_{j}\\
\sum t_{j}=p
}
}p_{(t_{1},\dots,t_{j},0,...,0)}^{k,n}d_{1}^{t_{1}}\cdots d_{j-1}^{t_{j-1}}i^{t_{j}}\\
 & =\frac{1}{g^{k+1}}\sum_{p=0}^{ck+1}\sum_{\substack{0\leqslant t_{1},\dots,t_{j}\\
\sum t_{j}=p
}
}p_{(t_{1},\dots,t_{j},0,...,0)}^{k,n}d_{1}^{t_{1}}\cdots d_{j-1}^{t_{j-1}}d_{j}^{t_{j}}\sum_{i=1}^{d_{j}}\left(1-\frac{i}{d_{j}}\right)^{t_{j}}.
\end{align*}
For any non-negative integer $s$ we have the bound
\[
\sum_{i=1}^{d_{j}}\left(1-\frac{i}{d_{j}}\right)^{s}\leqslant\frac{d_{j}}{s+1},
\]
which is trivial when $s=0$ and for $s\geqslant1$, one can bound
by the integral
\[
\int_{0}^{d_{j}}\left(1-\frac{x}{d_{j}}\right)^{s}\mathrm{d}x=\frac{d_{j}}{s+1}.
\]
Thus 
\begin{align*}
 & \left|\sum_{i=0}^{d_{j}-1}\left(\frac{\left[\tau_{d_{1}}\cdots\tau_{d_{j-1}}\tau_{i}\tau_{0}^{n-j}\right]_{g,n}-\left[\tau_{d_{1}}\cdots\tau_{d_{j-1}}\tau_{i+1}\tau_{0}^{n-j}\right]_{g,n}}{V_{g,n}}-\sum_{t=1}^{k}\frac{e_{(d_{1},\ldots,d_{j},i,0,\ldots,0),n}^{t}}{g^{t}}\right)\right|\\
 & \leqslant\frac{1}{g^{k+1}}\underbrace{\sum_{p=0}^{ck+1}\sum_{\substack{0\leqslant t_{1},\dots,t_{j}\\
\sum t_{j}=p
}
}p_{(t_{1},\dots,t_{j},0,...,0)}^{k,n}d_{1}^{t_{1}}\cdots d_{j-1}^{t_{j-1}}\frac{d_{j}^{t_{j}+1}}{t_{j}+1}}_{\eqdf\tilde{Q}_{n}^{k,j}(d_{1},\ldots,d_{j})}.
\end{align*}
Now $\text{\ensuremath{\tilde{Q}_{n}^{k,j}}}$ is a polynomial in
$d_{1},\dots,d_{j}$ of degree at most $2k+2=2(k+1)$ and so we define
\[
\tilde{Q}_{n}^{k}\left(d_{1},\dots,d_{n}\right)\eqdf\sum_{j=1}^{n}\tilde{Q}_{n}^{k,j}\left(d_{1},\dots,d_{j}\right),
\]
which is also a polynomial of degree at most $2(k+1)$ in $d_{1},\ldots,d_{n}$.
We wish to determine a bound on the coefficient of a given monomial
$d_{1}^{a_{1}}\cdots d_{n}^{a_{n}}$ in $Q_{n}^{k}(d_{1},\ldots,d_{n})$.
Note that the contribution to this coefficient only comes from $\tilde{Q}_{n}^{k,m}$
where $m\leqslant n$ is the largest index for which $a_{m}$ is non-zero.
Indeed, if $j<m$ then $\text{\ensuremath{\tilde{Q}_{n}^{k,j}}}$
is a polynomial in $d_{1},\ldots,d_{j}$ and in particular the power
of $d_{m}$ is always zero in any of its monomials so it cannot give
contribution to the coefficient since $a_{m}>0$. Moreover, if $j>m$
then $\text{\ensuremath{\tilde{Q}_{n}^{k,j}}}$ is a polynomial in
$d_{1},\ldots,d_{j}$ such that every monomial has $d_{j}$ occurring
with exponent at least 1. But $a_{m}$ is by definition the last non-zero
exponent and so $d_{1}^{a_{1}}\cdots d_{n}^{a_{n}}$ does not feature
as a monomial in $\tilde{Q}_{n}^{k,j}$.

The coefficient of $d_{1}^{a_{1}}\cdots d_{n}^{a_{n}}$ hence satisfies
\[
\frac{p_{(a_{1},\dots,a_{m-1},a_{m}-1,0,...,0)}^{k,n}}{a_{m}}\leqslant\frac{(c(k+1))^{c(k+1)}(n+k+1)^{c(k+1)}}{a_{1}!\cdots a_{m}!}.
\]
We hence obtain an expansion for $A1(k)$ since the above shows that
\[
\left|1-\frac{\left[\tau_{d_{1}}\cdots\tau_{d_{n}}\right]_{g,n}}{V_{g,n}}-\sum_{t=1}^{k}\frac{1}{g^{t}}\sum_{j=1}^{n}\sum_{i=0}^{d_{j}-1}e_{(d_{1},\ldots,d_{j},i,0,\ldots,0),n}^{t}\right|\leqslant\frac{\tilde{Q}_{n}^{k}(d_{1},\ldots,d_{n})}{g^{k+1}},
\]
so that for $t=1,\ldots,k$ we set
\[
b_{\mathbf{d},n}^{t}\eqdf-\sum_{j=1}^{n}\sum_{i=0}^{d_{j}-1}e_{(d_{1},\ldots,d_{j},i,0,\ldots,0),n}^{t}.
\]
An identical argument to the error term but replacing $k+1$ by $t$
shows that the $b_{\mathbf{d},n}^{t}$ can be majorised by polynomials
of degree at most $2(t-1)+2=2t$ in $d_{1},\ldots,d_{n}$ whose coefficient
of $d_{1}^{a_{1}}\cdots d_{n}^{a_{n}}$ is bounded by
\[
\frac{(ct)^{ct}(n+t)^{ct}}{a_{1}!\cdots a_{m}!},
\]
and so $A1(k)$ holds because the polynomials $q_{n}^{t}$ and $Q_{n}^{k}$
majorize polynomials with the above properties by definition.
\end{proof}
\begin{prop}
The base cases of the induction, $A1(0),A2(0),A3(0),A4(0)$ hold true.
\end{prop}

\begin{proof}
$A1(0)$ and $A2(0)$ follow from identical computation in \cite[Theorem A.1]{Li.Wr2024}
except for modifying the bound on the sum of the products of the volumes
in the proof of Lemma A.9 therein with a simple argument using the
bounds in \cite[Lemma 5.1]{Mi.Zo2015} to allow for extension of the
estimates from $n=o(\sqrt{g})$ to $n<cg$. $A3(0)$ and $A4(0)$
follow from $A1(0)$ via the exact same method as in the proof of
Lemma \ref{lem:A3A4}.
\end{proof}

\subsection{Proof of Corollary \ref{thm:main-estimate-WP-expansion}}

We now prove Corollary \ref{thm:main-estimate-WP-expansion} and note
that the constants $c,C$ appearing can\textbf{ }once again change
from line to line as we are only interested in their existence.
\begin{lem}
\label{lem:bordered}There is a $c>0$ such that for any $n$ and
$k$ there exists continuous functions $\left\{ \alpha_{n,j}\left(\mathbf{x}\right)\right\} _{j=1}^{k}$
with 
\[
\left|\frac{V_{g,n}\left(\mathbf{x}\right)}{V_{g,n}}-\prod_{i=1}^{n}\frac{\sinh\left(\frac{x_{i}}{2}\right)}{\left(\frac{x_{i}}{2}\right)}-\sum_{j=1}^{k}\frac{1}{g^{j}}\alpha_{n,j}\left(\mathbf{x}\right)\right|\leqslant\frac{\left(cnk\right)^{ck}\left(1+\left|\mathbf{x}\right|\right)^{ck}\exp\left(\frac{1}{2}\left|\mathbf{x}\right|\right)}{g^{k+1}},
\]
for $g>c(n+k)^{c}$. Furthermore each $\alpha_{n,j}\left(\mathbf{x}\right)$
satisfies 
\begin{equation}
\alpha_{n,j}\left(\mathbf{x}\right)\leqslant\left(cnj\right)^{cj}\left(1+\left|\mathbf{x}\right|\right)^{cj}\exp\left(\frac{1}{2}\left|\mathbf{x}\right|\right).\label{eq:expansion-function-est}
\end{equation}
\end{lem}

\begin{proof}
We recall Theorem \ref{thm:Mirz-vol-exp} says that
\[
V_{g,n}\left(\mathbf{x}\right)=\sum_{\substack{d_{1},\dots,d_{n}\\
|\mathbf{d}|\leqslant3g+n-3
}
}\frac{1}{2^{2\left|\d\right|}}\left[\prod_{i=1}^{n}\tau_{d_{i}}\right]_{g,n}\frac{x_{1}^{2d_{1}}}{\left(2d_{1}+1\right)!}\ensuremath{\cdots}\frac{x_{n}^{2d_{n}}}{\left(2d_{n}+1\right)!}.
\]
We set 
\begin{equation}
\alpha_{n,j}\left(\mathbf{x}\right)\eqdf\sum_{d_{1},\dots,d_{n}=0}^{\infty}\frac{b_{\mathbf{d},n}^{j}}{2^{2\left|\d\right|}}\frac{x_{1}^{2d_{1}}}{\left(2d_{1}+1\right)!}\cdots\frac{x_{n}^{2d_{n}}}{\left(2d_{n}+1\right)!}\label{eq:coeficients-vol-boundary}
\end{equation}
so that by Theorem \ref{thm:expansions}, (since $b_{\mathbf{d},n}^{0}=1$)
\begin{align*}
 & \left|\frac{V_{g,n}\left(\mathbf{x}\right)}{V_{g,n}}-\sum_{j=0}^{k}\frac{1}{g^{j}}\sum_{d_{1},\dots,d_{n}=0}^{\infty}\frac{b_{\mathbf{d},n}^{j}}{2^{2\left|\d\right|}}\frac{x_{1}^{2d_{1}}}{\left(2d_{1}+1\right)!}\cdots\frac{x_{n}^{2d_{n}}}{\left(2d_{n}+1\right)!}\right|\\
= & \left|\frac{V_{g,n}\left(\mathbf{x}\right)}{V_{g,n}}-\prod_{i=1}^{n}\frac{\sinh\left(\frac{x_{i}}{2}\right)}{\left(\frac{x_{i}}{2}\right)}-\sum_{j=1}^{k}\frac{1}{g^{j}}\sum_{d_{1},\dots,d_{n}=0}^{\infty}\frac{b_{\mathbf{d},n}^{j}}{2^{2\left|\d\right|}}\frac{x_{1}^{2d_{1}}}{\left(2d_{1}+1\right)!}\cdots\frac{x_{n}^{2d_{n}}}{\left(2d_{n}+1\right)!}\right|\\
\leqslant & \sum_{j=1}^{k}\frac{1}{g^{j}}\sum_{\substack{d_{1},\dots,d_{n}\\
\left|\d\right|\geqslant3g+n-2
}
}\frac{|b_{\mathbf{d},n}^{j}|}{2^{2\left|\d\right|}}\frac{x_{1}^{2d_{1}}}{\left(2d_{1}+1\right)!}\cdots\frac{x_{n}^{2d_{n}}}{\left(2d_{n}+1\right)!}\\
 & \,\,\,\,\,\,\,\,\,\,\,\,\,\,\,\,\,\,\,\,\,\,\,+\frac{1}{g^{k+1}}\sum_{\substack{d_{1},\dots,d_{n}\\
|\mathbf{d}|\leqslant3g+n-3
}
}Q_{n}^{k}(\d)\frac{1}{2^{2\left|\d\right|}}\frac{x_{1}^{2d_{1}}}{\left(2d_{1}+1\right)!}\cdots\frac{x_{n}^{2d_{n}}}{\left(2d_{n}+1\right)!}.
\end{align*}
By majorising the $b_{\d,n}^{j}$ by the polynomials $q_{n}^{j}(\d)$
and then applying Lemma \ref{lem:poly-rels} to bound $q_{n}^{j}(\d)g^{-j}$
by $g^{-k-1}Q_{n}^{k}(\d)$ in the first term on the right-hand side
(since $|\d|>3g+n-2$) we can bound the right-hand side by 
\begin{align*}
\frac{k}{g^{k+1}}\sum_{\substack{d_{1},\dots,d_{n}}
}Q_{n}^{k}(\d)\prod_{i=1}^{n}\frac{\left(\frac{x_{i}}{2}\right)^{2d_{i}}}{\left(2d_{i}+1\right)!} & \leqslant\frac{k(c(k+1))^{c(k+1)}(n+k+1)^{k+1}}{g^{k+1}}\sum_{\substack{t_{1,},\ldots,t_{n}\\
\sum_{i}t_{i}\leqslant2(k+1)
}
}\prod_{i=1}^{n}\sum_{d=0}^{\infty}\frac{d^{t_{i}}\left(\frac{x_{i}}{2}\right)^{2d}}{t_{i}!\left(2d\right)!}\\
 & \hspace*{-1cm}=\frac{k(c(k+1))^{c(k+1)}(n+k+1)^{k+1}}{g^{k+1}}\sum_{\substack{t_{1,},\ldots,t_{n}\\
\sum_{i}t_{i}\leqslant2(k+1)
}
}\prod_{i=1}^{n}\frac{1}{t_{i}!}\left(\frac{x_{i}}{2}\frac{\mathrm{d}}{\mathrm{d}x_{i}}\right)^{t_{i}}\cosh\left(\frac{x_{i}}{2}\right).
\end{align*}
Expanding the differential operator and noting that its coefficients
satisfy the recursion relation of the Stirling numbers of the second
kind ${t_{i} \brace p}$, we obtain
\[
\left(\frac{x_{i}}{2}\frac{\mathrm{d}}{\mathrm{d}x_{i}}\right)^{t_{i}}=\frac{1}{2^{t_{i}}}\sum_{p=1}^{t_{i}}{t_{i} \brace p}x_{i}^{p}\frac{\mathrm{d}^{p}}{\mathrm{d}x_{i}^{p}}.
\]
Then since $\frac{\mathrm{d}^{p}}{\mathrm{d}x_{i}^{p}}\cosh\left(\frac{x_{i}}{2}\right)\leqslant\frac{1}{2^{p}}e^{\frac{x_{i}}{2}}$
and ${t_{i} \brace p}\leqslant{t_{i} \choose p}p^{t_{i}-p}$ we obtain
\[
\left(\frac{x_{i}}{2}\frac{\mathrm{d}}{\mathrm{d}x_{i}}\right)^{t_{i}}\cosh\left(\frac{x_{i}}{2}\right)\leqslant\frac{e^{\frac{x_{i}}{2}}}{2^{t_{i}}}\sum_{p=1}^{t_{i}}{t_{i} \choose p}x_{i}^{p}t_{i}^{t_{i}-p}\leqslant\frac{e^{\frac{x_{i}}{2}}}{2^{t_{i}}}(x_{i}+t_{i})^{t_{i}}\leqslant\frac{e^{\frac{x_{i}}{2}}t_{i}!}{2^{t_{i}}}(x_{i}+1)^{t_{i}}.
\]
It hence follows that
\begin{align*}
\frac{k}{g^{k+1}}\sum_{\substack{d_{1},\dots,d_{n}}
}Q_{n}^{k}(\d)\prod_{i=1}^{n}\frac{\left(\frac{x_{i}}{2}\right)^{2d_{i}}}{\left(2d_{i}+1\right)!} & \leqslant e^{\frac{|\mathbf{x}|}{2}}(|\mathbf{x}|+1)^{c(k+1)}\sum_{\substack{t_{1,},\ldots,t_{n}\\
\sum_{i}t_{i}\leqslant2(k+1)
}
}\frac{1}{2^{\sum_{i}t_{i}}}\\
 & =e^{\frac{|\mathbf{x}|}{2}}(|\mathbf{x}|+1)^{c(k+1)}\sum_{p=0}^{2(k+1)}{p+n-1 \choose n-1}\frac{1}{2^{p}}\\
 & \leqslant e^{\frac{|\mathbf{x}|}{2}}(|\mathbf{x}|+1)^{c(k+1)}(n+c(k+1))^{c(k+1)}\sqrt{e},
\end{align*}
and the result follows immediately since the bound on the coefficients
is dealt with identically to the error term.
\end{proof}
We now prove Corollary \ref{thm:main-estimate-WP-expansion}.
\begin{proof}[Proof of Corollary \ref{thm:main-estimate-WP-expansion}]
 For the first part, we have $a,b$ and $q\geqslant1$ with $2g>2a+b=2g-q$
and $b\leqslant3q$. Then,
\begin{align*}
\frac{V_{a,b}\left(\mathbf{x}\right)}{V_{g}} & =\frac{V_{a,b}\left(\mathbf{x}\right)}{V_{a,b}}\cdot\frac{V_{a,b}}{V_{g}}\\
 & =\frac{V_{a,b}\left(\mathbf{x}\right)}{V_{a,b}}\cdot\prod_{j=0}^{q-1}\frac{4\pi^{2}\left(2a+b+j-2\right)V_{a,b+j}}{V_{a,b+j+1}}\prod_{j=0}^{g-a-1}\frac{V_{a+j,2g-2a-2j}}{V_{a+j+1,2g-2a-2j-2}}\prod_{j=0}^{q-1}\frac{1}{4\pi^{2}\left(2a+b+j-2\right)}.
\end{align*}
Now let $a=g-p$ so that because $b\leqslant3q,$ we have $p\leqslant2q$.
By Lemma \ref{lem:bordered}, there is a $c>0$ such that for any
$k$ there exist continuous functions $\left\{ \alpha_{b,j}\left(\mathbf{x}\right)\right\} _{j=1}^{k}$
with 
\begin{equation}
\left|\frac{V_{a,b}\left(\mathbf{x}\right)}{V_{a,b}}-\prod_{i=1}^{b}\frac{\sinh\left(\frac{x_{i}}{2}\right)}{\left(\frac{x_{i}}{2}\right)}-\sum_{j=1}^{k}\frac{1}{(g-p)^{j}}\alpha_{b,j}\left(\mathbf{x}\right)\right|\leqslant\frac{\left(cbk\right)^{ck}\left(1+\left|\mathbf{x}\right|\right)^{ck}\exp\left(\frac{1}{2}\left|\mathbf{x}\right|\right)}{(g-p)^{k+1}},\label{eq:bordered-expansion}
\end{equation}
whenever $g>c(p+k)^{c}$. By $A3(k)$ in Theorem \ref{thm:expansions}
we obtain an expansion for $\frac{4\pi^{2}\left(2a+b+j-2\right)V_{a,b+j}}{V_{a,b+j+1}}$
up to order $k$ with a base $g-p$ holding for $g>c(p+k)^{c}$ (the
$c$ is taken larger than that obtained in the induction statement)
and by $A4(k)$ in Theorem \ref{thm:expansions}, we obtain an expansion
for $\frac{V_{a+j,2p-2j}}{V_{a+j+1,2p-2j-2}}$ up to order $k$ with
a base $g-p+j+1$ holding for $g>c(p+k)^{c}$. Moreover, we obtain
an asymptotic expansion of $\prod_{j=0}^{2p-b-1}\frac{1}{4\pi^{2}\left(2a+b+j-2\right)}$
via a Taylor expansion. Using Lemma \ref{lem:shift} to shift these
expansions to have base $g$ and then Corollary \ref{cor:product-with-zeroterms}
to obtain an expansion of their product, we obtain the result in a
very similar way as in the proofs of Lemmas \ref{lem:S_1}, \ref{lem:S_2}
and \ref{lem:S_3} after noting that by assumption, $q\leqslant k$.
Note that due to the $\prod_{j=0}^{q-1}\frac{1}{4\pi^{2}\left(2a+b+j-2\right)}$
term, the asymptotic expansion starts at order $g^{-q}$ which is
at most order $g^{-1}$ due to $q\geqslant1$.

For the second part, since $2a+b=2g+n$ and $g-q\leqslant a\leqslant g$,
we have 
\[
\frac{V_{a,b}\left(\mathbf{x}\right)}{V_{g,n}}=\frac{V_{a,b}\left(\mathbf{x}\right)}{V_{a,b}}\cdot\prod_{i=0}^{g-a-1}\frac{V_{a+i,b-2i}}{V_{a+i+1,b-2i-2}}.
\]
By Lemma \ref{lem:bordered}, there is a $c>0$ such that for any
$k$ there exist continuous functions $\left\{ \alpha_{b,j}\left(\mathbf{x}\right)\right\} _{j=1}^{k}$
for which the expansion (\ref{eq:bordered-expansion}) holds whenever
$g>c(b+q+k)^{c}$. And from $A4(k)$ in in Theorem \ref{thm:expansions},
we obtain an expansion for $\frac{V_{a+i,b-2i}}{V_{a+i+1,b-2i-2}}$
up to order $k$ with a base $a+i+1$ holding for $g>c(n+q+k)^{c}$.
The desired asymptotic expansion then holds after application of Lemma
\ref{lem:shift} and Lemma (\ref{lem:product-exp}). The leading order
$g^{0}$ term is the product of the $g^{0}$ expansions of $\frac{V_{a,b}\left(\mathbf{x}\right)}{V_{a,b}}$
and $\frac{V_{a+i,b-2i}}{V_{a+i+1,b-2i-2}}$ which are $\prod_{i=1}^{b}\frac{\sinh\left(\frac{x_{i}}{2}\right)}{\left(\frac{x_{i}}{2}\right)}$
and $1$ respectively which yields the stated leading order asymptotic.
\end{proof}

\bibliographystyle{abbrv}
\bibliography{unitarybundlesbib}

\noindent Will Hide, \\
Mathematical Institute,\\
University of Oxford, \\
Andrew Wiles Building, OX2 6GG Oxford,\\
United Kingdom\\
\texttt{william}\texttt{.hide@}\texttt{maths}\texttt{.ox.ac.}\texttt{uk}\texttt{}~\\
\texttt{}~\\
Davide Macera, \\
Institute for Applied Mathematics\\
Faculty of Mathematics and Natural Sciences\\
Endenicher Allee 60\\
53115 Bonn\\
\texttt{macera}\texttt{@}\texttt{iam}\texttt{.}\texttt{uni-bonn}\texttt{.}\texttt{de}\texttt{}~\\
\texttt{}~\\
Joe Thomas, \\
Department of Mathematical Sciences,\\
Durham University, \\
Lower Mountjoy, DH1 3LE Durham,\\
United Kingdom\\
\texttt{joe}\texttt{.}\texttt{thomas}\texttt{@}\texttt{durham}\texttt{.ac.}\texttt{uk}
\end{document}